\documentclass[a4paper,11pt]{amsbook}
\usepackage[french,english]{babel}
\usepackage[babel=true,kerning=true]{microtype}
\usepackage[utf8]{inputenc} 
\usepackage{graphicx}
\usepackage{amssymb}
\usepackage{amsmath}
\usepackage{amsthm}
\usepackage{geometry}
\usepackage{tikz}
\usepackage{pgfkeys}
\usepackage{stmaryrd}
\usepackage{ifthen}
\usepackage{caption}
\usepackage{hyperref}

\newcommand{\abs}[1]{\left\lvert #1 \right\rvert}

\hypersetup{
colorlinks=true,
pdfauthor={Mickaël Kourganoff},
pdftitle={Universality theorems for linkages in homogeneous surfaces},
pdfcreator=PDFLaTeX,
pdfproducer=PDFLaTeX,
linkcolor=[rgb]{0,0,0.7},
citecolor=[rgb]{0.7,0,0},
}
\usetikzlibrary{shapes.geometric}
\usetikzlibrary{through,calc}
\usetikzlibrary{plotmarks}

\tikzset{
    scale plot marks/.is choice,
    scale plot marks/false/.code={
        \def\pgfuseplotmark##1{\pgftransformresetnontranslations\csname pgf@plot@mark@##1\endcsname}
    },
    scale plot marks/true/.style={},
    scale plot marks/.default=true
}

\tikzstyle{vertex}=[circle, fill=black, inner sep=0pt, minimum size=5pt]
\tikzstyle{edge}=[fill=black!0, scale=0.8]
\tikzstyle{fixed}=[rectangle, fill={rgb:blue,2;black,2}, inner sep=0pt, minimum size=5pt]
\tikzstyle{input}=[regular polygon, regular polygon sides=3, rotate=180, fill={rgb:red,0;green,2;black,2}, inner sep=0pt, minimum size=8pt]
\tikzstyle{output}=[regular polygon, regular polygon sides=3, fill={rgb:red,2;black,2}, inner sep=0pt, minimum size=8pt]

\usepackage{url}

\makeatletter
\def\url@leostyle{
 \@ifundefined{selectfont}{\def\UrlFont{\sf}}{\def\UrlFont{\small\ttfamily}}}
\makeatother
\newcommand{\arcosh}{\operatorname{arcosh}}

\newtheorem{prop}{Proposition}[chapter]
\newtheorem{fact}[prop]{Fact}
\newtheorem{thm}[prop]{Theorem}

\newtheorem{corollary}[prop]{Corollary}

\theoremstyle{remark}

\newtheorem*{remarque}{Remark}
\newtheorem*{exemple}{Example}
\newtheorem*{question}{Question}

\theoremstyle{definition}

\newtheorem{definition}[prop]{Definition}

\DeclareRobustCommand*{\refkempemink}{\ref*{kempe-mink}}
\DeclareRobustCommand*{\refkempehyp}{\ref*{kempe-hyp}}
\DeclareRobustCommand*{\refkempesphere}{\ref*{kempe-sphere}}

\begin{document}

\selectlanguage{english}

\title{Universality theorems for linkages in homogeneous surfaces}
\author{Mickaël Kourganoff} 

\address{
UMPA, ENS de Lyon, \newline
46, allée d'Italie,  69364 Lyon Cedex 07  France}
\email{mickael.kourganoff@ens-lyon.fr}

\begin{abstract}
A mechanical linkage is a mechanism made of rigid rods linked together by flexible joints, in which some vertices are fixed and others may move. The partial configuration space of a linkage is the set of all the possible positions of a subset of the vertices. We characterize the possible partial configuration spaces of linkages in the (Lorentz-)Minkowski plane, in the hyperbolic plane and in the sphere. We also give a proof of a differential universality theorem in the Minkowski plane and in the hyperbolic plane: for any compact manifold $M$, there is a linkage whose configuration space is diffeomorphic to the disjoint union of a finite number of copies of $M$. In the Minkowski plane, it is also true for any manifold $M$ which is the interior of a compact manifold with boundary.
\end{abstract}

\maketitle

\chapter{Introduction and generalities}
A \emph{mechanical linkage} is a mechanism made of rigid rods linked together by flexible joints. 
Mathematically, we consider a linkage as a \emph{marked} graph: lengths are assigned to the edges, and some vertices are pinned down while others may move.

A \emph{realization} of a linkage $\mathcal{L}$ in a manifold $\mathcal{M}$ is a mapping which sends each vertex of the graph to a point of $\mathcal{M}$, respecting the lengths of the edges.  
The \emph{configuration space} $\mathrm{Conf}_\mathcal{M}(\mathcal{L})$ is the set of all realizations.
Intuitively, the configuration space is the set of all the possible states of the mechanical linkage. This supposes, classically, the ambient manifold $\mathcal M$ to have a Riemannian structure: thus the configuration space may be seen as the space of ``isometric immersions'' of the metric graph $\mathcal L$ in $\mathcal M$. 

Here we will always deal with (non-trivially) marked connected graphs, that is, a non-empty set of vertices have fixed realizations (in fact, when $\mathcal{M}$ is homogeneous, considering a linkage without fixed vertices only adds a translation factor to the configuration space). Hence, our configurations spaces will be compact even if $\mathcal M$ is not compact, but rather complete.

\section{Some historical background}  \label{sectHistorical}
Most existing studies deal with the special case where $\mathcal M$ is the Euclidean plane and some with the higher dimensional Euclidean case 
(see for instance~\cite{MR2455573} and ~\cite{king}). There are also studies about polygonal linkages in the standard 2-sphere (see \cite{kmsphere}), or in the hyperbolic plane (see~\cite{kmhyp}).

 \medskip

 {\noindent{\it Universality theorems.}
When $\mathcal M$ is the Euclidean plane $\mathbb E^2$, a configuration space is an algebraic set. This set is smooth for a generic length structure on the underlying graph.

Universality theorems tend to state that, playing with mechanisms, we get any algebraic set of $\mathbb{R}^n$, and any manifold, as a configuration space! In contrast, it is a hard task to understand the topology or geometry of the configuration space of a given mechanism, even for a simple one.

Universality theorems have been announced in the ambient manifold $\mathbb{E}^2$ by Thurston in oral lectures, and then proved by Kapovich and Millson in~\cite{mk}. They have been proved in $\mathbb{E}^n$ by King~\cite{king}, and in $\mathbb{R}P^2$ and in the 2-sphere by Mnëv (see~\cite{mnev} and~\cite{mk}). 
It is our aim in  the present article to prove them in the cases of: the hyperbolic plane $\mathbb{H}^2$, the sphere $\mathbb{S}^2$ and the (Lorentz-)Minkowski plane $\mathbb{M}$.
These are simply connected homogeneous pseudo-Riemannian surfaces (the list of such spaces includes in addition the Euclidean and the de Sitter planes). 
Then it becomes natural to ask whether universality theorems hold in a more general class of manifolds, for instance on Riemannian surfaces without a homogeneity hypothesis.

In order to be more precise, it will be useful to introduce partial configuration spaces:
for $W$ a subset of the vertices of $\mathcal L$, one defines $\mathrm{Conf}_{\mathcal{M}}^W(\mathcal L)$ as the set of realizations of the subgraph induced by $W$ that extend to realizations of $\mathcal L$. One has in particular a restriction map: 
$ \mathrm{Conf}_{\mathcal M} (\mathcal L) \to \mathrm{Conf}_{\mathcal M}^W (\mathcal L) $.

If $W = \{a\}$ is a vertex of $\mathcal L$, its partial configuration space is its \emph{workspace}, \emph{i.e.} the set of all its positions in $\mathcal M$ corresponding to realizations of $\mathcal L$.

\medskip \noindent{\it Euclidean planar linkages.}
Now regarding the algebraic side of universality, the history starts (and almost ends) in 1876 with the well-known Kempe's theorem~\cite{kempe}:

\begin{thm} Any algebraic curve of the Euclidean plane $\mathbb{E}^2$, intersected with a Euclidean ball, is the workspace of some vertex of some mechanical linkage.
\end{thm}

This theorem has the following natural generalization, which we will call the \emph{algebraic universality theorem}, proved by Kapovich and Millson (see~\cite{mk}):
\begin{thm} \label{kempe-eucl} Let $A$ be a compact semi-algebraic subset (see Definition~\ref{semiAlg}) of $(\mathbb{E}^2)^n$ (identified with $\mathbb R^{2n}$). Then, $A$ is a partial configuration space $\mathrm{Conf}^W_{\mathbb{E}^2}(\mathcal L)  $  of some linkage $\mathcal L$ in $\mathbb E^2$. When $A$ is algebraic, one can choose $\mathcal L$ such that the restriction map $\mathrm{Conf}_{\mathbb{E}^2}(\mathcal L) \to A=
\mathrm{Conf}^W_{\mathbb{E}^2}(\mathcal L) $ is a smooth finite trivial covering.
\end{thm}

When $\mathrm{Conf}_{\mathbb{E}^2}(\mathcal L)$ is not a smooth manifold, as usual, by a smooth map on it, we mean the restriction of a smooth map defined on the ambient $\mathbb{R}^{2n}$.

From Theorem~\ref{kempe-eucl}, Kapovich and Millson easily derive the \emph{differential universality theorem} on the Euclidean plane:

\begin{thm} \label{univ-eucl} Any compact connected smooth manifold is diffeomorphic to one connected component of the configuration space of some linkage in the Euclidean plane $\mathbb{E}^2$. More precisely, there is a configuration space whose components are all diffeomorphic to the given differentiable manifold.
\end{thm}

Jordan and Steiner also proved a weaker version of this theorem with more elementary techniques (see~\cite{js}).

\medskip \noindent{\it How to go from the algebraic universality to the differentiable one?} The differentiable universality theorems (Theorems~\ref{univ-eucl}, \ref{univ-mink} and~\ref{univ-hyp}) follow immediately from the algebraic ones (Theorems~\ref{kempe-eucl}, \ref{kempe-mink} and~\ref{kempe-hyp}) once we know which smooth manifolds are diffeomorphic to algebraic sets. In 1952, Nash~\cite{nash} proved that for any smooth connected compact manifold $M$, one may find an algebraic set which has one component diffeomorphic to $M$. In 1973, Tognoli~\cite{tognoli} proved that there is in fact an algebraic set which is diffeomorphic to $M$ (a proof may be found in~\cite{ak}, or in \cite{bochnak}).

In the non-compact case (in which we will be especially interested), Akbulut and King~\cite{akbulutking} proved that every smooth manifold which is obtained as the interior of a compact manifold (with boundary) is diffeomorphic to an algebraic set. Note that conversely, any (non-singular) algebraic set is diffeomorphic to the interior of a compact manifold with boundary.

\section{Results} 
It is very natural to ask if these algebraic and differential universality theorems can be formulated and proved for configuration spaces in a general target space $\mathcal{M}$. Our results suggest this could be true: indeed, we naturally generalize universality theorems to the cases of $\mathcal{M}=$ $\mathbb{M}$, $\mathbb{H}^2$ and 
$\mathbb{S}^2$, the Minkowski and hyperbolic planes and the sphere, respectively. Notice that for a general $\mathcal{M}$, there is no notion of algebraic subset of $\mathcal{M}^n$! We will however
observe that there is a natural one in the cases we are considering here. In the general case, the question around Kempe's theorem could be rather formulated as: ``Characterize curves in $\mathcal{M}$ that are workspaces of some vertex of a linkage.''

\medskip \noindent{\it Minkowski planar linkages.}  These linkages are studied in Chapter~\ref{chapMink}. Classically, the structure of $\mathcal{M}$ needed to define realizations of a linkage is that of a Riemannian manifold. Observe however  that a distance, not necessarily of Riemannian type,  on $\mathcal{M}$ would also suffice for this task. But our idea here is instead  to relax positiveness of the metric. Instead of a Riemannian metric, we will assume $\mathcal M$ has a pseudo-Riemannian one.  We will actually restrict ourselves to the simple flat case where $\mathcal M$ is a linear space endowed with a non-degenerate quadratic form, and more specially to the 2-dimensional case, that is the Minkowski plane $\mathbb{M}$. On the graph side, weights of edges are no longer assumed to be positive numbers. This framework extension is mathematically natural, and may be related to the problem of the embedding of causal sets in physics, but the most important (as well as exciting) fact for us is that configuration spaces are (a priori) no longer compact, and we want to see what new spaces we get in this new setting.

The Lorentz-Minkowski plane $\mathbb{M}$ is $\mathbb R^2$ endowed with a non-degenerate indefinite quadratic form. We denote the ``space coordinate'' by $x$ and the ``time coordinate'' by $t$.

The configuration space 
$\mathrm{Conf}_{\mathbb{M}}(\mathcal L)$ is an algebraic subset (defined by polynomials of degree $2$)
of $\mathbb{M}^n = \mathbb R^{2n}$ ($n$ is the number of vertices of $\mathcal{L}$), and similarly a partial configuration space  $\mathrm{Conf}^W_{\mathbb{M}}(\mathcal L)$ is semi-algebraic (see Definiton~\ref{semiAlg}). In contrast to the Euclidean case, these sets may be non-compact (even if $\mathcal{L}$ has some fixed vertices in $\mathbb{M}$).  We will prove:

\begin{thm} \label{kempe-mink} Let $A$ be a semi-algebraic subset of $\mathbb M^n$ (identified with $\mathbb R^{2n}$). Then, $A$ is a partial configuration space 
$\mathrm{Conf}^W_{\mathbb{M}}(\mathcal L)$
of some linkage $\mathcal L$ in $\mathbb M$. When $A$ is algebraic, one can choose $\mathcal L$ such that the restriction map $\mathrm{Conf}_{\mathbb{M}}(\mathcal L) \to A$ is a smooth finite trivial covering.
\end{thm}

Somehow, considering Minkowskian linkages is the exact way of realizing non-compact algebraic sets!
In particular, Kempe's theorem extends (globally, i.e. without taking the intersection with balls) to the Minkowski plane: any algebraic curve is the workspace of one vertex of some linkage.

\begin{remarque} If the restriction map is injective, then it is a bijective algebraic morphism from $\mathrm{Conf}_{\mathbb{M}}(\mathcal L)$ to $A$, but not necessarily an algebraic isomorphism.
 In fact, it is true for non-singular complex algebraic sets that bijective morphisms are isomorphisms, but this is no longer true in the real algebraic case (see for instance \cite{Mumford}, Chapter 3).
\end{remarque}

We also have a differential version of the universality theorem in the Minkowski plane (which follows directly from Theorem~\ref{kempe-mink}, as explained at the end of Section~\ref{sectHistorical}):

\begin{thm} \label{univ-mink} For any differentiable manifold $M$ with finite topology, i.e. diffeomorphic to the interior of a compact manifold with boundary, there is a linkage in the Minkowski plane with a configuration space whose components are all diffeomorphic to $M$. More precisely, there is a partial configuration space $\mathrm{Conf}_{\mathbb{M}}^W(\mathcal L)$ which is diffeomorphic to $M$ and such that the restriction map $\mathrm{Conf}_{\mathbb{M}}(\mathcal L) \to \mathrm{Conf}_{\mathbb{M}}^W(\mathcal L)$ is a smooth finite trivial covering.
\end{thm}

\medskip \noindent{\it Hyperbolic planar linkages.} In Chapter~\ref{chapHyp}, we prove that both algebraic and differential universality theorems hold in the hyperbolic plane. The problem is that the notion of algebraic set has no intrinsic definition in the hyperbolic plane. However, it is possible to define an algebraic set in the Poincaré half-plane model $\left\{ \begin{pmatrix} x \\ y \end{pmatrix} \in \mathbb R^2 ~ \middle\arrowvert ~ y > 0 \right\}$ (and hence in $\mathbb H^2$) as an algebraic set of $\mathbb{R}^2$ which is contained in the half-plane. In fact, it turns out that the analogous definitions in the other usual models (the Poincaré disc model, the hyperboloid model, or the Beltrami-Klein model) are all equivalent. With this definition, we obtain the same results as in the Euclidean case:

\begin{thm} \label{kempe-hyp} Let $A$ be a compact semi-algebraic subset of $(\mathbb H^2)^n$ (identified with a subset of $\mathbb R^{2n}$ using the Poincaré half-plane model). Then, $A$ is a partial configuration space of some linkage $\mathcal L$ in $\mathbb H^2$. When $A$ is algebraic, one can choose $\mathcal L$ such that the restriction map $\mathrm{Conf}_{\mathbb{H}^2}(\mathcal L) \to A$ is a smooth finite trivial covering.
\end{thm}

Conversely, any partial configuration space of any linkage with at least one fixed vertex is a compact semi-algebraic subset of $(\mathbb H^2)^n$, so this theorem characterizes the sets which are partial configuration spaces (see Definiton~\ref{semiAlg} for the notion of ``semi-algebraic'').

In particular, Kempe's theorem holds in the hyperbolic plane.

And here follows the differential version:

\begin{thm} \label{univ-hyp} For any compact differentiable manifold $M$, there is a linkage in the hyperbolic plane with a configuration space whose components are all diffeomorphic to $M$. More precisely, there is a partial configuration space $\mathrm{Conf}_{\mathbb{H}^2}^W(\mathcal L)$ which is diffeomorphic to $M$ and such that the restriction map $\mathrm{Conf}_{\mathbb{H}^2}(\mathcal L) \to \mathrm{Conf}_{\mathbb{H}^2}^W(\mathcal L)$ is a smooth finite trivial covering.
\end{thm}

\medskip \noindent{\it Spherical linkages.} These linkages are the subject of Chapter~\ref{chapSphere}. In 1988, Mnëv~\cite{mnev} proved that the algebraic and differential universality theorems hold true in the real projective plane $\mathbb RP^2$ endowed with
its usual metric  as a quotient of the standard $2$-sphere. Even better, he showed that the number of copies in the differential universality for $\mathbb RP^2$ can be reduced to $1$, \emph{i.e.} any manifold is the configuration space of some linkage. As Kapovich and Millson pointed out  \cite{mk}, a direct consequence of Mnëv's theorem is the differential universality theorem for the 2-sphere (but, this time, we get several copies of the desired manifold):
\begin{thm}[Mnëv-Kapovich-Millson] \label{univ-sphere} For any compact differentiable manifold $M$, there is a linkage in the sphere with a configuration space whose components are all diffeomorphic to $M$.
\end{thm}

However, it seems impossible to use Mnëv's techniques to prove the \emph{algebraic} universality for spherical linkages: for example, all the configuration spaces of his linkages are symmetric with respect to the origin of $\mathbb R^3$. In order to obtain any semi-algebraic set as a partial configuration space, we need to start again from scratch and construct linkages specifically for the sphere.

Contrary to the Minkowski and hyperbolic cases, the generalization of the theorems to higher dimensional spheres is straightforward. Thus, we are able to prove the following:

\begin{thm} \label{kempe-sphere} Let $d \geq 2$ and let $A$ be a compact semi-algebraic subset of $(\mathbb S^d)^n$ (identified with a subset of $\mathbb R^{(d+1)n}$). Then, $A$ is a partial configuration space of some linkage $\mathcal L$ in $\mathbb S^d$.
\end{thm}

In particular, Kempe's theorem holds in the sphere.

Conversely, any partial configuration space of any linkage is a compact semi-algebraic subset of $(\mathbb S^d)^n$ (see Section~\ref{sectAlgSets}), so this theorem characterizes the sets which are partial configuration spaces.

Let us note that even  when $A$ is algebraic, our construction does not provide a linkage $\mathcal L$ such that the restriction map $\mathrm{Conf}_{\mathbb{S}^d}(\mathcal L) \to A$ is a smooth finite trivial covering. We do not know whether such a linkage exists.

\bigskip \noindent{\it Some questions.} Our results suggest naturally -- among many questions -- the following:

\begin{enumerate}
\item Besides the 2-dimensional case, are the results in the Minkowski plane true for any (finite-dimensional) linear space endowed with a non-degenerate quadratic form? And what about higher-dimensional hyperbolic spaces? It is likely that the adaptation of the $2$-dimensional proof hides no surprise, like in the Euclidean case, but it would probably require tedious work to prove it.
\item In our definition of linkages in the Minkowski plane, we allow some edges to have imaginary lengths (they are ``timelike"). Is it possible to require the graphs of Theorems~\ref{kempe-mink} and~\ref{univ-mink} to be spacelike, \emph{i.e.} require all their edges to have real lengths?
\item In all the universality theorems that we prove, we obtain a linkage whose configuration space is diffeomorphic to the sum of a finite number of copies of the given manifold $M$. Is it possible to choose this sum trivial, that is, with exactly one copy of $M$? (This question is also open in the Euclidean plane.)
\item Is the differential universality theorem true on any Riemannian manifold?
\end{enumerate}

\medskip \noindent{\it Linkages on Riemannian manifolds.} Let us give a partial answer to the last question using the following idea: just as the surface of the earth looks flat to us, any Riemannian manifold will almost behave as the Euclidean space if one considers a linkage which is small enough. However, our linkage has to be robust to small perturbations of the lengths, which is \emph{not} the case for many of the linkages described in this paper (consider for example the \emph{rigidified square linkage}).

\begin{thm} \label{thmQuelconque}
Consider a linkage $\mathcal L$ in the Euclidean space $\mathbb E^n$, with at least one fixed vertex, such that for any small perturbation of the length vector $l \in (\mathbb R_{\geq 0})^E$, the configuration space $\mathrm{Conf}_{\mathbb{R}^n}(\mathcal L)$ remains the same up to diffeomorphism. Then for any Riemannian manifold $\mathcal M$, there exists a linkage $\mathcal L_{\mathcal M}$ in $\mathcal M$ whose configuration space is diffeomorphic to $\mathrm{Conf}_{\mathbb{E}^n}(\mathcal L)$.
\end{thm}

In particular, Theorem~\ref{thmQuelconque} combined with the work of Jordan and Steiner~\cite{jordan2001compact} yields directly
\begin{corollary}
In any Riemannian surface $\mathcal M$, the differentiable universality theorem is true for compact orientable surfaces. In other words, any compact orientable surface is diffeomorphic to the configuration space $\mathrm{Conf}_\mathcal{M}(\mathcal{L})$ of some linkage $\mathcal L$.
\end{corollary}

This leads to the following
\begin{question}
Which manifolds can be obtained as the configuration space of some linkage in $\mathbb R^n$ which is robust to small perturbations (in the sense of Theorem~\ref{thmQuelconque}) ?
\end{question}

This question is probably very difficult, but it is clear that there are restrictions on such manifolds: for example, they have to be orientable (see again~\cite{jordan2001compact}).

\section{Ingredients of the proofs} \label{sectIngred} There are essentially three technical as well as conceptual tools: functional linkages, combination of elementary linkages, and regular inputs. The main idea is always the same as in all the known proofs of Universality theorems (see the proofs of Thurston, Mnëv~\cite{mnev}, King~\cite{king} or Kapovich and Millson~\cite{mk}): one combines elementary linkages to construct a \emph{``polynomial linkage''}.

\bigskip
\noindent
{\it Functional linkages.} 
One major ingredient in the proofs is the notion of functional linkages. Here we enrich the graph structure by marking two new vertex subsets $P$ and $Q$ playing the role of inputs and outputs, respectively. If the partial realization of $Q$ is determined by the partial realization of $P$, by means of a function $f: \mathrm{Conf}_\mathcal{M}^P(\mathcal{L}) \to \mathcal M^Q$ (called the \emph{input-output function}), then we say that we have a functional linkage for $f$ (for us, $\mathcal M$ will be the Minkowski plane $\mathbb M$, the hyperbolic plane $\mathbb H^2$ or the sphere $\mathbb S^d$). The Peaucellier linkage is a famous historical example: it is functional for an inversion with respect to a circle.
\begin{figure}[!h]
\begin{center}
\shorthandoff{:}
\begin{tikzpicture}[scale=2]
\node[vertex, label=below:$a$] (v5) at (-2,0) {};
\node[vertex, label=left:$g$] (v4) at (-1,0) {};
\node[vertex, label=below:$d$] (v3) at (0,0) {};
\node[vertex] (v2) at ({sqrt(2)/2},{sqrt(2)/2}) {};
\node[vertex] (v1) at ({sqrt(2)/2},{-sqrt(2)/2}) {};
\node[vertex, label=right:$e$] (v0) at ({sqrt(2)},0) {};
\draw (v4) -- (v3) -- (v2) -- (v0) -- (v1) -- (v3);
\draw (v2) -- (v5) -- (v1);
\end{tikzpicture}
\shorthandon{:}
\end{center}
\caption{In the Euclidean plane, the \emph{Peaucellier-Lipkin straight-line motion linkage} forces the point $e$ to move on a straight line. The vertices $a$ and $g$ are pinned down. It is a functional linkage for the inversion with respect to a circle centered at $a$: the input is $e$ and the output is $d$.}
\label{pl}
\end{figure}
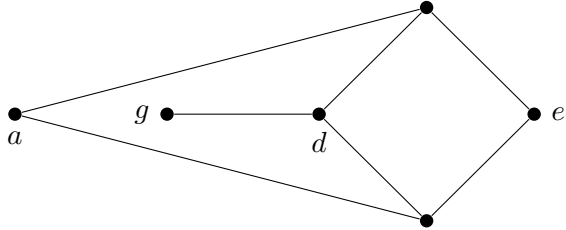

\noindent
{\it  Combination.}
Another major step in the proofs consists in proving the existence of functional linkages associated to any given polynomial $f$. This will be done by 
``combining'' elementary functional linkages. We define combination so that combining two functional linkages for the functions $f_1$ and $f_2$ provides a functional linkage for $f_1 \circ f_2$.
\bigskip

\bigskip
\noindent
{\it Elementary linkages.}
All the work then concentrates in proving the existence of linkages for suitable elementary functions (observe that even for elementary linkages one uses a combination of more elementary ones). As an example, we give the list of the elementary linkages needed to prove Theorem~\ref{kempe-mink} (in the Minkowski case):
\begin{enumerate}
\item The linkages for geometric operations:
\begin{enumerate}
\item The robotic arm linkage (Section~\ref{articArmSect}): one of the most basic linkages, used everywhere in our proofs and in robotics in general.
\item The rigidified square (Section~\ref{rigidSquareSect}): a way of getting rid of degenerate configurations of the square using a well-known construction.
\item The Peaucellier inversor (Section~\ref{peauc}): this famous linkage of the 1860's has a slightly different behavior in the Minkowski plane but achieves basically the same goal.
\item The partial $t_0$-line linkage (Section~\ref{partLineSect}): it is obtained using a Peaucellier linkage, but does not trace out the whole line.
\item The $t_0$-integer linkage (Section~\ref{intLinkSect}): it is a linkage with a discrete configuration space.
\item The $t_0$-line linkage (Section~\ref{sectLineLink}): it draws the whole line, and is obtained by combining the two previous linkages.
\item The horizontal parallelizer (Section~\ref{horizParSect}): it forces two vertices to have the same ordinate, and it is obtained by combining several line linkages.
\item The diagonal parallelizer (Section~\ref{diagParSect}): its role is similar to the horizontal parallelizer but its construction is totally different.
\end{enumerate}
\item The linkages for algebraic operations, which realize computations on the $t = 0$ line:
\begin{enumerate}
\item The average function linkage (Section~\ref{symSect}): it computes the average of two numbers, and is obtained by combining several of the previous linkages.
\item The adder (Section~\ref{adder}): it is functional for addition on the $t = 0$ line, and is obtained from several average function linkages.
\item The square function linkage (Section~\ref{sectSquareLink}): it is functional for the square function and is obtained by combining the Peaucellier linkage (which is functional for inversion) with adders. This linkage is somewhat difficult to obtain because we want the inputs to be able to move everywhere in the line, while the inversion is of course not defined at $x = 0$.
\item The multiplier (Section~\ref{multiplier}): it is functional for multiplication and is obtained from square function linkages.
\item The polynomial linkage (Section~\ref{polylink}): obtained by combining adders and multipliers, it is functional for a given polynomial function $f$. This linkage is used to prove the universality theorems: if the outputs are fixed to $0$, the inputs are allowed to move exactly in $f^{-1}(0)$.
\end{enumerate}
\end{enumerate}

\medskip

\noindent
{\it  Regular inputs.}
In our theorems, we need the restriction map $\mathrm{Conf}_{\mathcal{M}}(\mathcal L) \to \mathrm{Conf}_{\mathcal{M}}^P(\mathcal L)$ to be a smooth finite trivial covering. In the differential universality Theorem, it ensures in particular that the whole configuration space consists in several copies of the given manifold $M$. The set of regular inputs $\mathrm{Reg}_{\mathcal{M}}^P(\mathcal L)$ is the set of all realizations of the inputs which admit a neighborhood onto which the restriction map is a smooth finite covering. We have to be very careful, because even for quite simple linkages such as the robotic arm, the restriction map is not a smooth covering everywhere! There are mainly two possible reasons for the restriction map not to be a smooth covering:
\begin{enumerate}
\item One realization of the inputs may correspond to infinitely many realizations of the whole linkage (for example, when the robotic arm in Section~\ref{articArmSect} has two inputs fixed at the same location, the workspace of the third vertex is a whole circle).
\item Even if it corresponds only to a finite number of realizations, these realizations may not depend smoothly on the inputs (for example, when the robotic arm in Section~\ref{articArmSect} is stretched).
\end{enumerate}

\medskip
\noindent
{\it New difficulties in each case.}
While the idea is always the same in all known proofs of universality theorems for linkages, \emph{i.e.} combine elementary linkages to form a functional linkage for polynomials, each case has its own new difficulties due to different geometric properties, and the elementary linkages always require major changes to work correctly. Here follow examples of such differences with the Euclidean case:

\smallskip \subparagraph{\it The Minkowski case}
\begin{enumerate}
\item The Minkowski plane $\mathbb{M}$ is not isotropic: its directions are not all equivalent. Indeed, these directions have a causal character in the sense that they may be spacelike, lightlike or timelike. For example, one needs different linkages in order to draw spacelike, timelike and lightlike lines.
\item In the Euclidean plane, two circles $\mathcal{C}(x, r)$ and $\mathcal{C}(x', r')$ intersect if and only if $\lvert r - r' \rvert \leq \lVert x - x' \rVert \leq r + r'$, but in the Minkowski plane, the condition of intersection is much more complicated to state (see Section~\ref{sectInterHyp}).
\item In the Euclidean plane, one only has to consider compact algebraic sets. Applying a homothety, one may assume such a set to be inside a small neighborhood of zero, which makes the proof easier. Here, the algebraic sets are no longer compact, so we have to work with mechanisms which are able to deal with the whole plane.
\end{enumerate}

\smallskip \subparagraph{\it The hyperbolic case} \begin{enumerate}
\item The \emph{rigidified square linkage}, used extensively in all known proofs in the flat case, does not work anymore in its usual form, and does not have a simple analogue.
\item There is no natural notion of homothety: in particular, the \emph{pantograph} does not compute the middle of a hyperbolic segment, contrary to the flat cases.
\item The notion of \emph{algebraic set} is less natural than in the flat case.
\item In every standard model of the hyperbolic plane (such as the Poincaré half-plane), the expression of the distance between two points is much more complicated than in the flat case.

\end{enumerate}

\smallskip \subparagraph{\it The spherical case} 
\begin{enumerate}
\item Just as in the hyperbolic case, the curvature prevents the \emph{rigidified square linkage} from working correctly.
\item There is no natural notion of homothety.
\item In the Euclidean or hyperbolic planes, we only need to prove algebraic universality for bounded algebraic sets, which means that our functional linkages do not need to work on the whole surface. In the sphere, all the distances are uniformly bounded (even the lengths of the edges of our linkages), so we need to take into account the whole sphere when constructing linkages.
\item The compactness of the sphere also makes it difficult to construct linkages which deal with algebraic operations (addition, multiplication, division) since there is no proper embedding of $\mathbb R$ in the sphere.
\end{enumerate}

\section{Algebraic and semi-algebraic sets} \label{sectAlgSets}

In this section, we recall the standard definitions of algebraic and semi-algebraic sets. We adapt them to the Minkowski plane, the hyperbolic plane and the sphere in a natural way and state some of their properties.

\begin{definition} \label{semiAlg}
An \emph{algebraic subset of $\mathbb{R}^n$} is a set $A \subseteq \mathbb{R}^n$ such that there exist $m \in \mathbb{N}$ and $f: \mathbb{R}^n \to \mathbb{R}^m$ a polynomial such that $A = f^{-1}(0)$.

We define a \emph{semi-algebraic subset of $\mathbb{R}^n$} (see~\cite{bochnak}) as the projection of an algebraic set\footnote{Our definition of semi-algebraic sets is not the standard one, but we know from the Tarski–Seidenberg theorem that the two definitions are equivalent (see~\cite{bochnak}).}. More precisely, it is a set $B$ such that there exists $N \geq n$ and an algebraic set $A$ of $\mathbb{R}^N$ such that $B = \pi(A)$, where $\pi$ is the projection onto the first coordinates
\[ \begin{aligned} \pi: \mathbb{R}^N = \mathbb{R}^n \times \mathbb{R}^{N-n} & \to \mathbb{R}^n
\\ (x, y) & \longmapsto x. \end{aligned} \]

We define the \emph{(semi-)algebraic subsets of $\mathbb{M}^n$} by identifying $\mathbb{M}^n$ with $(\mathbb{R}^2)^n = \mathbb{R}^{2n}$.

We also define the \emph{(semi-)algebraic subsets of $(\mathbb H^2)^n$}, using the Poincaré half-plane model (see Definition~\ref{defHalfPlane}), as the (semi-)algebraic subsets of $\mathbb{R}^{2n}$ which are contained in $\left\{ \begin{pmatrix}x\\y\end{pmatrix} \in \mathbb{R}^2 \middle\arrowvert y > 0 \right\}^n$.

Finally, a \emph{(semi-)algebraic subset of $(\mathbb S^d)^n$} (for $d \geq 2$) is a semi-algebraic subset of $\mathbb R^{d+1}$ which is contained in the unit sphere of $\mathbb R^{d+1}$.
\end{definition}

\begin{prop} For any \emph{compact} semi-algebraic subset $B$ of $\mathbb R^n$, there exists $N \geq n$ and a \emph{compact} algebraic subset $A$ of $\mathbb R^N$ such that $B = \pi(A)$, where $\pi$ is the projection onto the first coordinates: $\mathbb R^N \to \mathbb R^n$.
\end{prop}
\begin{proof}
{\bf First case.} Assume for the moment that there exist polynomials $f_1, \dots, f_m: \mathbb R^n \to \mathbb R$ such that
\[ B = \left\{ x \in \mathbb R^n ~ \middle\arrowvert ~ \forall i \in \{1, \dots, m\} ~~ f_i(x) \geq 0 \right\}.\]

Let 
\[ \begin{aligned} h: \mathbb{R}^{n+m} = \mathbb{R}^n \times \mathbb{R}^m & \to \mathbb R^m
\\ \left(x, \begin{pmatrix}y_1 \\ \vdots \\ y_m \end{pmatrix}\right) & \longmapsto \begin{pmatrix} f_1(x) - y_1^2 \\ \vdots \\ f_m(x) - y_m^2 \end{pmatrix} \end{aligned} \]
and $A = h^{-1}(0)$. Then the projection of $A$ onto the first $n$ coordinates is obviously $B$. Moreover, $A$ is compact since it is the image of $B$ by the continuous function
\[ \begin{aligned} g: B & \to \mathbb{R}^{n+m} = \mathbb{R}^n \times \mathbb{R}^m
\\ x & \longmapsto \left(x, \begin{pmatrix}\sqrt{f_1(x)} \\ \vdots \\ \sqrt{f_m(x)} \end{pmatrix}\right) \end{aligned}. \]

{\bf General case.} The \emph{finiteness theorem for semi-algebraic sets} (see~\cite{bochnak}, 2.7.2) states that any closed algebraic set can be described as the union of a finite number of sets $B_1, \dots, B_k$ which satisfy the assumption of the first case: apply the first case to each of the $B_i$'s to end the proof.
\end{proof}

We end this section with two analogous propositions for the hyperbolic plane and the sphere.

\begin{prop} \label{projAlgHyp} For any compact semi-algebraic subset $B$ of $(\mathbb H^2)^n$, there exists $N \geq n$ and a compact algebraic subset $A$ of $(\mathbb H^2)^N$ (with some $N \geq n$) such that $B = \pi(A)$, where $\pi$ is the projection onto the first coordinates: $(\mathbb H^2)^N \to (\mathbb{H}^2)^n$.
\end{prop}
\begin{proof}
Let $A'$ be a compact algebraic set of $\mathbb{R}^{N'}$ (with some $N' \geq 2n$) such that $B = \pi(A')$, where $\pi$ is the projection onto the first coordinates: $\mathbb R^{N'} \to \mathbb{R}^{2n}$. Then the projection of the compact algebraic set
\footnotesize \[ A:= \left\{ \left(\begin{pmatrix}x_1 \\ y_1\end{pmatrix}, \dots, \begin{pmatrix}x_n \\ y_n\end{pmatrix}, \begin{pmatrix}x_{n+1} \\ 1\end{pmatrix}, \dots, \begin{pmatrix}x_{N'-n} \\ 1\end{pmatrix} \right) ~ \middle\arrowvert ~ \left(x_1, y_1, \dots, x_n, y_n, x_{n+1}, \dots, x_{N'-n}\right) \in A' \right\} \]
\normalsize (where $A \subseteq (\mathbb H^2)^{N'-n}$) is exactly $B$.
\end{proof}

\begin{prop} \label{projAlgSph} For any compact semi-algebraic subset $B$ of $(\mathbb S^2)^n$, there exists $N \geq n$ and a (compact) algebraic subset $A$ of $(\mathbb S^2)^N$ (with some $N \geq n$) such that $B = \pi(A)$, where $\pi$ is the projection onto the first coordinates: $(\mathbb S^2)^N \to (\mathbb{S}^2)^n$.
\end{prop}
\begin{proof}
Let $A'$ be a compact algebraic set of $\mathbb{R}^{N'}$ (with some $N' \geq 3n$) such that $B = \pi(A')$, where $\pi$ is the projection onto the first coordinates: $\mathbb R^{N'} \to \mathbb{R}^{3n}$. Since $A'$ is compact, there is a $\lambda$ such that $\pi_2(A') \in [-\lambda, \lambda]^{N'-3n}$, where $\pi_2$ is the projection onto the last coordinates: $\mathbb R^{N'} \to \mathbb{R}^{N'-3n}$. Then the projection of the compact algebraic set
\tiny \[ A:= \left\{ \left(\begin{pmatrix}x_1 \\ y_1 \\ z_1\end{pmatrix}, \dots, \begin{pmatrix}x_n \\ y_n \\ z_n\end{pmatrix}, \begin{pmatrix}x_{n+1} \\ y_{n+1} \\ 0\end{pmatrix}, \dots, \begin{pmatrix}x_{N'-2n} \\ y_{N'-2n} \\ 0\end{pmatrix} \right) ~ \middle\arrowvert ~ \begin{aligned} \left(x_1, y_1, z_1, \dots, x_n, y_n, z_n, \lambda x_{n+1}, \dots, \lambda x_{N'-2n}\right) \in A' \\ x_{n+1}^2+y_{n+1}^2 = 1, \dots, x_{N'-2n}^2+y_{N'-2n}^2 = 1 \end{aligned} \right\} \]
\normalsize (where $A \subseteq (\mathbb S^2)^{N' - 2n}$) is exactly $B$.
\end{proof}

Of course, Proposition~\ref{projAlgSph} extends to $\mathbb S^d$ with any $d \geq 2$.

\section{Generalities on linkages}

In the present section, we develop generalities on linkages which apply to the Minkowski plane, the hyperbolic plane and the sphere.
Thus, we consider a smooth manifold $\mathcal M$ endowed with a distance function 
\[ \delta: \mathcal{M} \times \mathcal{M} \to \mathbb{R}_{\geq 0} \cup i\mathbb{R}_{\geq 0}. \] 
In the case of a Riemannian manifold (in particular, for the hyperbolic plane and the sphere), the metric determines a real-valued distance on $\mathcal M$.

In the case of the Minkowski plane, $\mathcal M$ is the plane $\mathbb R^2$. Here, we argue by a naive algebraic analogy and define a distance as \[ \delta\left(\begin{pmatrix}x_1 \\ t_1\end{pmatrix}, \begin{pmatrix}x_2 \\ t_2\end{pmatrix} \right) = \sqrt{(x_2-x_1)^2 - (t_2-t_1)^2} \in \mathbb{R}_{\geq 0} \cup i\mathbb{R}_{\geq 0}. \]


Accordingly, the length structure of the linkage will be generalized by taking values in $\mathbb{R}_{\geq 0} \cup i\mathbb{R}_{\geq 0}$ (instead of $\mathbb{R}_{\geq 0}$) as follows:

\begin{definition} \label{defLinkage}
A \emph{linkage} $\mathcal{L}$ in $\mathcal{M}$ is a graph $(V, E)$ together with:
\begin{enumerate}
\item A function $l: E \to \mathbb{R}_{\geq 0} \cup i\mathbb{R}_{\geq 0}$ (which gives the length of each edge\footnote{Of course, when $\mathcal{M}$ is a Riemannian manifold, we may choose all the lengths in $\mathbb{R}_{\geq 0}$!});
\item A subset $F \subseteq V$ of \emph{fixed vertices} (represented by 
\begin{tikzpicture}[scale=1]
\node[fixed] (a) at (0,0) {};
\end{tikzpicture}
on the figures);
\item A function $\phi_0: F \to \mathcal{M}$ which indicates where the vertices of $F$ are fixed;
\end{enumerate}

When the linkage is named $\mathcal{L}_1$, we usually write $\mathcal{L}_1 = (V_1, E_1, l_1, \dots)$ and name its vertices $a_1, b_1, c_1, \dots$. If the linkage $\mathcal{L}_1$ is a copy of the linkage $\mathcal{L}$, the vertex $a_1 \in V_1$ corresponds to the vertex $a \in V$, and so on.
\end{definition}

\begin{definition}
Let $\mathcal{L}$ be a linkage in $\mathcal{M}$. A \emph{realization} of a linkage $\mathcal{L}$ in $\mathcal{M}$ is a function $\phi: V \to \mathcal{M}$ such that:
\begin{enumerate}
\item For each edge $v_1v_2 \in E$, $\delta(\phi(v_1), \phi(v_2)) = l(v_1v_2)$;
\item $\phi|_F = \phi_0$.
\end{enumerate}
\end{definition}

\begin{remarque}
On the figures of this paper, linkages are represented by abstract graphs. The edges are not necessarily represented by straight segments, and the positions of the vertices on the figures do not necessarily correspond to a realization (unless otherwise stated).
\end{remarque}

\begin{definition}
Let $\mathcal{L}$ be a linkage in $\mathcal{M}$. Let $W \subseteq V$. The \emph{partial configuration space of $\mathcal{L}$ in $\mathcal{M}$ with respect to $W$} is
\[ \mathrm{Conf}_{\mathcal{M}}^W(\mathcal{L}) = \left\{\phi|_W ~ \middle\arrowvert ~ \phi \text{ realization of }\mathcal{L} \right\}. \]

In other words, $\mathrm{Conf}_{\mathcal{M}}^W(\mathcal{L})$ is the set of all the maps $\phi: W \to \mathcal M$ which extend to realizations of $\mathcal L$. In particular, the \emph{configuration space} $\mathrm{Conf}_{\mathcal{M}}(\mathcal{L}) = \mathrm{Conf}_{\mathcal{M}}^V(\mathcal{L})$ is the set of all realizations of $\mathcal{L}$.
\end{definition}

\begin{definition} \label{defInputs}
A \emph{marked linkage} is a tuple $(\mathcal{L}, P, Q)$, where $P$ and $Q$ are subsets of $V$: $P$ is called the ``input set" and its elements, called the ``inputs", are represented by
\begin{tikzpicture}[scale=1]
\node[input] (a) at (0,0) {};
\end{tikzpicture}
on the figures, whereas $Q$ is called the ``output set" and its elements, called the ``outputs", are represented by 
\begin{tikzpicture}[scale=1]
\node[output] (a) at (0,0) {};
\end{tikzpicture}
on the figures.

The \emph{input map} $p: \mathrm{Conf}_{\mathcal{M}}(\mathcal{L}) \to \mathcal{M}^P$ is the map induced by the projection $\mathcal{M}^V \to \mathcal{M}^P$ (the restriction map). In other words, for all $\phi \in \mathrm{Conf}_{\mathcal{M}}(\mathcal{L})$, we have $p(\phi) = \phi|_P$.

Likewise, we define the \emph{output map} $q: \mathrm{Conf}_{\mathcal{M}}(\mathcal{L}) \to \mathcal{M}^Q$ by $q(\phi) = \phi|_Q$.
\end{definition}

The notion of marked linkage is not necessary to study configuration spaces. However, in the linkages we use in our proofs, some vertices play an important role (the inputs and the outputs) while others do not: this is why we always consider marked linkages. The following notion\footnote{already defined informally at the beginning of Section~\ref{sectIngred}} accounts for the names ``inputs'' and ``outputs'':

\begin{definition}
We say that $\mathcal{L}$ is a \emph{functional linkage} for the input-output function $f: \mathrm{Conf}_{\mathcal{M}}^P(\mathcal{L}) \to \mathcal{M}^Q$ if
\[ \forall \phi \in \mathrm{Conf}_{\mathcal{M}}(\mathcal{L}) ~~~~~~~~ f(p(\phi)) = q(\phi). \]
\end{definition}

\section{Regularity}

\begin{definition} \label{defR}
Let $\mathcal{L}$ be a linkage. Let $W \subseteq V$ and $\psi \in \mathrm{Conf}_{\mathcal{M}}^P(\mathcal{L})$. Let $\pi_W$ be the restriction map \[ \pi_W: \mathrm{Conf}_{\mathcal{M}}^{W \cup P}(\mathcal{L}) \to \mathrm{Conf}_{\mathcal{M}}^P(\mathcal{L}). \] We say that \emph{$\psi$ is a regular input for $W$} if there exists an open neighborhood $U \subseteq \mathrm{Conf}_{\mathcal{M}}^P(\mathcal{L})$ of $\psi$ such that $\pi_W|_{\pi_W^{-1}(U)}$ is a finite smooth covering\footnote{We do not require $U$ or $\pi_W^{-1}(U)$ to be smooth manifolds: recall that a smooth map on $\pi_W^{-1}(U)$ is, by definition, the restriction of a smooth map defined on the ambient $\mathcal M^{W \cup P}$.}.

We write $\mathrm{Reg}_\mathcal{M}^P(\mathcal{L}, W) \subseteq \mathrm{Conf}_{\mathcal{M}}^P(\mathcal{L})$ the set of regular inputs for $W$. When $W$ is the set $V$ of all vertices, we simply write $\mathrm{Reg}_\mathcal{M}^P(\mathcal{L})$.
\end{definition}

Roughly speaking, $\psi$ is a regular input for $W$ if it determines a finite number of realizations $\phi$ of $W$, and if these configurations are determined smoothly with respect to $\psi$ (in other words, $\pi_W^{-1}$ is a smooth multivalued function in a neighborhood of $\psi$).

\smallskip

The following fact is simple but essential:

\begin{fact} For any $W_1, W_2 \subseteq V$, we have \[ \mathrm{Reg}_\mathcal{M}^P(\mathcal{L}, W_1) \cap \mathrm{Reg}_\mathcal{M}^P(\mathcal{L}, W_2) \subseteq \mathrm{Reg}_\mathcal{M}^P(\mathcal{L}, W_1 \cup W_2). \] \end{fact}

Therefore, in practice, in order to prove that $\mathrm{Reg}_\mathcal{M}^P(\mathcal{L}) = \mathrm{Conf}_\mathcal{M}^P(\mathcal{L})$, we only have to prove that $\mathrm{Reg}_\mathcal{M}^P(\mathcal{L}, \{v\}) = \mathrm{Conf}_\mathcal{M}^P(\mathcal{L})$ for all $v \in V$.

\section{Changing the input set}
In this proposition, we take a linkage, then consider the same linkage with a different set of inputs $P$ and analyse the impact on $\mathrm{Reg}_\mathcal{M}^P(\mathcal{L})$.
\begin{fact} \label{changeInputs}
Let $\mathcal{L}_1 = (V_1, E_1, l_1, F_1, \phi_{01}, P_1, Q_1)$, $P_2 \subseteq V_1$ and define
\[ \mathcal{L}_2 = (V_1, E_1, l_1, F_1, \phi_{01}, P_2, Q_1). \] Recall that $p_1 : \mathrm{Conf}_\mathcal{M}(\mathcal L_1) \to \mathcal{M}^{P_1}$ and $p_2 : \mathrm{Conf}_\mathcal{M}(\mathcal L_2) \to \mathcal{M}^{P_2}$ are the respective input maps of $\mathcal L_1$ and $\mathcal L_2$. Then $\mathrm{Reg}_\mathcal{M}^{P_2}(\mathcal{L}_2)$ contains
\[ \left\{ \psi \in \mathrm{Conf}_{\mathcal{M}}^{P_2}(\mathcal{L}_2) ~ \middle\arrowvert ~ \forall \phi \in p_2^{-1}(\psi) ~~~ p_1(\phi) \in \mathrm{Reg}_\mathcal{M}^{P_1}(\mathcal{L}_1) \right\} \cap \mathrm{Reg}_\mathcal{M}^{P_2}(\mathcal{L}_2, P_1). \]
\end{fact}
\begin{proof}[Proof]
This is a simple consequence of the fact that the composition of two smooth functions is a smooth function. \end{proof}

\section{Combining linkages} \label{sectComb}
This notion is essential to construct complex linkages from elementary ones. The proofs in this section are straightforward and left to the reader.

Let $\mathcal{L}_1 = (V_1, E_1, l_1, F_1, \phi_{01}, P_1, Q_1)$ and $\mathcal{L}_2 = (V_2, E_2, l_2, F_2, \phi_{02}, P_2, Q_2)$ be two linkages, $W_1 \subseteq V_1$, and $\beta: W_1 \to V_2$.

The idea is to construct a new linkage $\mathcal{L}_3 = \mathcal{L}_1 \cup_\beta \mathcal{L}_2$ as follows:
\subparagraph{Step 1} Consider $\mathcal{L}_1 \cup \mathcal{L}_2$, the disjoint union of the two graphs $(V_1, E_1)$ and $(V_2, E_2)$.
\subparagraph{Step 2} Identify some vertices of $V_1$ with some vertices of $V_2$ via $\beta$, without removing any edge.

~

Since linkages are graphs which come with an additional structure, we need to clarify what happens to the other elements ($l$, $F$, $\phi_0$, $P$, $Q$). In particular, note that the inputs of $\mathcal{L}_2$ which are in $\beta(W_1)$ are \emph{not} considered as inputs in the new linkage $\mathcal{L}_3$.

\begin{definition}[Combining two linkages] \label{comblinkages}
We define \[ \mathcal{L}_3 = \mathcal{L}_1 \cup_\beta \mathcal{L}_2 = (V_3, E_3, l_3, F_3, \phi_{03}, P_3, Q_3) \] in the following way:
\begin{enumerate}
\item $V_3 = (V_1 \setminus W_1) \cup V_2$;
\item $E_3 = (E_1 \cap (V_1 \setminus W_1)^2) \cup (E_2 \cap V_2^2) \cup \left\{v\beta(v') ~ \middle\arrowvert ~ v \in V_1 \setminus W_1, v' \in W_1, vv' \in E_1\right\}$ \\ $\cup \left\{\beta(v)\beta(v') ~ \middle\arrowvert ~ v, v' \in W_1, vv' \in E_1\right\}$;
\item For all $v_1, v_1' \in V_1 \setminus W_1$, $w_1, w_1' \in W_1$, $v_2, v_2' \in V_2$, define
\[ \begin{aligned} l_3(v_1v_1') = l_1(v_1v_1'), & ~~~ l_3(v_1\beta(w_1)) = l_1(v_1w_1), \\ l_3(v_2v_2') = l_2(v_2v_2'), & ~~~ l_3(\beta(w_1)\beta(w_1')) = l_1(w_1w_1'); \end{aligned} \]
\item $F_3 = (F_1 \setminus W_1) \cup \beta (F_1 \cap W_1) \cup F_2$;
\item $\phi_{03}|_{F_1 \setminus W_1} = \phi_{01}|_{F_1 \setminus W_1}, ~~~ \phi_{03} \circ \beta = \phi_{01}|_{W_1}, ~~~ \phi_{03}|_{F_2 \setminus \beta(W_1)} = \phi_{02}|_{F_2 \setminus \beta(W_1)}$;
\item $P_3 = (P_1 \setminus W_1) \cup \beta (P_1 \cap W_1) \cup (P_2 \setminus \beta(W_1))$;
\item $Q_3 = (Q_1 \setminus W_1) \cup Q_2$.
\end{enumerate}
The combination of two linkages is prohibited in the following cases:
\begin{enumerate}
\item There exist $a_1, b_1 \in F_1 \cap W_1$ such that $\beta(a_1) = \beta(b_1)$ and $\phi_{01}(a_1) \neq \phi_{01}(b_1)$ (two vertices are fixed at different places but should be attached to the same other vertex).
\item There exist $a_1, b_1 \in W_1$ such that $a_1b_1 \in E_1$, $\beta(a_1)\beta(b_1) \in E_2$, and $l_1(a_1b_1) \neq l_2(\beta(a_1)\beta(b_1))$ (two edges of different lengths should join one couple of vertices).
\item There exist $a_1 \in V_1$ and $b_1, c_1 \in W_1$, such that $a_1b_1 \in E_1$, $a_1c_1 \in E_1$, $\beta(b_1) = \beta(c_1)$, and $l_1(a_1b_1) \neq l_1(a_1c_1)$ (again, two edges of different lengths should join one couple of vertices).
\end{enumerate}
\end{definition}

\begin{exemple}
Consider the two identical linkages $\mathcal{L}_1$ and $\mathcal{L}_2$:

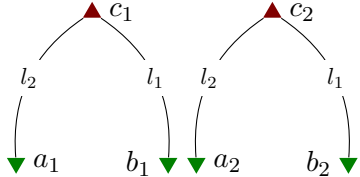
\begin{center}
\begin{tikzpicture}[scale=1]
\node[input, label=left:$a_1$] (a) at (-1,0) {};
\node[input, label=right:$b_1$] (b) at (1,0) {};
\node[output, label=right:$c_1$] (c) at (0,2) {};
\draw (c) to[bend right] node[edge] {$l_2$} (a);
\draw (c) to[bend left] node[edge] {$l_1$} (b);
\end{tikzpicture}
\begin{tikzpicture}[scale=1]
\node[input, label=left:$a_2$] (a) at (-1,0) {};
\node[input, label=right:$b_2$] (b) at (1,0) {};
\node[output, label=right:$c_2$] (c) at (0,2) {};
\draw (c) to[bend right] node[edge] {$l_2$} (a);
\draw (c) to[bend left] node[edge] {$l_1$} (b);
\end{tikzpicture}
\end{center}

The inputs of $\mathcal{L}_i$ are $a_i, b_i$ and the output is $c_i$.

To combine the two linkages, let $W_1 = \{c_1\}$ and $\beta(c_1) = a_2$. Then $\mathcal{L}_3 := \mathcal{L}_1 \cup_\beta \mathcal{L}_2$ is the following linkage:

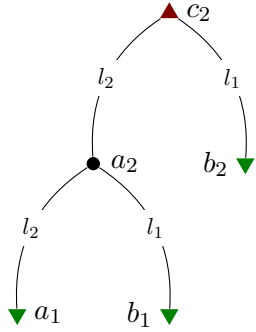
\begin{center}
\begin{tikzpicture}[scale=1]
\node[input, label=left:$a_1$] (a1) at (-1,0) {};
\node[input, label=right:$b_1$] (b1) at (1,0) {};
\node[vertex, label=right:$a_2$] (a2) at (0,2) {};
\node[input, label=right:$b_2$] (b2) at (2,2) {};
\node[output, label=right:$c_2$] (c2) at (1,4) {};
\draw (a2) to[bend right] node[edge] {$l_2$} (a1);
\draw (a2) to[bend left] node[edge] {$l_1$} (b1);
\draw (c2) to[bend right] node[edge] {$l_2$} (a2);
\draw (c2) to[bend left] node[edge] {$l_1$} (b2);
\end{tikzpicture}
\end{center}

The inputs of $\mathcal{L}_3$ are $a_1, b_1, b_2$ and the output is $c_2$.

\end{exemple}

\bigskip

We end this section with three facts whose proofs are straightforward. The first describes $\mathrm{Conf}_{\mathcal{M}}(\mathcal{L})$ when $\mathcal L$ is obtained as the combination of two linkages, the second one describes $\mathrm{Reg} _{\mathcal{M}}(\mathcal{L})$, while the third one establishes a link between the combination of functional linkages and the composition of functions.

\begin{fact} \label{propCombConfSpaces}
Let $\mathcal{L}_1$, $\mathcal{L}_2$ be two linkages, $W_1 \subseteq V_1$, $\beta: W_1 \to V_2$, and $\mathcal{L}_3 = \mathcal{L}_1 \cup_\beta \mathcal{L}_2$ be defined as in Definition~\ref{comblinkages}. Then
\[ \mathrm{Conf}_{\mathcal{M}}(\mathcal{L}_3) = \left\{ \phi_3 \in \mathcal{M}^{V_3} ~ \middle\arrowvert ~ \exists (\phi_1, \phi_2) \in \mathrm{Conf}_{\mathcal{M}}(\mathcal{L}_1) \times \mathrm{Conf}_{\mathcal{M}}(\mathcal{L}_2) \right. \] \[ \left. \phi_1|_{V_1 \setminus W_1} = \phi_3|_{V_1 \setminus W_1}, ~ \phi_1|_{W_1} = \phi_3|_{\beta(W_1)} \circ \beta, ~ \phi_2 = \phi_3|_{V_2} \right\}. \]
\end{fact}

\begin{fact} \label{propComb}
Let $\mathcal{L}_1$, $\mathcal{L}_2$ be two linkages, $W_1 \subseteq V_1$, $\beta: W_1 \to P_2$, and $\mathcal{L}_3 = \mathcal{L}_1 \cup_\beta \mathcal{L}_2$. Suppose that $\psi_3 \in \mathrm{Conf}_{\mathcal{M}}^{P_3}(\mathcal{L}_3)$ satisfies both of the following properties:
\begin{enumerate}
\item $\exists \psi_1 \in \mathrm{Reg}_\mathcal{M}^{P_1}(\mathcal{L}_1) ~~ \psi_1|_{P_1 \setminus W_1} = \psi_3|_{P_1 \setminus W_1}, \psi_1|_{P_1 \cap W_1} = \psi_3|_{\beta(P_1 \cap W_1)} \circ \beta$;
\item $\forall \phi \in p_3^{-1}(\psi_3) ~~ \phi|_{P_2} \in \mathrm{Reg}_\mathcal{M}^{P_2}(\mathcal{L}_2)$.
\end{enumerate}
Then $\psi_3 \in \mathrm{Reg}_\mathcal{M}^{P_3}(\mathcal{L}_3)$.
\end{fact}

\begin{fact} \label{propCombFunc}
Let $\mathcal{L}_1$, $\mathcal{L}_2$ be two linkages with $\mathrm{card} (Q_1) = \mathrm{card} (P_2)$.

Assume that $\mathcal{L}_1$ is a functional linkage for $f_1: \mathrm{Conf}_{\mathcal{M}}^{P_1}(\mathcal{L}_1) \to \mathcal{M}^{Q_1}$ and that $\mathcal{L}_2$ is a functional linkage for $f_2: \mathrm{Conf}_{\mathcal{M}}^{P_2}(\mathcal{L}_2) \to \mathcal{M}^{Q_2}$. Let $W_1 = Q_1$, $\beta: W_1 \to P_2$ a bijection, and $\mathcal{L}_3 = \mathcal{L}_1 \cup_\beta \mathcal{L}_2$. The bijection $\beta$ induces a bijection $\hat{\beta}$ between $\mathcal{M}^{Q_1}$ and $\mathcal{M}^{P_2}$.

Then $\mathcal{L}_3$ is functional for $f_2 \circ \hat\beta \circ f_1|_{\mathrm{Conf}_{\mathcal{M}}^{P_3}(\mathcal{L}_3)}$.
\end{fact}

\section{Appendix: Linkages on any Riemannian manifold}

The aim of this section is to prove Theorem~\ref{thmQuelconque}.

Consider a linkage $\mathcal L = (V, E, F, l, \phi_0, P, Q)$ in the Euclidean space $\mathbb E^n$ as in the statement of the theorem: we may assume without loss of generality that $\mathcal L$ is a connected graph, that the sum of the lengths of the edges is smaller than $1$, and that one of the vertices is fixed to $0$, so that the configuration space of $\mathcal L$ is a subset of $B^V$, where $B$ is the unit ball of $\mathbb E^n$. We introduce $C$ the set of all mappings $\phi : V \to B$ such that $\phi|_F = \phi_0$ (namely, those which map the fixed points to their assigned locations), and define the mapping
\[ \begin{aligned} \Phi : C & \to (\mathbb{R}_{\geq 0})^E
\\ \phi & \longmapsto \left( (v, v') \mapsto \lVert \phi(v') - \phi(v') \rVert^2 \right). \end{aligned} \]

Then the configuration space of $\mathcal L$ in $\mathbb E^n$ is $\Phi^{-1}(l^2)$. Making a small perturbation of $l$, we may assume by the Lemma of Sard that $l^2$ is a regular value of $\Phi$. By assumption, this perturbation does not change $\mathrm{Conf}_{\mathbb E^n}(\mathcal L)$, up to diffeomorphism.

Let $U$ be an open neighborhood of $0$ in $\mathbb E^n$, equipped with a metric $g$, such that $(U,g)$ is isometric to an open subset of the Riemannian manifold $\mathcal M$, and denote by $\delta$ the associated distance on $U$. Applying a linear transformation to $U$, we may assume that $g_0$ (the metric $g$ at $0$) coincides with the canonical Euclidean scalar product on $\mathbb R^n$.

For a small enough $r > 0$, the mapping
\[ \begin{aligned}
\Psi : C \times \left(( -r, r ) \setminus \{0\}\right) & \to \mathbb{R}^E \times \mathbb{R}
\\ \begin{pmatrix}\phi \\ \lambda \end{pmatrix} & \longmapsto \begin{pmatrix} \psi :  (v, v') \mapsto \frac{1}{\lambda^2} \delta^2 (\lambda \phi(v), \lambda \phi(v')) \\ \lambda \end{pmatrix}
\end{aligned} \]
is well-defined, smooth, and may be extended smoothly to $C \times (-r, r)$ (apply Taylor's formula).

Then for all small enough $\lambda \in \mathbb R$,
\[ \Psi^{-1} \begin{pmatrix}l^2 \\ \lambda \end{pmatrix} = \frac{1}{\lambda} \mathrm{Conf}_{(U, g)}(\mathcal L_\lambda) \times \{\lambda\}, \]
where $\mathcal{L}_\lambda = (V, E, F, \lambda l, \lambda \phi_0, P, Q)$. Notice that $\mathrm{Conf}_{(U, g)}(\mathcal L_\lambda)$ is diffeomorphic to the configuration space of some linkage in $\mathcal M$, since $(U,g)$ is isometric to an open set of $\mathcal M$.

The key to the proof is the following fact:

\begin{fact} \label{factQuelconque}
For all $\phi \in C$, $\Psi \begin{pmatrix} C \\ 0 \end{pmatrix} = \begin{pmatrix} \Phi(C) \\ 0 \end{pmatrix}$.
\end{fact}
\begin{proof} In this proof, for any open set $W \in \mathbb R^n$, we will write $C^1([0,1], W)$ the set of $C^1$ paths which take their values in $W$.

Let $\phi \in C$. For any small enough $\lambda \in \mathbb R$, we have:

\[ \begin{aligned} \Psi_1 \begin{pmatrix}\phi\\ \lambda\end{pmatrix} \cdot (v, v') & = \frac{1}{\lambda^2}  \inf_{\substack{\gamma \in C^1([0,1], U) \\ \gamma(0) = \lambda\phi(v), ~ \gamma(1) = \lambda\phi(v')}} \left(\int_0^1 \sqrt{g_{\gamma(t)}(\gamma'(t), \gamma'(t))}dt \right)^2
\\
& = \inf_{\substack{\gamma \in C^1([0,1], \frac{1}{\lambda} U)\\ \gamma(0) = \phi(v), ~ \gamma(1) = \phi(v')}} \left(\int_0^1 \sqrt{g_{\lambda \gamma(t)}(\gamma'(t), \gamma'(t))}dt\right)^2. \end{aligned} \]

Any $C^1$ path from $[0,1]$ to $\mathbb R^n$ takes it values in $\frac{1}{\lambda} U$ for some $\lambda > 0$. Thus, taking the limit as $\lambda \to 0$, we obtain:
\[ \Psi_1 \begin{pmatrix}\phi\\ 0\end{pmatrix} \cdot (v, v') = \inf_{\substack{\gamma \in C^1([0,1], \mathbb{R}^n) \\ \gamma(0) = \phi(v), ~ \gamma(1) = \phi(v')}} \left(\int_0^1 \sqrt{g_0 (\gamma'(t), \gamma'(t))}dt \right)^2 = \lVert \phi(v')-\phi(v) \rVert^2, \]
since $g_0$ is the canonical Euclidean scalar product.
\end{proof}

Fact~\ref{factQuelconque} shows that $\begin{pmatrix}l^2\\0\end{pmatrix}$ is a regular value of $\Psi$, and that $\Psi^{-1} \begin{pmatrix}l^2 \\ 0\end{pmatrix} = \mathrm{Conf}_{\mathbb E^n}(\mathcal L)$.

Hence, for any small enough $\lambda \in \mathbb R$, $\mathrm{Conf}_{(U, g)}(\mathcal L_\lambda)$ is diffeomorphic to $\mathrm{Conf}_{\mathbb E^n}(\mathcal L)$, which ends the proof.

\chapter{Linkages in the Minkowski plane} \label{chapMink}

The aim of this chapter is to prove Theorem~\ref{kempe-mink}.

\section{Generalities on the Minkowski plane}

\subsection{Notation} \label{notations}
The Minkowski plane $\mathbb{M}$ is $\mathbb{R}^2$ equipped with the bilinear form $\varphi\left(\begin{pmatrix}x\\t\end{pmatrix}, \begin{pmatrix}x'\\t'\end{pmatrix}\right) = xx' - tt'$. The pseudo-norm $\lVert \cdot \rVert$ is defined by $\lVert \alpha \rVert^2 = \varphi(\alpha, \alpha)$, and $\lVert \alpha \rVert \in \mathbb{R}_{\geq 0} \cup i\mathbb{R}_{\geq 0}$. The ``distance'' between $\alpha_1$ and $\alpha_2$ is defined by $\delta(\alpha_1, \alpha_2) = \lVert \alpha_1 - \alpha_2 \rVert$.

For $\alpha \in \mathbb{M}$, we write $x_\alpha$ and $t_\alpha$ the usual coordinates in $\mathbb{R}^2$, so that $\lVert \alpha \rVert^2 = x_\alpha^2 - t_\alpha^2$.

Sometimes, it will be more convenient to use \emph{lightlike coordinates}, defined by $y_\alpha = x_\alpha + t_\alpha$ and $z_\alpha = x_\alpha - t_\alpha$, so that $\lVert \alpha \rVert^2 = y_\alpha z_\alpha$.

We write $\mathcal{I} = \left\{ \begin{pmatrix}x\\t\end{pmatrix} \in \mathbb{M} ~ \middle\arrowvert ~ t = 0\right\}$.

\subsection{Intersection of two hyperbolae} \label{sectInterHyp}
In the Minkowski plane, hyperbolae play a central role (instead of circles in the Euclidean plane): for any $\alpha \in \mathbb{M}$ and $r^2 \in \mathbb{R}$, the hyperbola $\mathcal{H}(\alpha, r)$ is the set of all $\gamma \in \mathbb{M}$ such that $\delta(\alpha, \gamma)^2 = r^2$.

Let $\alpha_0, \alpha_1 \in \mathbb{M}$ and $r_0^2, r_1^2 \in \mathbb{R}$, and write $d = \lVert \alpha_0 - \alpha_1 \rVert$. Our aim in this section is to determine the cardinality of $I = \mathcal{H}(\alpha_0, r_0) \cap \mathcal{H}(\alpha_1, r_1)$.

\begin{prop} \label{nbIntersection}
If $\alpha_0 \neq \alpha_1$ and $r_0^2 r_1^2 \neq 0$, we have $\mathrm{card} (I) \leq 2$.
\end{prop}
\begin{proof}
We write $y_0 = y_{\alpha_0}$ and $z_0 = z_{\alpha_0}$. We may assume $\alpha_1 = 0$ and $y_0 \neq 0$. Then, $I$ is the set of the solutions of the system with unknown $(y, z)$:
\[
\left\{\begin{aligned}yz = r_1^2 \\(y-y_0)(z - z_0) = r_0^2\end{aligned}\right.
\]
which is equivalent to
\[
\left\{\begin{aligned}yz = r_1^2 \\y_0 z^2 - (y_0z_0 + r_1^2 - r_0^2)z + r_1^2z_0 = 0\end{aligned}\right.
\]

Thus, $z$ is one of the roots of a polynomial of degree $2$ and $y$ is fully determined by $z$, so there are at most two solutions to the system.
\end{proof}

\begin{prop} \label{intersectionOfHyperbolae}
\begin{enumerate}
\item If $r_0^2 r_1^2 < 0$ and $d^2 \neq 0$, then $\mathrm{card} (I) = 2$. Moreover, if $d'$ is the distance between the two points of $I$, then $d^2 {d'}^2 < 0$.
\item If $r_0^2 r_1^2 < 0$ and $d^2 = 0$, then $\mathrm{card} (I) = 1$.
\item If $r_0^2 r_1^2 > 0$ and $r_0^2 d^2 < 0$, then $\mathrm{card} (I) = 2$.
\end{enumerate}
\end{prop}
\begin{proof} Examine the following figures.
\end{proof}

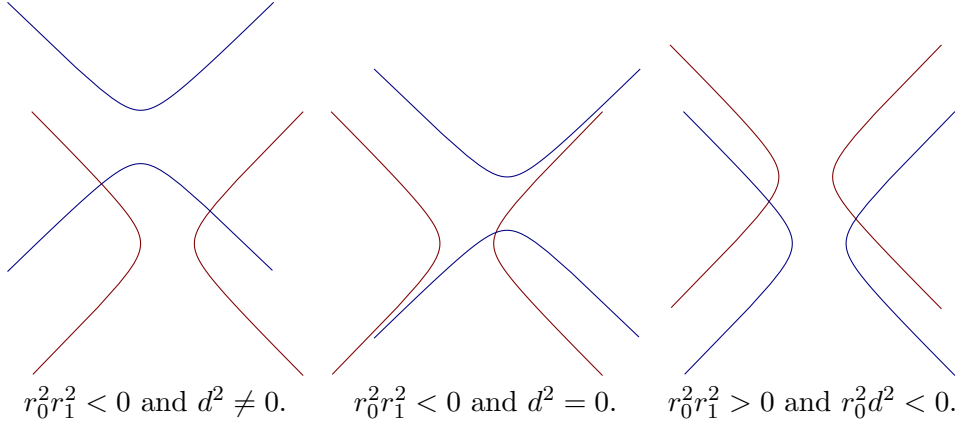
\begin{center}
\begin{tabular}{ccc}

\begin{tikzpicture} 
\draw [domain=-10:-0.1,samples=100,scale=0.25, rotate=-45, color={rgb:red,1;black,1}] plot (\x, {1/\x});
\draw [domain=0.1:10,samples=100,scale=0.25, rotate=-45, color={rgb:red,1;black,1}] plot (\x, {1/\x});

\draw [domain=-10+3:-0.1+3,samples=100,scale=0.25, rotate=45, color={rgb:blue,1;black,1}] plot (\x, {1/(\x-3)+5});
\draw [domain=0.1+3:10+3,samples=100,scale=0.25, rotate=45, color={rgb:blue,1;black,1}] plot (\x, {1/(\x-3)+5});
\end{tikzpicture} &
\begin{tikzpicture} 
\draw [domain=-10:-0.1,samples=100,scale=0.25, rotate=-45, color={rgb:red,1;black,1}] plot (\x, {1/\x});
\draw [domain=0.1:10,samples=100,scale=0.25, rotate=-45, color={rgb:red,1;black,1}] plot (\x, {1/\x});

\draw [domain=-10+3:-0.1+3,samples=100,scale=0.25, rotate=45, color={rgb:blue,1;black,1}] plot (\x, {1/(\x-3)});
\draw [domain=0.1+3:10+3,samples=100,scale=0.25, rotate=45, color={rgb:blue,1;black,1}] plot (\x, {1/(\x-3)});
\end{tikzpicture} &
\begin{tikzpicture} 
\draw [domain=-10:-0.1,samples=100,scale=0.25, rotate=-45, color={rgb:red,1;black,1}] plot (\x, {1/\x});
\draw [domain=0.1:10,samples=100,scale=0.25, rotate=-45, color={rgb:red,1;black,1}] plot (\x, {1/\x});

\draw [domain=-10+3:-0.1+3,samples=100,scale=0.25, rotate=-45, color={rgb:blue,1;black,1}] plot (\x, {1/(\x-3)-2});
\draw [domain=0.1+3:10+3,samples=100,scale=0.25, rotate=-45, color={rgb:blue,1;black,1}] plot (\x, {1/(\x-3)-2});
\end{tikzpicture}
\\ $r_0^2 r_1^2 < 0$ and $d^2 \neq 0$. & $r_0^2 r_1^2 < 0$ and $d^2 = 0$. & $r_0^2 r_1^2 > 0$ and $r_0^2 d^2 < 0$.
\end{tabular}
\end{center}

\subsection{The case of equality in the triangle inequality}
In the Minkowski plane, the triangle inequality is not always valid, but the equality case is the same as in the Euclidean plane.
\begin{prop}
Let $\alpha, \beta \in \mathbb{M}$. If $\lVert \alpha \rVert + \lVert \beta \rVert = \lVert \alpha + \beta \rVert$, then $\alpha$ and $\beta$ are colinear.
\end{prop}
\begin{proof}
We have
\[ \left(\lVert \alpha \rVert + \lVert \beta \rVert \right)^2 = \lVert \alpha + \beta \rVert^2 \]
\[ \lVert \alpha \rVert^2 + \lVert \beta \rVert^2 + 2 \lVert \alpha \rVert \lVert \beta \rVert = \lVert \alpha \rVert^2 + \lVert \beta \rVert^2 + 2 \varphi(\alpha, \beta) \]
\[\varphi(\alpha, \beta) = \lVert \alpha \rVert \lVert \beta \rVert.\]

Therefore, the discriminant of the polynomial function
\[ \lambda \longmapsto \lVert \beta \rVert^2 \lambda^2 + 2 \varphi(\alpha, \beta) \lambda + \lVert \alpha \rVert^2 \]
is zero. Thus, $\lVert \alpha + \lambda \beta \rVert^2$ is either nonnegative for all $\lambda$ or nonpositive for all $\lambda$. This means that $\alpha$ and $\beta$ are not linearly independent.
\end{proof}

\subsection{The dual linkage} \label{sectDualLinkage}
Let $\mathcal{L}_1$ be a linkage in the Minkowski plane. We define the reflection
\[ \begin{aligned} s: \mathbb{C} & \to \mathbb{C}
\\ a + ib & \longmapsto b + ia \end{aligned} \]
and construct $\mathcal{L}_2$, the \emph{dual linkage} of $\mathcal{L}_1$, by
\begin{enumerate}
\item $V_2 = V_1$
\item $E_2 = E_1$
\item $F_2 = F_1$
\item $P_2 = P_1$
\item $Q_2 = Q_1$
\item $l_2 = s \circ l_1$
\item $\phi_{02} = s \circ \phi_{01}$ (with $\mathbb{R}^2$ identified to $\mathbb{C}$ with the coordinates $(x, t)$).
\end{enumerate}

For all $W \subseteq V_1$, this linkage satisfies
\[ \mathrm{Conf}_{\mathbb{M}}^W(\mathcal{L}_2) = \left\{ s \circ \phi ~ \middle\arrowvert ~ \phi \in \mathrm{Conf}_{\mathbb{M}}^W(\mathcal{L}_1) \right\}, \]
\[ \mathrm{Reg}_{\mathbb{M}}^{P_2}(\mathcal{L}_2) = \left\{ s \circ \phi ~ \middle\arrowvert ~ \phi \in \mathrm{Reg}_{\mathbb{M}}^W(\mathcal{L}_1) \right\}. \]

\section{Elementary linkages for geometric operations}

\subsection{The robotic arm linkage} 
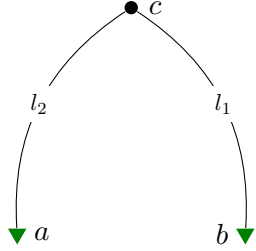
\label{articArmSect}

\begin{center}
\begin{tikzpicture}[scale=1.5]
\node[input, label=left:$a$] (a) at (-1,0) {};
\node[input, label=right:$b$] (b) at (1,0) {};
\node[vertex, label=right:$c$] (c) at (0,2) {};
\draw (c) to[bend right] node[edge] {$l_2$} (a);
\draw (c) to[bend left] node[edge] {$l_1$} (b);
\end{tikzpicture}
\end{center}

We let $P = \{a, b\}$ and $F = \emptyset$. We assume $l_1 \neq 0$ and $l_2 \neq 0$.

We translate Proposition~\ref{intersectionOfHyperbolae} in terms of linkages.
\begin{fact} \label{articArm}
\begin{enumerate}
\item If $l_1^2$ and $l_2^2$ have different signs,
\[ \mathrm{Conf}_{\mathbb{M}}^P(\mathcal{L}) \supseteq \mathrm{Reg}_\mathbb{M}^P(\mathcal{L}) \supseteq \left\{ \psi \in \mathbb{M}^P ~ \middle\arrowvert ~ \lVert \psi(a) - \psi(b) \rVert \neq 0 \right\} \]
\item If $l_1^2$ and $l_2^2$ have the same sign,
\[ \mathrm{Conf}_{\mathbb{M}}^P(\mathcal{L}) \supseteq \mathrm{Reg}_\mathbb{M}^P(\mathcal{L}) \supseteq \left\{ \psi \in \mathbb{M}^P ~ \middle\arrowvert ~ \lVert \psi(a) - \psi(b) \rVert^2 \cdot l_1^2 < 0 \right\} \]
\item More generally, let $\psi \in \mathrm{Conf}_{\mathbb{M}}^P(\mathcal{L})$. If the intersection $\mathcal{H}(\psi(a), l_2) \cap \mathcal{H}(\psi(b), l_1)$ contains exactly two elements, then $\psi \in \mathrm{Reg}_\mathbb{M}^P(\mathcal{L})$.
\end{enumerate}
\end{fact}
\begin{proof}
The fact that the intersection $\mathcal{H}(\psi(a), l_2) \cap \mathcal{H}(\psi(b), l_1)$ contains exactly two elements implies that these are obtained from simple roots of a polynomial of degree $2$ (see the proof of Proposition~\ref{nbIntersection}). Therefore, locally, the roots depend smoothly on the coefficients.
\end{proof}

\subsection{The rigidified square linkage} \label{rigidSquareSect}

This linkage is well-known in the Euclidean plane. It is the usual solution to the problem of \emph{degenerate configurations} of the square. It is very useful to notice that it does behave in the same way in the Minkowski plane.

We first explain why we need to rigidify square linkages. If one considers the ordinary square linkage (see the following figure), there are many realizations $\phi$ in which $\phi(a)\phi(b)\phi(c)\phi(d)$ is not a parallelogram (we call these realizations \emph{degenerate realizations of the square}).

\begin{center}
\begin{tabular}{ccc}
$\vcenter{\hbox{
\begin{tikzpicture}[scale=1]
\node[vertex, label=left:$a$] (a) at (-1,1) {};
\node[vertex, label=left:$b$] (b) at (-1,-1) {};
\node[vertex, label=right:$c$] (c) at (1,-1) {};
\node[vertex, label=right:$d$] (d) at (1,1) {};
\draw (a) to[bend right=20] node[edge] {$l$} (b);
\draw (b) to[bend right=20] node[edge] {$l$} (c);
\draw (c) to[bend right=20] node[edge] {$l$} (d);
\draw (d) to[bend right=20] node[edge] {$l$} (a);
\end{tikzpicture}
}}$
&
$\vcenter{\hbox{
\begin{tikzpicture}[scale=1]
\node[vertex, label=left:$a$] (a) at (2,1) {};
\node[vertex, label=left:$b$] (b) at (0,0) {};
\node[vertex, label=right:$c$] (c) at ({sqrt(3)},0) {};
\node[vertex, label=right:$d$] (d) at ({2+sqrt(3)},1) {};
\draw (a) to node[edge] {$l$} (b);
\draw (b) to node[edge] {$l$} (c);
\draw (c) to node[edge] {$l$} (d);
\draw (d) to node[edge] {$l$} (a);
\end{tikzpicture}
}}$
&
$\vcenter{\hbox{
\begin{tikzpicture}[scale=1]
\node[vertex, label=left:$a$] (a) at (2,1) {};
\node[vertex, label=left:{$b=d$}] (b) at (0,0) {};
\node[vertex, label=right:$c$] (c) at ({sqrt(3)},0) {};
\draw (a) to node[edge] {$l$} (b);
\draw (b) to node[edge] {$l$} (c);
\end{tikzpicture}
}}$
\\
\begin{minipage}{4cm}
\begin{center}
The ordinary square linkage
\end{center}
\end{minipage}
&
\begin{minipage}{4cm}
\begin{center}
A realization of the ordinary square linkage
\end{center}
\end{minipage}
&
\begin{minipage}{4cm}
\begin{center}
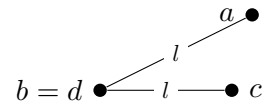
A degenerate realization of the ordinary square linkage
\end{center}
\end{minipage}
\end{tabular}
\end{center}

\bigskip

In degenerate realizations, two vertices are sent to the same point of $\mathbb{M}$.

To avoid degenerate realizations, we add two vertices and five edges to the square $abcd$. We call this operation ``rigidifying the square''.

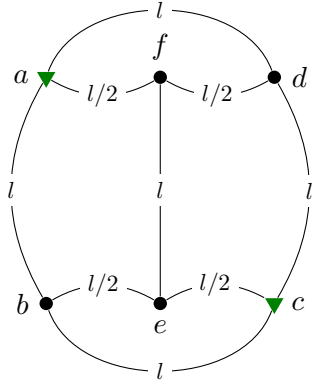
\begin{center}
\begin{tikzpicture}[scale=1.5]
\node[input, label=right:$a$] (a) at (-1,1) {};
\node[vertex, label=left:$b$] (b) at (-1,-1) {};
\node[input, label=left:$c$] (c) at (1,-1) {};
\node[vertex, label=right:$d$] (d) at (1,1) {};
\node[vertex, label=below:$e$] (e) at (0,-1) {};
\node[vertex, label=above:$f$] (f) at (0,1) {};
\draw (a) to[bend right] node[edge] {$l$} (b);
\draw (b) to[bend right=70] node[edge] {$l$} (c);
\draw (c) to[bend right] node[edge] {$l$} (d);
\draw (d) to[bend right=70] node[edge] {$l$} (a);
\draw (a) to[bend right] node[edge] {$l/2$} (f);
\draw (f) to[bend right] node[edge] {$l/2$} (d);
\draw (c) to[bend right] node[edge] {$l/2$} (e);
\draw (e) to[bend right] node[edge] {$l/2$} (b);
\draw (e) to node[edge] {$l$} (f);
\end{tikzpicture}
\end{center}

This linkage is called the \emph{rigidified square}. The input set is $P = \{a, c\}$.

We assume $l \neq 0$.

\begin{prop} \label{rigidSquare}
\begin{enumerate}
\item For all $\phi \in \mathrm{Conf}_{\mathbb{M}}(\mathcal{L})$ we have \[ \phi(b) - \phi(a) = \phi(c) - \phi(d) \] ($\phi(a)\phi(b)\phi(c)\phi(d)$ is an ``affine parallelogram'').

\item For all $\phi \in \mathrm{Conf}_{\mathbb{M}}(\mathcal{L})$ such that $\phi(b) \neq \phi(d)$ and $\phi(a) \neq \phi(c)$, we have $\phi|_P \in \mathrm{Reg}_\mathbb{M}^P(\mathcal{L})$. In particular, $\mathrm{Reg}_\mathbb{M}^P(\mathcal{L})$ contains
\[ \left\{\psi \in \mathbb{M}^{\{a, c\}} \middle\arrowvert \lVert \psi(a) - \psi(c) \rVert \cdot l^2 < 0 \right\}. \]
\end{enumerate}
\end{prop}

\begin{proof}
\begin{enumerate}
\item Let $\phi \in \mathrm{Conf}_{\mathbb{M}}(\mathcal{L})$. From the equality case in the triangle inequality, we have $\phi(f) = \frac{\phi(a) + \phi(d)}{2}$ and $\phi(e) = \frac{\phi(b) + \phi(c)}{2}$.
\subparagraph{Case 1: $\phi(a) = \phi(c)$.} In this case, $\lVert\phi(f) - \phi(c) \rVert = \lVert \phi(f) - \phi(a) \rVert$. Therefore, \[\lVert \phi(c) - \phi(e) \rVert + \lVert \phi(f) - \phi(c) \rVert = \lVert \phi(f) - \phi(e) \rVert,\] so $\phi(e)$, $\phi(c)$ and $\phi(f)$ are aligned, so $\phi(b), \phi(c)$ and $\phi(d)$ are aligned and therefore $\phi(b) - \phi(a) = \phi(c) - \phi(d)$.
\subparagraph{Case 2: $\phi(d) = \frac{\phi(a) + \phi(c)}{2}$.} We have \[ \lVert \phi(b) - \phi(a) \rVert + \lVert \phi(b) - \phi(c) \rVert = \lVert \phi(c) - \phi(a) \rVert ~~ (= 2l),\] so $\phi(a)$, $\phi(c)$ and $\phi(b)$ are aligned, thus $\phi(d) = \phi(b)$. We are taken back to the first case.
\subparagraph{Case 3: $\phi(d) \neq \frac{\phi(a) + \phi(c)}{2}$ and $\phi(a) \neq \phi(c)$.} Let $I = \mathcal{H}(\phi(a), l) \cap \mathcal{H}(\phi(c), l)$. We have $\phi(d) \in I$, $\phi(b) \in I$, and $\mathrm{card} (I) \leq 2$.

We have $\phi(a) + \phi(c) - \phi(d) \in I$. If $\phi(a) + \phi(c) - \phi(d) = \phi(d)$ then we are in the second case. If not, we have $I = \{\phi(d), \phi(a) + \phi(c) - \phi(d)\}$ and therefore, either $\phi(b) = \phi(d)$ (this is again the first case) or $\phi(b) = \phi(a) + \phi(c) - \phi(d)$, \emph{i.e.} $\phi(b) - \phi(a) = \phi(c) - \phi(d)$.

\item This is a consequence of Fact~\ref{articArm}.
\end{enumerate}
\end{proof}

\subsection{The Peaucellier inversor} 
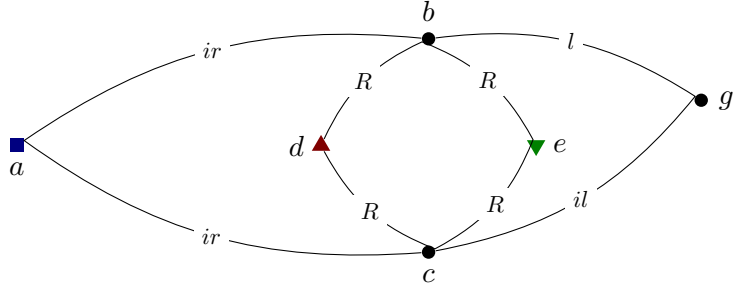
\label{peauc}

\begin{center}
\begin{tikzpicture}[scale=2]
\node[fixed, label=below:$a$] (v5) at (-2,0) {};
\node[output, label=left:$d$] (v3) at (0,0) {};
\node[vertex, label=above:$b$] (v2) at ({sqrt(2)/2},{sqrt(2)/2}) {};
\node[vertex, label=below:$c$] (v1) at ({sqrt(2)/2},{-sqrt(2)/2}) {};
\node[input, label=left:$e$] (v0) at ({sqrt(2)},0) {};
\node[vertex, label=right:$g$] (v6) at (2.5,0.3) {};
\draw (v3) to[bend left=20] node[edge] {$R$} (v2) to[bend left=20]node[edge] {$R$}  (v0) to[bend left=20]node[edge] {$R$} (v1) to[bend left=20] node[edge] {$R$} (v3);
\draw (v2) to[bend right=20] node[edge] {$ir$} (v5) to[bend right=20] node[edge] {$ir$} (v1);
\draw (v2) to[bend left=20] node[edge] {$l$} (v6) to[bend left=20] node[edge] {$il$} (v1);
\end{tikzpicture}
\end{center}

Choose $r, R, l > 0$ and let $F = \{a\}$, $\phi_0(a) = \begin{pmatrix}0 \\ 0 \end{pmatrix}$, $P = \{e\}$, $Q = \{d\}$. The square $bdce$ is rigidified (see Section~\ref{rigidSquareSect}), but for convenience, we do not draw on the figure the vertices which are necessary for the rigidification.

The vertex $g$ and the two edges $(bg)$ and $(cg)$ are not part of the traditional Peaucellier inversor, but they are here to prevent $\phi(b)$ and $\phi(c)$ from being equal.

\begin{figure}[!ht]
\begin{center}

\begin{tikzpicture}[scale=0.4]
\node[vertex, label=below:$b$] (b) at (-7.05, 20.58) {};
\node[input, label=above:$e$] (e) at (0.62, 14.64) {};
\node[vertex, label=below:$c$] (c) at (5.74, 12.94) {};
\node[output, label=below:$d$] (d) at (-1.93, 18.91) {};
\node[fixed, label=below:$a$] (a) at (3.40, 9.95) {};
\node[vertex, label=below:$g$] (g) at (-4.06, 22.94) {};
\draw (b) -- (e) -- (c) -- (d) -- (b);
\draw (b) -- (a) -- (c);
\draw (b) -- (g) -- (c);
\end{tikzpicture}
\hspace{2cm}
\begin{tikzpicture}[scale=0.4]
\node[vertex, label=above:$b$] (b) at (2.45, 12.09) {};
\node[input, label=above:$e$] (e) at (-3.09, 9.40) {};
\node[vertex, label=below:$c$] (c) at (2.07, 7.62) {};
\node[output, label=below:$d$] (d) at (7.61, 10.30) {};
\node[fixed, label=left:$a$] (a) at (3.40, 9.95) {};
\node[vertex, label=below:$g$] (g) at (4.69, 10.80) {};
\draw (b) -- (e) -- (c) -- (d) -- (b);
\draw (b) -- (a) -- (c);
\draw (b) -- (g) -- (c);
\end{tikzpicture}
\end{center}
\caption{Two realizations of the same Peaucellier inversor in the Minkowski plane.}
\end{figure}
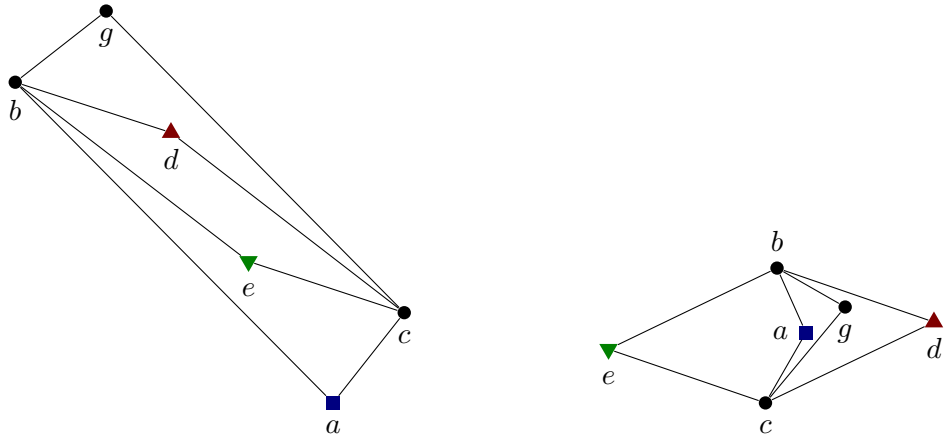

In the Euclidean plane, it is well-known that the Peaucellier linkage is a functional linkage for the inversion with respect to the circle $\mathcal{C} \left(0, \sqrt{\lvert R^2 - r^2 \rvert}\right)$, that is, the function $\alpha \longmapsto \frac{\lvert R^2 - r^2 \rvert}{\lVert \alpha \rVert^2} \alpha$. In the Minkowski plane, we will prove that it is essentially functional for inversion with respect to the hyperbola $\mathcal{H} \left(0, \sqrt{R^2 + r^2}\right)$. More precisely, it is functional for the function $\alpha \longmapsto -\frac{R^2 + r^2}{\lVert \alpha \rVert^2} \alpha$ (in the version of the Peaucellier inversor which we choose, a ``-'' sign appears).

\begin{prop} \label{peaucFonct}
For all $\phi \in \mathrm{Conf}_{\mathbb{M}}(\mathcal{L})$, we have $\lVert \phi(e) \rVert \neq 0$ and $\phi(d) = - \frac{R^2 + r^2}{\lVert \phi(e) \rVert^2} \phi(e)$.
\end{prop}
\begin{proof}
Let $\phi \in \mathrm{Conf}_{\mathbb{M}}(\mathcal{L})$. We know that $\phi(b) \neq \phi(c)$ (because the edges $(gb)$ and $(gc)$ do not have the same lengths) so the intersection of the two hyperbolae $\mathcal{H}(\phi(e), R)$ and $\mathcal{H}(\phi(a), ir)$ is exactly $\{\phi(b), \phi(c)\}$. Moreover, $\lVert \phi(e) \rVert \neq 0$ (because the edges $(be)$ and $(ba)$ do not have the same lengths).

Then, $(y_{\phi(b)}, z_{\phi(b)})$ and $(y_{\phi(c)}, z_{\phi(c)})$ are the two solutions of the following system with unknown $(y, z)$:

\[
\left\{ \begin{aligned} yz & = -r^2 \\ (y-y_{\phi(e)})(z - z_{\phi(e)}) & = R^2. \end{aligned} \right.
\]

This system is equivalent to

\[
\left\{ \begin{aligned} yz & = -r^2 \\ -y_{\phi(e)}z^2 + (y_{\phi(e)}z_{\phi(e)}-r^2-R^2)z + r^2 z_{\phi(e)} & = 0. \end{aligned} \right.
\]

We deduce that
\[ z_{\phi(b)} + z_{\phi(c)} = z_{\phi(e)} - \frac{r^2 + R^2}{y_{\phi(e)}} \]
and similarly
\[ y_{\phi(b)} + y_{\phi(c)} = y_{\phi(e)} - \frac{r^2 + R^2}{z_{\phi(e)}} \]
which gives the desired result, since $\phi(d) = \phi(b) + \phi(c) - \phi(e)$.
\end{proof}

\begin{fact} \label{confPeauc}
The workspace of the vertex $e$, $\mathrm{Conf}_{\mathbb{M}}^P(\mathcal{L})$, contains the spacelike cone
\[\left \{\alpha \in \mathbb{M} ~ \middle\arrowvert ~ \lVert \alpha \rVert^2 > 0 \right \}.\]
\end{fact}

\begin{proof}
Here, we use Proposition~\ref{intersectionOfHyperbolae} at each step. Choose any $\alpha$ in the spacelike cone, let $\phi(e) = \alpha$, and choose $\phi(b)$ and $\phi(c)$ such that $\{\phi(b), \phi(c)\} = \mathcal H(\phi_0(a), ir) \cap \mathcal H(\alpha, R)$. Then, $\lVert \phi(c) - \phi(b) \rVert^2 < 0$, so it is possible to choose $\phi(d)$ such that $\mathcal H(\phi(b), R) \cap \mathcal H(\phi(c), R) = \{ \phi(d), \phi(e) \}$. Finally, choose $\phi(g)$ in the intersection $\mathcal H(\phi(b), l) \cap \mathcal H(\phi(c), il)$.
\end{proof}

\begin{prop} \label{regPeauc}
For this linkage, $\mathrm{Reg}_\mathbb{M}^P(\mathcal{L}) = \mathrm{Conf}_{\mathbb{M}}^P(\mathcal{L})$.
\end{prop}

\begin{proof}
We give a detailed proof in order to illustrate the use of Fact~\ref{propComb}. This method is the key to many proofs concerning $\mathrm{Reg}_\mathbb{M}^P(\mathcal{L})$, later in this paper.

The Peaucellier inversor may be seen as the combination of the following linkages:
\\ $\mathcal{L}_1$: a robotic arm $\{a_1, c_1, e_1\}$ with one input $e_1$ and one fixed vertex $a_1$, one edge $\{a_1, c_1\}$ of length $ir$ and one edge $\{c_1, e_1\}$ of length $R$;
\\ $\mathcal{L}_2$: a robotic arm $\{a_2, b_2, e_2\}$ with one input $e_2$ and one fixed vertex $a_2$, one edge $\{a_2, b_2\}$ of length $ir$ and one edge $\{b_2, e_2\}$ of length $R$;
\\ $\mathcal{L}_3$: a rigidified square $\{b_3, d_3, c_3, e_3\}$ with inputs $b_3$, $c_3$ and four edges of length $R$;
\\ $\mathcal{L}_4$: a robotic arm $\{b_4, g_4, c_4\}$ with inputs $b_4, c_4$, one edge $\{b_4, g_4\}$ of length $l$ and one edge $\{g_4, c_4\}$ of length $il$.

\noindent
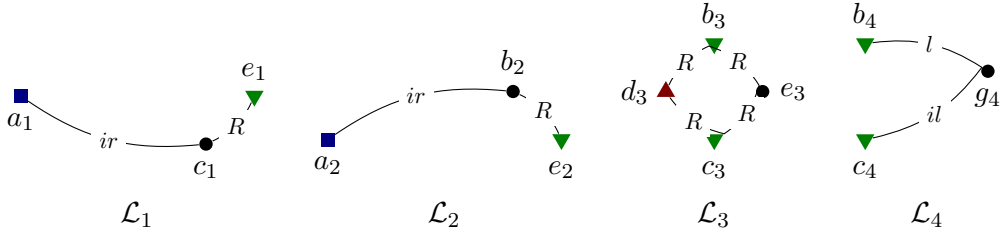
\begin{tabular}{cccc}

\begin{tikzpicture}[scale=0.9]
\node[fixed, label=below:$a_1$] (v5) at (-2,0) {};
\node[vertex, label=below:$c_1$] (v1) at ({sqrt(2)/2},{-sqrt(2)/2}) {};
\node[input, label=below:$e_1$] (v0) at ({sqrt(2)},0) {};
\draw (v0) to[bend left=20]node[edge] {$R$} (v1);
\draw (v5) to[bend right=20] node[edge] {$ir$} (v1);
\end{tikzpicture}
&
\begin{tikzpicture}[scale=0.9]
\node[fixed, label=below:$a_2$] (v5) at (-2,0) {};
\node[vertex, label=above:$b_2$] (v2) at ({sqrt(2)/2},{sqrt(2)/2}) {};
\node[input, label=above:$e_2$] (v0) at ({sqrt(2)},0) {};
\draw (v2) to[bend left=20]node[edge] {$R$}  (v0);
\draw (v2) to[bend right=20] node[edge] {$ir$} (v5);
\end{tikzpicture}
&
\begin{tikzpicture}[scale=0.9]
\node[output, label=left:$d_3$] (v3) at (0,0) {};
\node[input, label=below:$b_3$] (v2) at ({sqrt(2)/2},{sqrt(2)/2}) {};
\node[input, label=above:$c_3$] (v1) at ({sqrt(2)/2},{-sqrt(2)/2}) {};
\node[vertex, label=right:$e_3$] (v0) at ({sqrt(2)},0) {};
\draw (v3) to[bend left=20] node[edge] {$R$} (v2) to[bend left=20]node[edge] {$R$}  (v0) to[bend left=20]node[edge] {$R$} (v1) to[bend left=20] node[edge] {$R$} (v3);
\end{tikzpicture}
&
\begin{tikzpicture}[scale=0.9]
\node[input, label=below:$b_4$] (v2) at ({sqrt(2)/2},{sqrt(2)/2}) {};
\node[input, label=above:$c_4$] (v1) at ({sqrt(2)/2},{-sqrt(2)/2}) {};
\node[vertex, label=below:$g_4$] (v6) at (2.5,0.3) {};
\draw (v2) to[bend left=20] node[edge] {$l$} (v6) to[bend left=20] node[edge] {$il$} (v1);
\end{tikzpicture}
\\ $\mathcal{L}_1$ & $\mathcal{L}_2$ & $\mathcal{L}_3$ & $\mathcal{L}_4$
\end{tabular}

\bigskip

We combine the linkages in the following way (observe that the name of the vertices are chosen so that each $\beta_i$ preserves the letters and only changes indices):
\begin{enumerate}
\item Let $W_1 = \{c_1, e_1\}$. Let $\beta_1(c_1) = c_3$, $\beta_1(e_1) = e_3$. Let $\mathcal{L}_5 = \mathcal{L}_1 \cup_{\beta_1} \mathcal{L}_3$. The input set of $\mathcal{L}_5$ is $P_5 = \{e_3, b_3\}$.
\item Let $W_2 = \{b_2, e_2\}$. Let $\beta_2(b_2) = b_3$, $\beta_2(e_2) = e_3$. Let $\mathcal{L}_6 = \mathcal{L}_2 \cup_{\beta_2} \mathcal{L}_5$. The input set of $\mathcal{L}_6$ is $P_6 = \{e_3\}$.
\item Let $W_6 = \{b_3, c_3\}$. Let $\beta_6(b_3) = b_4$, $\beta_6(c_3) = c_4$. Let $\mathcal{L}_7 = \mathcal{L}_6 \cup_{\beta_6} \mathcal{L}_4$. The input set of $\mathcal{L}_7$ is $P_7 = \{e_3\}$.
\end{enumerate}

The linkage $\mathcal{L}_7$ is exactly the Peaucellier linkage.

Let $\psi \in \mathrm{Conf}_{\mathbb{M}}^{P_1}(\mathcal{L}_1)$ such that the intersection $\mathcal{H}(0, ir) \cap \mathcal{H}(\psi(e_1), R)$ has cardinality $2$. Facts~\ref{articArm} and~\ref{changeInputs} show that $\psi \in \mathrm{Reg}_{\mathbb{M}}^{P_1}(\mathcal{L}_1)$.

We may naturally identify $\mathrm{Conf}_{\mathbb{M}}^{\{e_3, b_4\}}(\mathcal{L}_7)$ with a subset $C$ of $\mathrm{Conf}_{\mathbb{M}}^{P_5}(\mathcal{L}_5)$ (identifying $b_4$ with $b_3$). Let us show, using Fact~\ref{propComb}, that $C$ is in fact a subset of $\mathrm{Reg}_{\mathbb{M}}^{P_5}(\mathcal{L}_5)$. Let $\psi \in C$, and let $\phi \in \mathrm{Conf}_{\mathbb{M}}^{V_7}(\mathcal{L}_7)$ such that $\phi(e_3) = \psi(e_3)$ and $\phi(b_4) = \psi(b_3)$. Let $\psi_1 \in \mathbb{M}^{\{e_1\}}$ defined by $\psi_1(e_1) = \psi(e_3)$: since $\phi(b_4) \neq \phi(c_3)$, the intersection $\mathcal{H}(0, ir) \cap \mathcal{H}(\psi_1(e_1), R)$ has cardinality at least $2$, but it is in fact exactly $2$ from Proposition~\ref{nbIntersection}. Therefore, $\psi_1 \in \mathrm{Reg}_{\mathbb{M}}^{P_1}(\mathcal{L}_1)$, so the first hypothesis of Fact~\ref{propComb} is satisfied. For the second hypothesis, we need to show that $\phi|_{P_3} \in \mathrm{Reg}_{\mathbb{M}}^{P_3}(\mathcal{L}_3)$. We know that $\phi(b_4) \neq \phi(c_4)$, and from Proposition~\ref{peaucFonct}, we also know that $\phi(e_3) \neq \phi(d_3)$. Therefore, Proposition~\ref{rigidSquare} tells us that $\phi|_{P_3} \in \mathrm{Reg}_{\mathbb{M}}^{P_3}(\mathcal{L}_3)$. The two hypotheses of Fact~\ref{propComb} are satisfied, so $C \subseteq \mathrm{Reg}_{\mathbb{M}}^{P_5}(\mathcal{L}_5)$.

In the same way, one may show that $\mathrm{Conf}_{\mathbb{M}}^{\{e_3\}}(\mathcal{L}_7) \subseteq \mathrm{Reg}_{\mathbb{M}}^{P_6}(\mathcal{L}_6)$, and finally, that $\mathrm{Conf}_{\mathbb{M}}^{\{e_3\}}(\mathcal{L}_7) \subseteq \mathrm{Reg}_{\mathbb{M}}^{P_7}(\mathcal{L}_7)$, so $\mathrm{Reg}_\mathbb{M}^P(\mathcal{L}_7) = \mathrm{Conf}_{\mathbb{M}}^P(\mathcal{L}_7)$.
\end{proof}

\begin{prop} \label{equivPeauc}
For all $\phi \in \mathrm{Conf}_{\mathbb{M}}(\mathcal{L})$ we have the equivalence
\[ \phi(d) \in \mathcal{H}\left(\begin{pmatrix}0\\-1\end{pmatrix}, i\right) \Longleftrightarrow y_{\phi(e)} - z_{\phi(e)} = - (R^2 + r^2). \]
\end{prop}
\begin{proof}
Let $\phi \in \mathrm{Conf}_{\mathbb{M}}(\mathcal{L})$. The following lines are equivalent:
\[ \phi(d) \in \mathcal{H}\left(\begin{pmatrix}0\\-1\end{pmatrix}, i\right) \]
\[ (y_{\phi(d)} + 1)(z_{\phi(d)} - 1) = -1 \]
\[ \left(- \frac{(R^2 + r^2)}{\lVert \phi(e) \rVert^2} y_{\phi(e)} + 1\right)\left(- \frac{(R^2 + r^2)}{\lVert \phi(e) \rVert^2} z_{\phi(e)} - 1\right) = -1 \]
\[ y_{\phi(e)} - z_{\phi(e)} = - (R^2 + r^2). \]
\end{proof}

\subsection{The partial \texorpdfstring{$t_0$}{t0}-line linkage} 
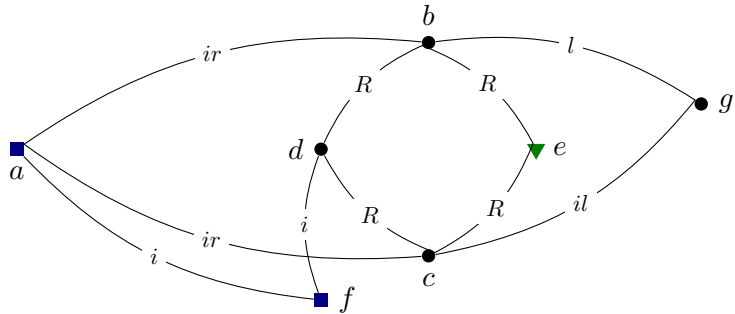
\label{partLineSect}
\begin{center}
\begin{tikzpicture}[scale=2]
\node[fixed, label=below:$a$] (v5) at (-2,0) {};
\node[vertex, label=left:$d$] (v3) at (0,0) {};
\node[vertex, label=above:$b$] (v2) at ({sqrt(2)/2},{sqrt(2)/2}) {};
\node[vertex, label=below:$c$] (v1) at ({sqrt(2)/2},{-sqrt(2)/2}) {};
\node[input, label=left:$e$] (v0) at ({sqrt(2)},0) {};
\node[vertex, label=right:$g$] (v6) at (2.5,0.3) {};
\node[fixed, label=right:$f$] (v7) at (0,-1) {};
\draw (v3) to[bend left=20] node[edge] {$R$} (v2) to[bend left=20]node[edge] {$R$}  (v0) to[bend left=20]node[edge] {$R$} (v1) to[bend left=20] node[edge] {$R$} (v3);
\draw (v2) to[bend right=20] node[edge] {$ir$} (v5) to[bend right=20] node[edge] {$ir$} (v1);
\draw (v2) to[bend left=20] node[edge] {$l$} (v6) to[bend left=20] node[edge] {$il$} (v1);
\draw (v3) to[bend right=20] node[edge] {$i$} (v7);
\draw (v7) to[bend left=20] node[edge] {$i$} (v5);
\end{tikzpicture}
\end{center}

$R = r = \frac{1}{\sqrt{2}}; l > 0; F = \{a, f\}; \phi_0(a) = \begin{pmatrix}5 \\ t_0 + 1/2 \end{pmatrix}, \phi_0(f) = \begin{pmatrix}5 \\ t_0 - 1/2 \end{pmatrix}; P = \{e\}$.

\begin{prop} The workspace of $e$, $\mathrm{Conf}_{\mathbb{M}}^P(\mathcal{L})$, is contained in the line $t = t_0$, but does not necessarily contain the whole line. More precisely
\[ \left\{\begin{pmatrix}x \\ t \end{pmatrix} \in \mathbb{M} ~ \middle\arrowvert ~ t = t_0, \lvert x - 5 \rvert > 1/2 \right\} \subseteq \mathrm{Reg}_\mathbb{M}^P(\mathcal{L}) = \mathrm{Conf}_{\mathbb{M}}^{P}(\mathcal{L}) \subseteq \left\{\begin{pmatrix}x \\ t \end{pmatrix} \in \mathbb{M} ~ \middle\arrowvert ~ t = t_0 \right\}. \]
\end{prop}
\begin{proof}

We apply Fact~\ref{confPeauc}, Propositions~\ref{regPeauc} and~\ref{equivPeauc}.
\end{proof}

For example, this linkage with the choice $t_0 = 0$ will be called the partial $(t_0 = 0)$-line linkage.

The dual of this linkage (see Section~\ref{sectDualLinkage}) is called the partial \emph{$x_0$-line linkage}.

\subsection{The \texorpdfstring{$t_0$}{t0}-integer linkage} \label{intLinkSect}

This linkage contains four vertices $a$, $b$, $c$, $d$ which are restricted to move on $\mathcal{I}$ (the $x$-axis) using a partial $t_0$-line linkages. More precisely, the $t_0$-integer linkage is obtained as the combination of the linkage on the figure below with is combined with four partial $t_0$-line linkages $\mathcal{L}_i, i = 1 \dots 4$, to form the $t_0$-integer linkage. The combination mappings $\beta_i$ send $a$, $b$, $c$ and $d$ respectively to the inputs $e_i$ of the linkages $\mathcal L_i$.

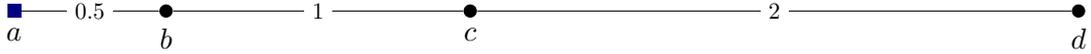
\begin{center}
\begin{tikzpicture}[scale=2]
\node[fixed, label=below:$a$] (v1) at (-3.5, 0) {};
\node[vertex, label=below:$b$] (v2) at (-2.5, 0) {};
\node[vertex, label=below:$c$] (v3) at (-0.5, 0) {};
\node[vertex, label=below:$d$] (v4) at (3.5, 0) {};
\draw (v1) to node[edge] {$0.5$} (v2) to node[edge] {$1$} (v3) to node[edge] {$2$} (v4);
\end{tikzpicture}
\end{center}

Take $F = \{a\}; \phi_0(a) = \begin{pmatrix}0.5\\t_0\end{pmatrix}; P = \emptyset$.

We have \[ \mathrm{Conf}_{\mathbb{M}}^{\{d\}}(\mathcal{L}) = \left\{\begin{pmatrix}-3\\t_0\end{pmatrix}, \begin{pmatrix}-2\\t_0\end{pmatrix}, \dots, \begin{pmatrix}3\\t_0\end{pmatrix}, \begin{pmatrix}4\\t_0\end{pmatrix}\right\}. \]

Moreover, $\mathrm{Conf}_{\mathbb{M}}(\mathcal{L})$ is a finite set so $\mathrm{Reg}_\mathbb{M}^P(\mathcal{L}) = \mathrm{Conf}_{\mathbb{M}}^P(\mathcal{L})$.

We will use this linkage twice to construct more complex linkages. In Section~\ref{sectLineLink}, we could have used a simpler linkage with a configuration space of cardinality $2$ instead of $8$, but we need it to have cardinality at least $7$ in Section~\ref{sectSquareLink}.

\subsection{The \texorpdfstring{$t_0$}{t0}-line linkage} \label{sectLineLink}
This linkage traces out the whole horizontal line $t = t_0$: it contains a vertex $e$, the input, such that \[\mathrm{Conf}_{\mathbb{M}}^{\{e\}}(\mathcal{L}) = \left\{ \alpha \in \mathbb{M} ~ \middle\arrowvert ~ t_\alpha = t_0 \right \}. \]

To construct it, the idea is to combine a partial $t_0$-line linkage with a $t_0$-integer linkage, as follows (to simplify the notations, we only give the construction of the $(t_0 = 0)$-line linkage). Let:

\begin{itemize}
\item $\mathcal{L}_1$ be a ($t_0 = \frac{1}{2}$)-integer linkage;
\item $\mathcal{L}_2$ a ($t_0 = -\frac{1}{2}$)-integer linkage;
\item $\mathcal{L}_3$ the combination (disjoint union) of the two linkages $\mathcal{L}_1$ and $\mathcal{L}_2$;
\item $\mathcal{L}_4$ a linkage similar to a partial $t_0$-line linkage, with the only difference that $F_4 = \emptyset$ instead of $F_4 = \{a_4, f_4\}$;
\item $W_3 = \{d_1, d_2\}$ and $\beta(d_1) = a_4$, $\beta(d_2) = f_4$;
\item $\mathcal{L}_5 = \mathcal{L}_3 \cup_\beta \mathcal{L}_4$. Since for any $x \in \mathbb R$, we have either $\abs{x-5} > 1/2$ or $\abs{x-7} > 1/2$, we obtain as desired
\[\mathrm{Conf}_{\mathbb{M}}^{\{e_5\}}(\mathcal{L}_5) = \left\{ \alpha \in \mathbb{M} ~ \middle\arrowvert ~ t_\alpha = t_0 \right \}. \]
\end{itemize}

Using Fact~\ref{propComb}, we also obtain $\mathrm{Reg}_\mathbb{M}^P(\mathcal{L}) = \mathrm{Conf}_{\mathbb{M}}^P(\mathcal{L})$.

For future reference, we let $a:= a_1$, $f:= f_1$, $e:= e_4$.

The dual of this linkage is called the \emph{$x_0$-line linkage}.

\subsection{The horizontal parallelizer} \label{horizParSect}
This linkage has the input set $P = \{e_3, e_4\}$. It satisfies
\[ \mathrm{Reg}_\mathbb{M}^P(\mathcal{L}) = \mathrm{Conf}_{\mathbb{M}}^{P}(\mathcal{L}) = \left\{ \psi \in \mathbb{M}^{\{e_3, e_4\}} ~ \middle\arrowvert ~ t_{\psi(e_3)} = t_{\psi(e_4)} \right\}. \]

Let: 
\begin{itemize}
\item $\mathcal{L}_1$ and $\mathcal{L}_2$ be two ($x_0 = 0$)-line linkages;
\item $\mathcal{L}_3$, $\mathcal{L}_4$ two linkages similar to ($t_0 = 0$)-line linkages, but with $F_3, F_4 = \emptyset$;
\item $\mathcal{L}_5$ the combination of $\mathcal{L}_1$ and $\mathcal{L}_2$;
\item $W_3 = \{a_3, f_3\}, \beta(a_3) = e_1, \beta(f_3) = e_2$, and $\mathcal{L}_6 = \mathcal{L}_3 \cup_\beta \mathcal{L}_5$;
\item $W_4 = \{a_4, f_4\}, \beta(a_4) = e_1, \beta(f_4) = e_2$, and $\mathcal{L}_7 = \mathcal{L}_4 \cup_\beta \mathcal{L}_6$.
\end{itemize}

$\mathcal{L}_7$ is the desired linkage.

For future reference, we let $a:= e_3$ and $b:= e_4$.

The dual of this linkage is called the \emph{vertical parallelizer}.

\subsection{The diagonal parallelizer} 
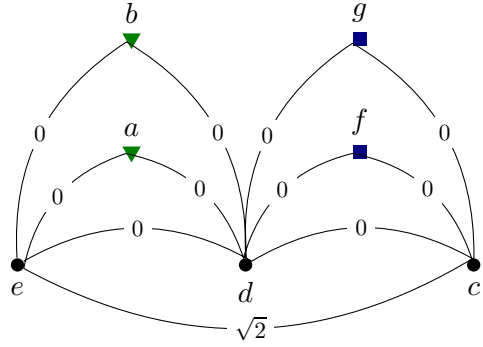
\label{diagParSect}

\begin{center}
\begin{tikzpicture}[scale=1.5]
\node[input, label=below:$a$] (a) at (-1,1) {};
\node[input, label=below:$b$] (b) at (-1,2) {};
\node[vertex, label=below:$c$] (c) at (2,0) {};
\node[vertex, label=below:$d$] (d) at (0,0) {};

\node[vertex, label=below:$e$] (e) at (-2,0) {};
\node[fixed, label=above:$f$] (f) at (1,1) {};
\node[fixed, label=above:$g$] (g) at (1,2) {};
\draw (e) to[bend left] node[edge] {$0$} (b) to[bend left] node[edge] {$0$} (d) to[bend left] node[edge] {$0$} (g) to[bend left] node[edge] {$0$} (c) to[bend left] node[edge] {$\sqrt{2}$} (e) to[bend left] node[edge] {$0$} (a) to[bend left] node[edge] {$0$} (d) to[bend left] node[edge] {$0$} (f) to[bend left] node[edge] {$0$} (c) to[bend right] node[edge] {$0$} (d) to[bend right] node[edge] {$0$} (e);
\end{tikzpicture}
\end{center}

$P = \{a, b\}, F = \{g, f\}, \phi_0(f) = \begin{pmatrix}1\\1\end{pmatrix}, \phi_0(g) = \begin{pmatrix}0\\0\end{pmatrix}$.

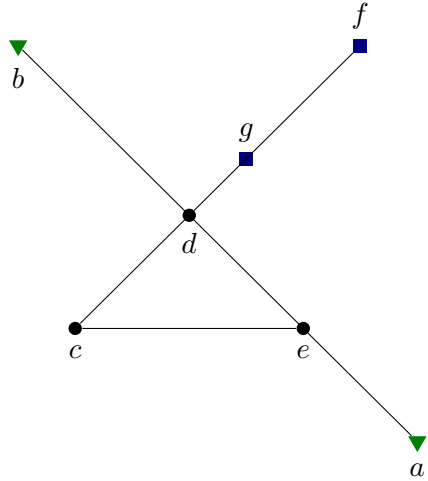
\begin{figure}[h!]
\begin{center}
\begin{tikzpicture}[scale=1.5]
\node[input, label=above:$a$] (a) at (1.5,-2.5) {};
\node[input, label=above:$b$] (b) at (-2,1) {};
\node[vertex, label=below:$e$] (c) at (0.5,-1.5) {};
\node[vertex, label=below:$d$] (d) at (-0.5,-0.5) {};
\node[vertex, label=below:$c$] (e) at (-1.5,-1.5) {};
\node[fixed, label=above:$f$] (f) at (1,1) {};
\node[fixed, label=above:$g$] (g) at (0,0) {};
\draw (f) -- (e) -- (c);
\draw (a) -- (b);
\end{tikzpicture}
\end{center}
\caption{One realization of the diagonal parallelizer.}
\end{figure}

In this section, we use the lightlike coordinates $y$ and $z$ (see Section~\ref{notations}).

\begin{prop} We have
\[ \mathrm{Reg}_\mathbb{M}^P(\mathcal{L}) = \mathrm{Conf}_{\mathbb{M}}^P(\mathcal{L}) = \left\{ \psi \in \mathbb{M}^P ~ \middle\arrowvert ~ y_{\psi(a)} = y_{\psi(b)} \right\}. \]
\end{prop}
\begin{proof}
The point is that for $\alpha_1, \alpha_2 \in \mathbb{M}$ such that $y_{\alpha_1} = y_{\alpha_2}$ and $\alpha_1 \neq \alpha_2$, the intersection $\mathcal{H}(\alpha_1, 0) \cap \mathcal{H}(\alpha_2, 0)$ is a straight line, more precisely:
\[ \mathcal{H}(\alpha_1, 0) \cap \mathcal{H}(\alpha_2, 0) = \left\{ \gamma ~ \middle\arrowvert ~ y_\gamma = y_{\alpha_1} \right\}. \]

First, let us prove the inclusion $\mathrm{Conf}_{\mathbb{M}}^P(\mathcal{L}) \subseteq \left\{ \psi \in \mathbb{M}^P ~ \middle\arrowvert ~ y_{\psi(a)} = y_{\psi(b)} \right\}$.

For all $\phi \in \mathrm{Conf}_{\mathbb{M}}(\mathcal{L})$, $\phi(c) \in \mathcal{H}(\phi(g), 0) \cap \mathcal{H}(\phi(f), 0)$ and $z_{\phi(f)}=z_{\phi(g)}$, so $z_{\phi(c)} = 0$. Likewise, $z_{\phi(d)} = 0$.

Since $\phi(e) \in \mathcal{H}(\phi(d), 0)$ and $\phi(e) \not\in \mathcal{H}(\phi(c), 0)$, we have $y_{\phi(d)} = y_{\phi(e)}$ and $\phi(d) \neq \phi(e)$.

Therefore, since $\phi(a) \in \mathcal{H}(\phi(d), 0) \cap \mathcal{H}(\phi(e), 0)$, we have $y_{\phi(a)} = y_{\phi(d)}$. Likewise, $y_{\phi(b)} = y_{\phi(d)}$ and finally, $y_{\phi(a)} = y_{\phi(b)}$.

Now, let us prove the inclusion $\mathrm{Conf}_{\mathbb{M}}^P(\mathcal{L}) \supseteq \left\{ \psi \in \mathbb{M}^P ~ \middle\arrowvert ~ y_{\psi(a)} = y_{\psi(b)} \right\}$. Let $\psi \in \mathbb{M}^{\{a, b\}}$ such that $y_{\psi(a)} = y_{\psi(b)}$. Construct $\phi \in \mathrm{Conf}_{\mathbb{M}}(\mathcal{L})$ such that $\phi|_{\{a, b\}} = \psi$. Let $\phi(d) \in \mathbb{M}$ such that $z_{\phi(d)} = 0$ and $y_{\phi(d)} = y_{\psi(a)}$. Let $\phi(e) = \phi(d) + \begin{pmatrix}1\\-1\end{pmatrix}$ and $\phi(c) = \phi(d) + \begin{pmatrix}-1\\-1\end{pmatrix}$ (in $(x, t)$ coordinates). Then $\phi \in \mathrm{Conf}_{\mathbb{M}}(\mathcal{L})$.

Finally, the coordinates of all vertices vary smoothly with respect to the coordinates of $a$ and $b$, so $\mathrm{Reg}_\mathbb{M}^P(\mathcal{L}) = \mathrm{Conf}_{\mathbb{M}}^P(\mathcal{L}).$
\end{proof}

\section{Elementary linkages for algebraic operations}

\subsection{The average function linkage} \label{symSect}
The average function linkage is a linkage with the input set $P = \{a, b\}$ and the output set $Q = \{c\}$ which is a functional linkage for the function
\[ \begin{aligned} f: \mathcal{I}^2 & \to \mathcal{I} \\ (x_1, x_2) & \longmapsto \frac{x_1+x_2}{2}, \end{aligned} \]
and such that $\mathrm{Reg}_\mathbb{M}^P(\mathcal{L}) = \mathrm{Conf}_{\mathbb{M}}^P(\mathcal{L}) = \mathcal{I}^P$.

Recall that by $\mathcal I$, we mean the $x$-axis, and by ``$\mathcal{L}$ is a functional linkage for $f$'', we mean that for all $\psi \in \mathrm{Conf}_{\mathbb{M}}^P(\mathcal{L})$
\[ x_{\psi(c)} = \frac{x_{\psi(a)}+x_{\psi(b)}}{2}. \]

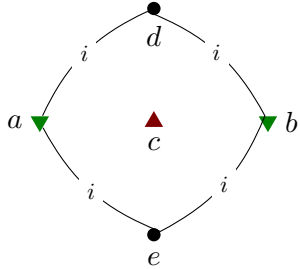
\begin{center}
\begin{tikzpicture}[scale=1.5]
\node[input, label=right:$a$] (a) at (-1,0) {};
\node[input, label=left:$b$] (b) at (1,0) {};
\node[output, label=below:$c$] (c) at (0,0) {};
\node[vertex, label=below:$d$] (d) at (0,1) {};
\node[vertex, label=below:$e$] (e) at (0,-1) {};
\draw (a) to[bend left=20] node[edge] {$i$} (d) to[bend left=20] node[edge] {$i$} (b) to[bend left=20] node[edge] {$i$} (e) to[bend left=20] node[edge] {$i$} (a);
\end{tikzpicture}
\end{center}

The vertices $a$, $b$, $c$ are restricted to move on the line $\mathcal{I}$ using ($t_0 = 0$)-line linkages: this means that the linkage in the figure above is combined with three ($t_0 = 0$)-line linkages. Likewise, the points $e$, $d$ and $c$ are restricted to have the same $x$ coordinate using a vertical parallelizer. The square $adbe$ is rigidified (thus, the actual average function linkage has much more than these $5$ vertices, but many of them are not represented on the figure).

To see that this linkage is the desired functional linkage, first notice that $\phi(c)$ is the middle of the segment $[\phi(a), \phi(b)]$ for all realization $\phi$, because $\phi(a) \phi(d) \phi(b) \phi(e)$ is a parallelogram. Moreover, the expression \[ t_{\phi(d)} = \pm \sqrt{1 + \frac{\abs{x_{\phi(b)} - x_{\phi(a)}}^2}{2}} \] shows that the coordinates of $\phi(d)$ (and similarly, $\phi(e)$) depend on $\phi(a)$ and $\phi(b)$ in a differentiable way, so $\mathrm{Reg}_\mathbb{M}^P(\mathcal{L}) = \mathcal{I}^P$.

\subsection{The adder} \label{adder}
The adder is a linkage with the input set $P = \{a_1, b_1\}$ and the output set $Q = \{b_2\}$ which is a functional linkage for the function
\[ \begin{aligned} f: \mathcal{I}^2 & \to \mathcal{I} \\ (x_1, x_2) & \longmapsto x_1+x_2, \end{aligned} \]
with $\mathrm{Reg}_\mathbb{M}^P(\mathcal{L}) = \mathrm{Conf}_{\mathbb{M}}^P(\mathcal{L}) = \mathcal{I}^P$.

It is constructed as $\mathcal{L}_1 \cup_\beta \mathcal{L}_2$, where $\mathcal L_1$ and $\mathcal L_2$ are average function linkages, with $W_1 = \{c_1\}$, $\beta(c_1) = c_2$, $F_2 = \{a_2\}$, and $\phi_{02}(a_2) = \begin{pmatrix}0\\0\end{pmatrix}$.

Note that we may obtain a functional linkage for substraction by letting $P = \{b_2, b_1\}$ and $Q = \{a_1\}$. One may also construct (by induction on $n$) a functional linkage for $x \mapsto nx$, where $n$ is any integer, and (by switching the input and the output) a functional linkage for $x \mapsto \frac{1}{n}x$.

\subsection{The square function linkage} \label{sectSquareLink}
The square function linkage is a linkage with the input set $P = \{a\}$ and the output set $Q = \{b\}$: it is functional for the function
\[ \begin{aligned} \mathcal{I} & \to \mathcal{I} \\ x & \longmapsto x^2, \end{aligned} \]
with $\mathrm{Reg}_\mathbb{M}^P(\mathcal{L}) = \mathrm{Conf}_{\mathbb{M}}^P(\mathcal{L}) = \mathcal{I}^P$.

To construct it, recall the algebraic trick described by Kapovich and Millson in~\cite{mk}:
\[ \forall x \in \mathbb{R} \setminus \{-0.5, 0.5\} ~~~ x^2 = 0.25 + \frac{1}{\frac{1}{x - 0.5} - \frac{1}{x + 0.5}}. \]

We have to find another trick to obtain a formula which works for every $x \in \mathbb{R}$.

To do this, notice that for all $x$ and $x'$ in $\mathbb{R}$ we have the identity
\[ x^2 = 2(x+x')^2 + 2(x')^2 - (x + 2x')^2. \]
Thus the expression $x^2$ can be rewritten
\footnotesize \begin{equation} \label{eqSquare} 2\left(0.25 + \frac{1}{\frac{1}{x+x' - 0.5} - \frac{1}{x+x' + 0.5}}\right) + 2\left(0.25 + \frac{1}{\frac{1}{x' - 0.5} - \frac{1}{x' + 0.5}}\right) - \left(0.25 + \frac{1}{\frac{1}{x + 2x' - 0.5} - \frac{1}{x + 2x' + 0.5}}\right). \end{equation}

\normalsize Moreover, for all $x \in \mathbb{R}$ there exists an $x' \in \{-3, -2, \dots, 3, 4\}$ such that \[\{x + x', x + 2x', x'\} \cap \{-0.5, 0.5\} = \emptyset.\]

Start with a ($t_0 = 0$)-integer linkage $\mathcal{L}_1$: think of the vertex $d_1$ as the number $x'$. Let $\mathcal{L}_2$ be the linkage $\mathcal{L}_1$ to which one adds new fixed vertices at $\begin{pmatrix}0.5 \\ 0\end{pmatrix}$ and $\begin{pmatrix}0.25 \\ 0\end{pmatrix}$, and a new mobile vertex which will represent $x$ and will be the input of the linkage (one does not add any new edge for now). Since Expression~\ref{eqSquare} is the composition of additions, subtractions and inversions, one may combine $\mathcal{L}_2$ with linkages for addition, subtraction and inversion (for the inversion, use the Peaucellier inversor), in the spirit of Fact~\ref{propCombFunc}, so that the output of the new linkage $\mathcal{L}$ corresponds to Expression~\ref{eqSquare}. This is the desired linkage.

\subsection{The multiplier} \label{multiplier}
The multiplier is a linkage with the input set $P = \{a, b\}$ and the output set $Q = \{c\}$ which is a functional linkage for the function
\[ \begin{aligned} f: \mathcal{I}^2 & \to \mathcal{I} \\ (x_1, x_2) & \longmapsto x_1x_2, \end{aligned} \]
such that $\mathrm{Reg}_\mathbb{M}^P(\mathcal{L}) = \mathrm{Conf}_{\mathbb{M}}^P(\mathcal{L}) = \mathcal{I}^P$.

Simply construct the multiplier by combining square function linkages and adders, using the identity
\[ \forall x_1, x_2 \in \mathbb{R} ~~~ x_1x_2 = \frac{1}{4}\left((x_1+x_2)^2 - (x_1-x_2)^2\right). \]

\subsection{The polynomial linkage} \label{polylink}
Let $f: \mathbb{R}^n \to \mathbb{R}^m$ be a polynomial. We identify $\mathbb{R}$ with $\mathcal{I}$.

The polynomial linkage is a functional linkage for the function $f$ with $\mathrm{card}(P)=n$ and
\[ \mathrm{Reg}_\mathbb{M}^P(\mathcal{L}) = \mathrm{Conf}_{\mathbb{M}}^P(\mathcal{L}) = \mathcal{I}^P. \]

The polynomial linkage is obtained by combining adders and multipliers (use Fact~\ref{propCombFunc}). The coefficients are represented by fixed vertices.

\smallskip \paragraph{\bf Example} To illustrate the general case, we give the following example: $n = 2$, $m = 1$, $f(x, y) = 2 x^3 y + \pi$.

To construct a functional linkage for $f$, start with a linkage $\mathcal{L}$ consisting of two fixed vertices $a, b$ with $\phi_0(a) = \begin{pmatrix}2\\0\end{pmatrix}$ and $\phi_0(b) = \begin{pmatrix}\pi\\0\end{pmatrix}$, but also two vertices $c, d$ which are the inputs and correspond respectively to the variables $x$ and $y$.

Combine this linkage with a multiplier: the combination mapping $\beta$ sends $c$ to one of the inputs of the multiplier and $d$ to the other one. The linkage (still called $\mathcal{L}$) is now functional for $(x, y) \longmapsto xy$.

Combine the new linkage with another multiplier: the combination mapping $\beta$ sends $c$ to one of the inputs and the output of $\mathcal{L}$ to the other one. The new linkage $\mathcal{L}$ is functional for $(x, y) \longmapsto x^2y$.

Repeating this process once, we obtain a functional linkage for $x^3y$, and then for $2 x^3 y$ (using the vertex $a$).

Finally, combine the linkage $\mathcal{L}$ with an adder: the combination mapping $\beta$ sends the output of $\mathcal{L}$ to one of the inputs, and $b$ to the other one.

\section{End of the proof of Theorem~\refkempemink}

Let $n \in \mathbb{N}$. We are given $A$ a semi-algebraic subset of $(\mathbb{R}^2)^n$, but we first assume that $A$ is in fact an \emph{algebraic} subset of $(\mathbb{R}^2)^n$, defined by a polynomial $f: \mathbb{R}^{2n} \to \mathbb{R}^m$ (so that $A = f^{-1}(0)$).

Take a polynomial linkage $\mathcal{L}$ for $f$. Name the elements of the input set: $P = \left\{a_1, \dots, a_{2n}\right\}$. The output set $Q$ has $2m$ elements.

The linkage $\mathcal{L}$ does not yet look like the desired linkage: since $\mathcal{L}$ has $2n$ inputs, the partial configuration space $\mathrm{Conf}_{\mathbb{M}}^P(\mathcal{L})$ is a subset of $(\mathbb{R}^2)^{2n}$ (in fact, it is a subset of $\mathcal{I}^{2n}$), while $A$ is a subset of $(\mathbb{R}^2)^n$ (in particular, we are looking for a linkage with $n$ inputs). To obtain $\mathrm{Conf}_{\mathbb{M}}^P(\mathcal{L}) = A \subseteq (\mathbb{R}^2)^n$, we have to modify $\mathcal{L}$ in the following way.

\begin{enumerate}
\item With several ($x_0 = 0$)-line linkages and diagonal parallelizers, extend the linkage $\mathcal{L}$ to a new one with new vertices $c_2, c_4, c_6, \dots, c_{2n}$ such that for all realization $\phi$ and for all $k \in \{1, \dots, n\}$
\[ x_{\phi(c_{2k})} = 0; \]
\[ y_{\phi(c_{2k})} = y_{\phi(a_{2k})} ~~~ \left(\text{\emph{i.e. }} x_{\phi(c_{2k})} + t_{\phi(c_{2k})} = x_{\phi(a_{2k})} + t_{\phi(a_{2k})} \right). \]
\item With several vertical and horizontal parallelizers, extend this linkage to a new one with vertices $d_2, d_4, d_6, \dots, d_{3n}$ such that for all realization $\phi$ and for all $k \in \{1, \dots, n\}$
\[ x_{\phi(d_{2k})} = x_{\phi(a_{2k-1})}; \]
\[ t_{\phi(d_{2k})} = t_{\phi(c_{2k})}. \]
\end{enumerate}

Thus, for all realization $\phi$ and all $k \in \{1, \dots, n\}$, we have $x_{\phi(d_{2k})} = x_{\phi(a_{2k-1})}$ and $t_{\phi(d_{2k})} = x_{\phi(a_{2k})}$.

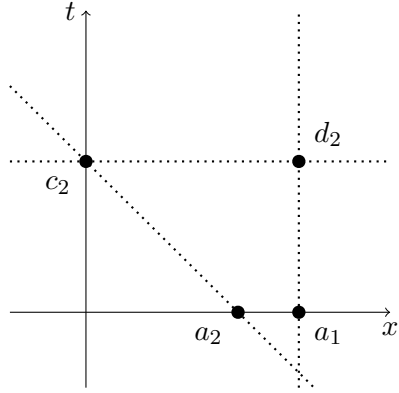
\begin{figure}[h!]
\begin{center}
\begin{tikzpicture}[scale=1]
\coordinate (x1) at (-1,0) {};
\coordinate[label=below:$x$] (x2) at (4,0) {};
\coordinate (y1) at (0,-1) {};
\coordinate[label=left:$t$] (y2) at (0,4) {};
\node[vertex, label=below:$a_1$, rotate=45] (a1) at (2.8,0) {};
\node[vertex, label=left:$a_2$, rotate=45] (a2) at (2,0) {};
\node[vertex, label=left:$c_2$, rotate=45] (c2) at (0,2) {};
\node[vertex, label=right:$d_2$, rotate=45] (d2) at (2.8,2) {};
\draw[dotted, thick] (-1,3) -- (3,-1);
\draw[dotted, thick] (-1,2) -- (4,2);
\draw[dotted, thick] (2.8,-1) -- (2.8,4);
\draw[->] (x1) -- (x2);
\draw[->] (y1) -- (y2);
\end{tikzpicture}
\end{center}
\caption{A partial realization of the four vertices $a_1, a_2, c_2, d_2$. We have $x_{\phi(d_{2})} = x_{\phi(a_{1})}$ and $t_{\phi(d_{2})} = x_{\phi(a_{2})}$.}
\end{figure}

Let $P = \{d_2, d_4, \dots, d_{2n}\}$. Note that the input map $p$ is a finite covering onto the simply connected set $(\mathbb R^2)^n$. Therefore, $p$ is a trivial covering. The output set $Q$ is unchanged.

Fix the outputs to the origin: precisely, replace $F$ by $F \cup Q$ and let \[ \forall a \in Q ~~~ \phi_0(a) = \begin{pmatrix}0\\0\end{pmatrix}. \]

We obtain as desired $\mathrm{Reg}_\mathbb{M}^P(\mathcal{L}) = \mathrm{Conf}_{\mathbb{M}}^P(\mathcal{L}) = A$.

Finally, if $A$ is any \emph{semi-algebraic} set of $(\mathbb{R}^2)^n$, then $A$ is the projection of an algebraic set $B$ of $(\mathbb{R}^2)^N$ for some $N \geq n$. Construct the linkage $\mathcal{L}_1$ such that $\mathrm{Conf}_{\mathbb{M}}^{P_1}(\mathcal{L}_1) = B$ and remove the unnecessary inputs $d_{2n+2}, d_{2n+4} \ldots, d_{2N}$. Then
\[ \mathrm{Conf}_{\mathbb{M}}^P(\mathcal{L}) = A, \]
which ends the proof Theorem~\refkempemink.

\chapter{Linkages in the hyperbolic plane} \label{chapHyp}

The aim of this chapter is to prove Theorem~\ref{kempe-hyp}.

\section{Generalities on the hyperbolic plane}

\begin{definition} \label{defHalfPlane} The \emph{Poincaré half-plane model} is the half-plane:
\[ \left\{ \begin{pmatrix}x\\y\end{pmatrix} \in \mathbb{R}^2 \middle\arrowvert y > 0 \right\} \]
endowed with the metric:
\[ \frac{(dx)^2 + (dy)^2}{y^2}. \]
This model is the one we will always use in this chapter.

The distance $\delta$ on $\mathbb{H}^2$ is given by the formula:
\[ \delta\left(\begin{pmatrix}x_1\\y_1\end{pmatrix}, \begin{pmatrix}x_2\\y_2\end{pmatrix}\right) = \arcosh \left( 1 + \frac{ {(x_2 - x_1)}^2 + {(y_2 - y_1)}^2 }{ 2 y_1 y_2 } \right). \]
\end{definition}

\subsection{Circles}
In the Poincaré half-plane model, a hyperbolic circle with hyperbolic center $\alpha$ and hyperbolic radius $R$ is in fact a Euclidean circle with center $\beta$ and radius $r$, where $y_\beta = y_\alpha \cosh R$, $x_\beta = x_\alpha$ and $r = y_\alpha \sinh R$. Also notice that $y_\beta = \sqrt{y_\alpha^2 + r^2}$.

\subsection{Some compact subsets of $\mathbb H^2$} \label{somecompact}
Since we work with linkages with compact configuration spaces, whereas $\mathbb H^2$ is not compact, we need to introduce some compact subsets on which we will use our linkages. 
Fix a real parameter $\eta > 1$, and think of it as a very large number (the precise meaning of ``large'' will be given later).

Let
\[ \mathcal{I}_0:= \left\{\alpha \in \mathbb{H}^2 ~ \middle\arrowvert ~ y_\alpha = 2, \lvert x_\alpha \rvert \leq 10 \eta \right\}. \]
\[ \mathcal{J}_0:= \left\{\alpha \in \mathbb{H}^2 ~ \middle\arrowvert ~ x_\alpha = 0, ~ 2e^{-10 \eta} \leq y_\alpha \leq 2e^{10 \eta} \right\}. \]
\[ \mathcal{B}_0:= \left\{\alpha \in \mathbb{H}^2 ~ \middle\arrowvert ~ \lvert x_\alpha \rvert \leq 10 \eta, ~ 2e^{-10 \eta} \leq y_\alpha \leq 2e^{10 \eta} \right\}. \]

For any segment of positive length $I$, we write $\hat{I}$ the line containing $I$. For example:
\[ \hat{\mathcal{I}_0}:= \left\{\alpha \in \mathbb{H}^2 ~ \middle\arrowvert ~ y_\alpha = 2 \right\}. \]

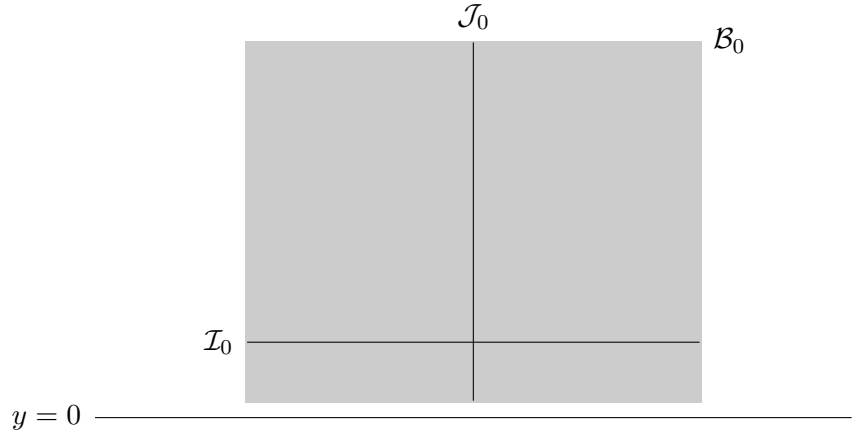
\begin{center}
\begin{tikzpicture}
\node[label=left:{$y = 0$}, inner sep=0] (a1) at (-5, 0) {};
\node[label=left:{$\mathcal I_0$}, inner sep=0] (a3) at (-3, 1) {};
\node[label=above:{$\mathcal J_0$}, inner sep=0] (a5) at (0, 5) {};
\node[label=right:{$\mathcal B_0$}, inner sep=0] (a7) at (3, 5) {};
\path[fill=black!20] (-3,0.2) rectangle (3, 5) {};
\node[inner sep=0] (a6) at (0, 0.2) {};
\node[inner sep=0] (a4) at (3, 1) {};
\node[inner sep=0] (a2) at (5, 0) {};
\draw (a1) -- (a2);
\draw (a3) -- (a4);
\draw (a5) -- (a6);
\end{tikzpicture}
\end{center}

\section{Elementary linkages for geometric operations}

\subsection{The circle linkage}

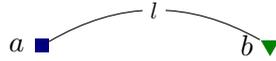
\begin{center}
\shorthandoff{:}
\begin{tikzpicture}[scale=1.5]
\node[fixed, label=left:$a$] (a) at (-1,0) {};
\node[input, label=right:$b$] (b) at (1,0) {};
\draw (a) to[bend left] node[edge] {$l$} (b);
\end{tikzpicture}
\shorthandon{:}
\end{center}

We let $F = \{a\}$ and $P = \{b\}$ (see Section~\ref{notations} for the notations).

In this linkage $\mathrm{Conf}_{\mathbb{H}^2}^P(\mathcal{L})$ is a hyperbolic circle, which is also a Euclidean circle. Conversely, if $\mathcal{C}$ is a Euclidean circle contained in the half-plane, it is also a hyperbolic circle, so there is a circle linkage such that $\mathrm{Conf}_{\mathbb{H}^2}^P(\mathcal{L}) = \mathcal{C}$. Moreover, $\mathrm{Reg}_{\mathbb{H}^2}^P(\mathcal{L}) = \mathrm{Conf}_{\mathbb{H}^2}^P(\mathcal{L})$.

\subsection{The robotic arm linkage}

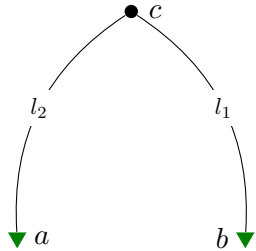
\begin{center}
\shorthandoff{:}
\begin{tikzpicture}[scale=1.5]
\node[input, label=left:$a$] (a) at (-1,0) {};
\node[input, label=right:$b$] (b) at (1,0) {};
\node[vertex, label=right:$c$] (c) at (0,2) {};
\draw (c) to[bend right] node[edge] {$l_2$} (a);
\draw (c) to[bend left] node[edge] {$l_1$} (b);
\end{tikzpicture}
\shorthandon{:}
\end{center}

We let $P = \{a, b\}, l_1 > 0, l_2 > 0$. We have:
\[ \mathrm{Conf}_{\mathbb{H}^2}^P(\mathcal{L}) = \left\{ \psi \in (\mathbb{H}^2)^P ~ \middle\arrowvert ~ \lvert l_1 - l_2 \rvert \leq \delta(\psi(a), \psi(b)) \leq l_1 + l_2 \right\} \]
and $\mathrm{Reg}_{\mathbb{H}^2}^P(\mathcal{L})$ contains:
\[ \left\{ \psi \in \mathrm{Conf}_{\mathbb{H}^2}^P(\mathcal{L}) ~ \middle\arrowvert ~ \lvert l_1 - l_2 \rvert < \delta(\psi(a), \psi(b)) < l_1 + l_2 \right\} \]
(recall that $\delta$ is the hyperbolic distance on $\mathbb H^2$).

\subsection{The Peaucellier inversor}

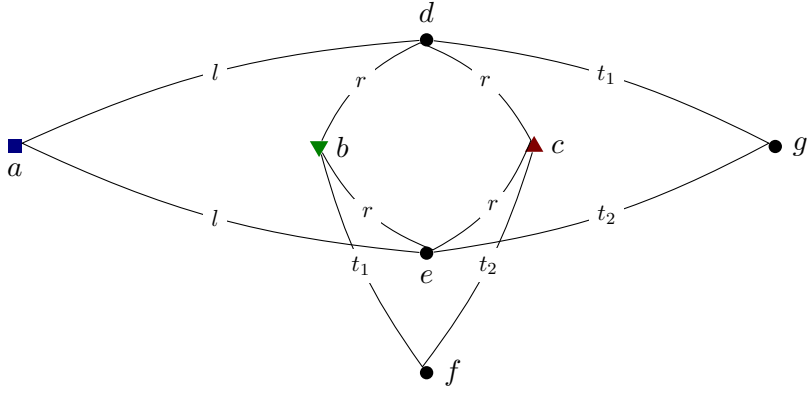
\begin{center}
\shorthandoff{:}
\begin{tikzpicture}[scale=2]
\node[fixed, label=below:$a$] (v5) at (-2,0) {};
\node[input, label=left:$b$] (v3) at (0,0) {};
\node[vertex, label=above:$d$] (v2) at ({sqrt(2)/2},{sqrt(2)/2}) {};
\node[vertex, label=below:$e$] (v1) at ({sqrt(2)/2},{-sqrt(2)/2}) {};
\node[output, label=right:$c$] (v0) at ({sqrt(2)},0) {};
\node[vertex, label=right:$f$] (v6) at ({sqrt(2)/2},-1.5) {};
\node[vertex, label=right:$g$] (v7) at (3,0) {};

\draw (v3) to[bend left=20] node[edge] {$r$} (v2) to[bend left=20] node[edge] {$r$} (v0) to[bend left=20] node[edge] {$r$} (v1) to[bend left=20] node[edge] {$r$} (v3);
\draw (v2) to[bend right=10] node[edge] {$l$} (v5) to[bend right=10] node[edge] {$l$} (v1);
\draw (v3) to[bend right=10] node[edge] {$t_1$} (v6) to[bend right=10] node[edge] {$t_2$} (v0);
\draw (v2) to[bend left=10] node[edge] {$t_1$} (v7) to[bend left=10] node[edge] {$t_2$} (v1);
\end{tikzpicture}
\shorthandon{:}
\end{center}

We let $F = \{a\}, P = \{b\}, Q = \{c\}$. We require $l \neq r, t_1 \neq t_2, t_1 > r, t_2 > r$.

\begin{prop}
This linkage is functional for the (Euclidean) inversion with respect to the circle $\mathcal{C}$ with hyperbolic center $\phi_0(a)$ and hyperbolic radius $\arcosh \frac{\cosh l}{\cosh r}$.
\end{prop}
\begin{proof}
Let $\phi \in \mathrm{Conf}_{\mathbb{H}^2}(\mathcal{L})$. Let $\mu$ be the middle of the hyperbolic segment $[\phi(d)\phi(e)]$.

\smallskip {\bf First case.} In this case, we assume $\phi_0(a) = \begin{pmatrix}0\\1\end{pmatrix}$, $x_\mu = 0$ and $y_\mu \leq 1$.

$\phi(d)$ and $\phi(e)$ have two possible values each, and $\phi(d) \neq \phi(e)$ because $t_1 \neq t_2$. By symmetry, $\mu$ is also the middle of the hyperbolic segment $[\phi(b), \phi(c)]$ and $x_{\phi(b)} = x_{\phi(c)} = 0$. If necessary, we exchange $b$ and $c$ so that $y_{\phi(c)} \geq y_{\phi(b)}$.

Let $\alpha$ be the Euclidean center of $\mathcal{C}$. We have
\[ y_\alpha = y_{\phi(a)} \cosh \arcosh \frac{\cosh l}{\cosh r} = \frac{\cosh l}{\cosh r}. \]

From the hyperbolic Pythagorean Theorem applied to the hyperbolic triangles $(\mu\phi(a)\phi(d))$, $(\mu\phi(b)\phi(d))$ and $(\mu\phi(c)\phi(d))$, letting $D = \delta(\phi(d), \mu)$, we get:
\[ \delta(\phi(a), \mu) = \arcosh \frac{\cosh l}{\cosh D} \]
\[ \delta(\phi(c), \mu) = \delta(\phi(b), \mu) = \arcosh \frac{\cosh r}{\cosh D} \]

We may now compute the coordinates of $\phi(b)$ and $\phi(c)$:
\[ y_{\phi(b)} = \frac{\exp \arcosh \frac{\cosh r}{\cosh D}}{\exp \arcosh \frac{\cosh l}{\cosh D}} \]
\[ y_{\phi(b)} = \frac{\cosh r + \sqrt{(\cosh r)^2 - (\cosh D)^2}}{\cosh l + \sqrt{(\cosh l)^2 - (\cosh D)^2}} \]

\[ y_{\phi(c)} = \frac{1}{\left(\exp \arcosh \frac{\cosh r}{\cosh D}\right)\left(\exp \arcosh \frac{\cosh l}{\cosh D}\right)} \]
\[ y_{\phi(c)} = \frac{(\cosh D)^2}{\left(\cosh r + \sqrt{(\cosh r)^2 - (\cosh D)^2}\right)\left(\cosh l + \sqrt{(\cosh l)^2 - (\cosh D)^2}\right)}. \]

Finally, we obtain as desired:
\[ (y_\alpha - y_{\phi(b)})(y_\alpha - y_{\phi(c)}) = \frac{(\cosh l)^2}{(\cosh r)^2} - 1. \]

\smallskip {\bf General case.}
Let $\Phi: \mathbb{H}^2 \mapsto \mathbb{H}^2$ be an isometry such that $\Phi(\phi(a)) = \begin{pmatrix}0\\1\end{pmatrix}$, $x_{\Phi(\mu)} = 0$ and $y_{\Phi(\mu)} \leq 1$. Let $i$ be the inversion with respect to the circle with hyperbolic center $\begin{pmatrix}0\\1\end{pmatrix}$ and hyperbolic radius $\arcosh \frac{\cosh l}{\cosh r}$. Then $\phi(c) = \Phi^{-1} \circ i \circ \Phi (\phi(b))$, and $\Phi^{-1} \circ i \circ \Phi$ is the inversion with respect to the circle $\mathcal{C}$.
\end{proof}

We now study the workspace of the input $b$. Obviously, the input cannot be in the image of the lower half-plane by the inversion, because the output has to remain in the upper half-plane. Moreover, since the two edges $(bf)$ and $(bc)$ have different lengths, the input cannot be a fixed point of the inversion.

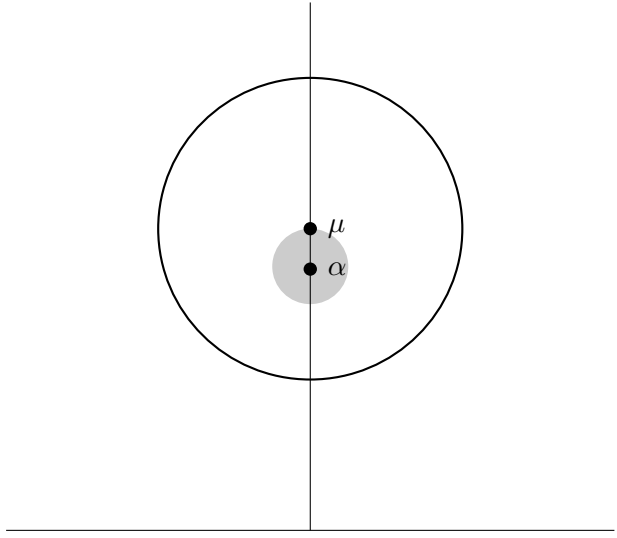
\begin{figure}[!ht]
\begin{center}
\shorthandoff{:}
\begin{tikzpicture}[scale=2]
\coordinate[] (c2) at (0,{sqrt(3)}) {};
\coordinate[] (c4) at (0,{7/4}) {};
\coordinate[] (c3) at (0,2) {};
\coordinate[] (a1) at (-2, 0) {};
\coordinate[] (a4) at (2, 0) {};
\coordinate[] (a5) at (0, 0) {};
\coordinate[] (a6) at (0, 3.5) {};
\draw[black, thick] (c3) circle (1);
\path[fill=black!20] (c4) circle ({1/4});
\node[vertex, label=right:$\alpha$] (a2) at (c2) {};
\node[vertex, label=right:$\mu$] (a3) at (c3) {};
\draw (a1) -- (a4);
\draw (a5) -- (a6);
\end{tikzpicture}
\shorthandon{:}
\end{center}
\caption{The set $K$ must not intersect the gray disk $\mathcal{D}$ or the black circle $\mathcal{C}$.}
\end{figure}

The following proposition tells us that these two obstructions are essentially the only ones.

\begin{prop} \label{domPeauc}
Let $\mathcal{C}$ be a circle of hyperbolic center $\alpha \in \mathbb{H}^2$, hyperbolic radius $R > 0$, Euclidean center $\mu$ and Euclidean radius $r$. Let $\mathcal D$ be the closed disk of hyperbolic center $\alpha$ and hyperbolic radius $Q = \delta(\alpha, \mu)$.

Let $K$ be a compact set in $\mathbb{H}^2 \setminus (\mathcal{C} \cup \mathcal{D})$. Then there exists a choice of $\phi_0(a), l, r, t_1, t_2$ such that the Peaucellier linkage with these lengths is functional for the inversion with respect to $\mathcal{C}$, and such that $K \subseteq \mathrm{Reg}_{\mathbb{H}^2}^P(\mathcal{L}) \subseteq \mathrm{Conf}_{\mathbb{H}^2}^P(\mathcal{L})$.
\end{prop}

\begin{proof}
Let $K'$ be a compact set in $\mathbb{H}^2 \setminus (\mathcal{C} \cup \mathcal{D})$ such that $K \subseteq \overset{\circ}{K'}$ (the interior of $K'$). Let $Q'$, such that $Q < Q' < R$ and $K' \subseteq \mathbb{H}^2 \setminus (\mathcal{C} \cup \mathcal{D}')$, where $\mathcal{D}'$ is the closed disk of hyperbolic center $\alpha$ and hyperbolic radius $Q'$.

Let $Q'' > 0$ such that $K' \subseteq \mathcal{D}''$, where $\mathcal{D}''$ is the open disk of hyperbolic center $\alpha$ and hyperbolic radius $Q''$.

Define
\[ \begin{aligned}
\Phi: \mathbb{R}_{\geq 0} & \to \mathbb{R}
\\ u & \to \arcosh ((\cosh u)(\cosh R)) - u.
\end{aligned} \]

Since $\lim_{u \rightarrow +\infty} \Phi(u) = \log (\cosh R) = Q$, there exists $u_0 \geq Q''$ such that $\Phi(u_0) \leq Q'$.

Let $\phi_0(a) = \alpha$, $r = u_0$ and $l = \Phi(u_0) + u_0$. Then, $R = \arcosh \frac{\cosh l}{\cosh r}$, $l - r \leq Q'$ and $l + r \geq Q''$.

Finally, choose $t_1$ and $t_2$ close enough to each other to have $K' \subseteq \mathrm{Conf}_{\mathbb{H}^2}^P(\mathcal{L})$. Then $K \subseteq \overset{\circ}{K'} \subseteq \mathrm{Reg}_{\mathbb{H}^2}^P(\mathcal{L})$.
\end{proof}

\subsection{The Euclidean line linkage}

The aim of this linkage is to trace out any given Euclidean segment. More precisely, let $\Delta$ be a straight line and $I \subseteq \Delta$ a Euclidean segment: we construct a linkage $\mathcal{L}$ with one input such that $I \subseteq \mathrm{Reg}_{\mathbb{H}^2}^P(\mathcal{L}) \subseteq \mathrm{Conf}_{\mathbb{H}^2}^P(\mathcal{L}) \subseteq \Delta$.

Let $\alpha \in \mathbb{R}^2$, $k_1, k_2 > 0$ such that:
\begin{enumerate}
\item $\Delta$ is outside the closed disk with hyperbolic center $\alpha$ and hyperbolic radius $k_1$;
\item $I$ is contained in the open disk with hyperbolic center $\alpha$ and hyperbolic radius $k_2$.
\end{enumerate}

Let $l$ and $r$ such that $l + r = k_2$. Choosing $l$ and $r$ sufficiently close to $\frac{k_2}{2}$, we may also require $\arcosh \frac{\cosh l}{\cosh r} \leq k_1$. From Proposition~\ref{domPeauc}, we deduce that there is a Peaucellier linkage $\mathcal{L}_1$ such that $I \subseteq \mathrm{Reg}_{\mathbb{H}^2}^P(\mathcal{L}) \subseteq \mathrm{Conf}_{\mathbb{H}^2}^P(\mathcal{L})$. 

Let $i$ be the inversion for which $\mathcal{L}$ is functional. Then $i(\Delta)$ is a circle contained in the half-plane (in which one point has been removed). Let $\mathcal{L}_2$ be a circle linkage for this circle. Let $W_1 = \{c_1\}$ and $\beta(c_1) = b_2$. Construct the combination $\mathcal{L}_3 = \mathcal{L}_1 \cup_\beta \mathcal{L}_2$. This linkage has the desired properties.

Rename the input: $b:= b_1$, and the fixed vertices: $a:= a_1$, $c:= a_2$.

We also add an edge between the two fixed vertices $a$ and $c$, of length $\delta(\phi_{03}(a), \phi_{03}(c))$. This new edge will be useful for Sections~\ref{vertPara} and~\ref{hypalign}.

\subsection{The vertical parallelizer} \label{vertPara}

The aim of this linkage is to force two vertices to have the same $x$ coordinate. More precisely, it has two inputs $a$ and $b$ with:
\small \[ \left\{\psi \in (\mathbb{H}^2)^P ~ \middle\arrowvert ~ x_{\psi(a)} = x_{\psi(b)} \right\} \cap (\mathcal{B}_0)^P \subseteq \mathrm{Reg}_{\mathbb{H}^2}^P(\mathcal{L}) \subseteq \mathrm{Conf}_{\mathbb{H}^2}^P(\mathcal{L}) \subseteq \left\{\psi \in (\mathbb{H}^2)^P ~ \middle\arrowvert ~ x_{\psi(a)} = x_{\psi(b)} \right\} \]
\normalsize (see the notations of Section~\ref{somecompact}).

To construct it, the idea is to allow two vertical Euclidean line linkages to move together horizontally.

Let $\mathcal{L}_1$ and $\mathcal{L}_2$ be two identical line linkages for the segment $\mathcal{J}_0$. Construct their disjoint union $\mathcal{L}_3 = \mathcal{L}_1 \cup \mathcal{L}_2$.

Let $I$ be a horizontal segment of Euclidean length $20 \eta$ centered at $\phi_{01}(a_1)$, and $I'$ a horizontal segment of Euclidean length $20 \eta$ centered at $\phi_{01}(c_1)$. Let $\mathcal{L}_4$ and $\mathcal{L}_5$ be line linkages for $I$ and $I'$ respectively, and construct their disjoint union $\mathcal{L}_6 = \mathcal{L}_4 \cup \mathcal{L}_5$.

Change $F_3$ to $\emptyset$, let $W_3 = \{a_1, c_1, a_2, c_2\}$, $\beta(a_1) = \beta(a_2) = b_4$, $\beta(c_1) = \beta(c_2) = b_5$ and $\mathcal{L}_7 = \mathcal{L}_3 \cup_\beta \mathcal{L}_6$.

Rename the inputs: $a:= b_1, b:=b_2$.

This linkage has the desired properties.

\subsection{The hyperbolic alignment linkage} \label{hypalign}

The \emph{hyperbolic alignment linkage} and the \emph{equidistance linkage} are not needed to prove the differential universality (Theorem~\ref{univ-hyp}), but we will use them to prove the algebraic universality (Theorem~\ref{kempe-hyp}).

The \emph{hyperbolic alignment linkage} forces its three inputs $a$, $b$, $c$ to be on the same hyperbolic line. Fix a real constant $l > 0$, then define
\[ A = \left\{\psi \in (\mathbb{H}^2)^P ~ \middle\arrowvert ~ \psi(a), \psi(b), \psi(c) \text{ are on the same hyperbolic line} \right\} \]
and
\[ B = \left\{\psi \in (\mathbb{H}^2)^P ~ \middle\arrowvert ~ 0 < \max (\delta(\psi(a), \psi(b)), \delta(\psi(b), \psi(c)), \delta(\psi(a), \psi(c))) \leq l \right\}. \]

We want to construct a linkage such that
\[ A \cap B \subseteq \mathrm{Reg}_{\mathbb{H}^2}^P(\mathcal{L}) \subseteq \mathrm{Conf}_{\mathbb{H}^2}^P(\mathcal{L}) \subseteq A. \]

Take a vertical hyperbolic segment $J$ of hyperbolic length $l$. Note that $J$ is also a Euclidean segment, and that any hyperbolic segment of length $l$ is the image of $J$ by a global isometry of $\mathbb{H}^2$.

Take three identical Euclidean line linkages $\mathcal{L}_1$, $\mathcal{L}_2$, $\mathcal{L}_3$ for $J$, with no fixed vertices ($F_1 = F_2 = F_3 = \emptyset$), and glue these three linkages together: let $W_1 = \{ a_1, c_1 \}$, $\beta(a_1) = a_2, \beta(c_1) = c_2$, and $\mathcal{L}_4 = \mathcal{L}_1 \cup_\beta \mathcal{L}_2$. Next, let $W_4 = \{a_2, c_2\}$, $\beta(a_2) = a_3, \beta(c_2) = c_3$, and $\mathcal{L}_5 = \mathcal{L}_4 \cup_\beta \mathcal{L}_3$. Rename the inputs: $a := b_1$, $b := b_2$, $c := b_3$. Since the isometries of $\mathbb H^2$ send the vertical line to other hyperbolic lines, the vertices $a$, $b$ and $c$ are always on the same hyperbolic line and $\mathcal L_5$ is the desired linkage.

\subsection{The equidistance linkage}

The equidistance linkage forces an input $a$ to be equidistant from the two other inputs $d$ and $e$.

Fix two real constants $k_1 > 0$, $k_2 > 0$ and define
\[ A = \left\{\psi \in (\mathbb{H}^2)^P ~ \middle\arrowvert ~ \delta(\psi(a), \psi(d)) = \delta(\psi(a), \psi(e)) \right\} \]
and
\[ B = \left\{\psi \in (\mathbb{H}^2)^P ~ \middle\arrowvert ~ \delta(\psi(d), \psi(e)) \geq k_1, \delta(\psi(a), \psi(d)) \leq k_2 \right\}. \]

We want to construct a linkage $\mathcal L$ such that
\[ A \cap B \subseteq \mathrm{Reg}_{\mathbb{H}^2}^P(\mathcal{L}) \subseteq \mathrm{Conf}_{\mathbb{H}^2}^P(\mathcal{L}) \subseteq A. \]

Start with the following linkage:

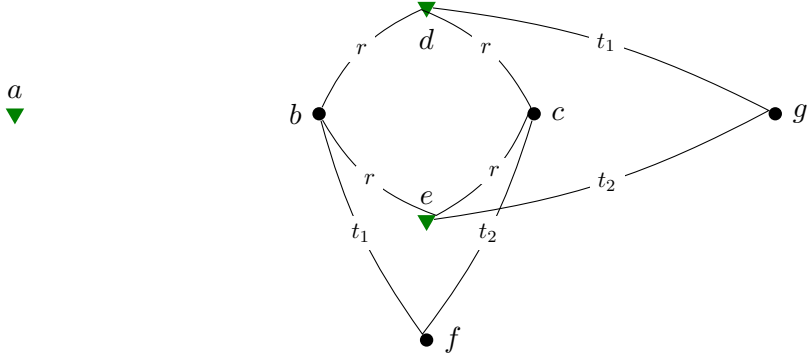
\begin{center}
\shorthandoff{:}
\begin{tikzpicture}[scale=2]
\node[input, label=below:$a$] (v5) at (-2,0) {};
\node[vertex, label=left:$b$] (v3) at (0,0) {};
\node[input, label=above:$d$] (v2) at ({sqrt(2)/2},{sqrt(2)/2}) {};
\node[input, label=below:$e$] (v1) at ({sqrt(2)/2},{-sqrt(2)/2}) {};
\node[vertex, label=right:$c$] (v0) at ({sqrt(2)},0) {};
\node[vertex, label=right:$f$] (v6) at ({sqrt(2)/2},-1.5) {};
\node[vertex, label=right:$g$] (v7) at (3,0) {};

\draw (v3) to[bend left=20] node[edge] {$r$} (v2) to[bend left=20] node[edge] {$r$} (v0) to[bend left=20] node[edge] {$r$} (v1) to[bend left=20] node[edge] {$r$} (v3);
\draw (v3) to[bend right=10] node[edge] {$t_1$} (v6) to[bend right=10] node[edge] {$t_2$} (v0);
\draw (v2) to[bend left=10] node[edge] {$t_1$} (v7) to[bend left=10] node[edge] {$t_2$} (v1);
\end{tikzpicture}
\shorthandon{:}
\end{center}

Let $P = \{a, d, e\}$, $F = \emptyset$ and $r > k_2$.

Using a hyperbolic alignement linkage with parameter $l = k_2 + 2r$, force the three vertices $a$, $b$ and $c$ to move on the same hyperbolic line. In other words, combine the linkage on the figure above with a hyperbolic alignment linkage (the combination mapping sends $a, b, c$ to the inputs of the hyperbolic alignment linkage).

Finally, choose $t_1, t_2 > r$ with $t_1 \neq t_2$ and $\lvert t_1 - t_2\rvert$ sufficiently small.

\section{Elementary linkages for algebraic operations}

In this section, we describe linkages which are functional for algebraic operations such as addition or multiplication on real numbers. The real line is identified with $\hat{\mathcal{I}_0} = \left\{\alpha \in \mathbb{H}^2 ~ \middle\arrowvert ~ y_\alpha = 2 \right\}$, which means that we will write simply $x$ instead of $(x, 2)$.

\subsection{The symmetrizer}
\subsubsection{First version}

The symmetrizer is a functional linkage for:
\[ \begin{aligned} f: \mathrm{Conf}_{\mathbb{H}^2}^P(\mathcal{L}) (\subseteq \hat{\mathcal{I}_0}^P) & \to \hat{\mathcal{I}_0} \\ (x_1, x_2) & \longmapsto \frac{x_1 + x_2}{2}. \end{aligned} \]

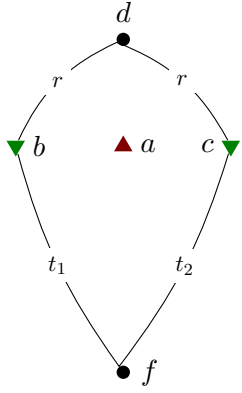
\begin{center}
\shorthandoff{:}
\begin{tikzpicture}[scale=2]
\node[input, label=left:$b$] (v3) at (0,0) {};
\node[vertex, label=above:$d$] (v2) at ({sqrt(2)/2},{sqrt(2)/2}) {};
\node[input, label=right:$c$] (v0) at ({sqrt(2)},0) {};
\node[vertex, label=right:$f$] (v6) at ({sqrt(2)/2},-1.5) {};
\node[output, label=right:$a$] (v4) at ({sqrt(2)/2},0) {};

\draw (v3) to[bend left=20] node[edge] {$r$} (v2) to[bend left=20] node[edge] {$r$} (v0);
\draw (v3) to[bend right=10] node[edge] {$t_1$} (v6) to[bend right=10] node[edge] {$t_2$} (v0);
\end{tikzpicture}
\shorthandon{:}
\end{center}

Let $P = \{b, c\}, Q = \{a\}$, $t_1 > r$, $t_2 > r$, $\lvert t_1 - t_2 \rvert = 1$, $r = 8 \eta$. The vertices $b$ and $c$ are restricted to move on $\mathcal{I}_0$ using two line linkages: this means that the linkage on the figure above is combined with two line linkages, which are not represented on the figure, with a combination map $\beta$ such that $\beta(b)$ and $\beta(c)$ are the two inputs of the line linkages. The vertices $a$ and $d$ are restricted to move on the same vertical line using a vertical parallelizer.

For this linkage,
\[ \mathrm{Conf}_{\mathbb{H}^2}^P(\mathcal{L}) = \left\{ \psi \in \hat{\mathcal{I}_0}^P ~ \middle\arrowvert ~ 1 \leq \lvert \psi(b) - \psi(c) \rvert \leq 16 \eta \right\} \]
and $\mathrm{Reg}_{\mathbb{H}^2}^P(\mathcal{L})$ contains
\[ \left\{ \psi \in \hat{\mathcal{I}_0}^P ~ \middle\arrowvert ~ 1 < \lvert \psi(b) - \psi(c) \rvert < 16 \eta \right\}. \]

\subsubsection{Modifying $\mathrm{Conf}_{\mathbb{H}^2}^P(\mathcal{L})$}
Since we want the symmetrizer to handle input vertices $b$ and $c$ which are close to each other, or even equal, the first version is not suitable for our purpose.
Notice that \[ \frac{x_1 + x_2}{2} = \frac{\frac{x_1 + 8 \eta}{2} + \frac{x_2 + (-8 \eta)}{2}}{2}. \]
Following this formula and the idea of Fact~\ref{propCombFunc}, take one symmetrizer $\mathcal{L}_1$, but add one input to the set of fixed vertices and set it to the coordinate $\begin{pmatrix}8 \eta \\ 2 \end{pmatrix}$. Next, take a second symmetrizer $\mathcal{L}_2$, add one input to the set of fixed vertices and set it to $\begin{pmatrix}- 8 \eta \\ 2 \end{pmatrix}$. Finally, take a third symmetrizer $\mathcal{L}_3$ and combine it with $\mathcal{L}_1$ and $\mathcal{L}_2$, using a combination mapping $\beta$ which sends the outputs of $\mathcal{L}_1$, $\mathcal{L}_2$ to the inputs of $\mathcal{L}_3$.

Thus, by combining three symmetrizers we get a new version of the symmetrizer $\mathcal{L}$, which is functional for the same function, but such that $\mathrm{Reg}_{\mathbb{H}^2}^P(\mathcal{L})$ ($\subseteq \mathrm{Conf}_{\mathbb{H}^2}^P(\mathcal{L})$) contains:
\[ \left\{ \psi \in \hat{\mathcal{I}_0}^P ~ \middle\arrowvert ~ \lvert \psi(b) \rvert < 7 \eta, \lvert \psi(c) \rvert < 7 \eta \right\}. \]

\subsection{The adder}
Take a symmetrizer $\mathcal{L}$. Fix the vertex $b$ to $0$, let $P = \{a\}$ and $Q = \{c\}$. We obtain a functional linkage for $x \longmapsto 2x$, for which $\mathrm{Reg}_{\mathbb{H}^2}^P(\mathcal{L})$ contains:
\[ \left\{ \psi \in \hat{\mathcal{I}_0}^P ~ \middle\arrowvert ~ \lvert \psi(a) \rvert < 3 \eta \right\}. \]
Combining this linkage with the symmetrizer, we get a functional linkage for \[ (x_1, x_2) \longmapsto 2 \frac{x_1 + x_2}{2} = x_1 + x_2.\] For this linkage, $\mathrm{Reg}_{\mathbb{H}^2}^P(\mathcal{L})$ contains:
\[ \left\{ \psi \in \hat{\mathcal{I}_0}^P ~ \middle\arrowvert ~ \lvert \psi(a) \rvert < 3 \eta, \lvert \psi(b) \rvert < 3 \eta \right\}. \]

This linkage is called the \emph{adder}.

\subsection{The opposite value linkage}
Take a symmetrizer $\mathcal{L}$. Fix the vertex $a$ to $0$, let $P = \{c\}$ and $Q = \{b\}$. We obtain a functional linkage for $x \longmapsto -x$, for which $\mathrm{Reg}_{\mathbb{H}^2}^P(\mathcal{L})$ contains:
\[ \left\{ \psi \in \hat{\mathcal{I}_0}^P ~ \middle\arrowvert ~ \lvert \psi(c) \rvert < 7 \eta \right\}. \]

\subsection{The rational homothety linkage}
Let $n$ be an integer. Using $n - 1$ adders, we get a functional linkage for $x \longmapsto nx$, for which $\mathrm{Reg}_{\mathbb{H}^2}^P(\mathcal{L})$ contains
\[ \left\{ \psi \in \hat{\mathcal{I}_0}^P ~ \middle\arrowvert ~ \lvert \psi(c) \rvert < \frac{3}{n} \eta \right\}. \]

Switching the input and the output, we get a functional linkage for $x \longmapsto \frac{x}{n}$, for which $\mathrm{Reg}_{\mathbb{H}^2}^P(\mathcal{L})$ contains
\[ \left\{ \psi \in \hat{\mathcal{I}_0}^P ~ \middle\arrowvert ~ \lvert \psi(c) \rvert < 3 \eta \right\}. \]

\subsection{The square function linkage}
With the input set $P = \{a\}$ and the output set $Q = \{b\}$, it is a functional linkage for the function:
\[ \begin{aligned} f: \mathrm{Conf}_{\mathbb{H}^2}^P(\mathcal{L}) (\subseteq \hat{\mathcal{I}_0}) & \to \hat{\mathcal{I}_0} \\ x & \longmapsto x^2, \end{aligned} \]
so that $\mathrm{Conf}_{\mathbb{H}^2}^P(\mathcal{L}) \subseteq \hat{\mathcal{I}_0}^P$ and $\mathrm{Reg}_{\mathbb{H}^2}^P(\mathcal{L})$ contains a neighborhood of $\begin{pmatrix}0\\2\end{pmatrix}$ in $\hat{\mathcal{I}_0}$.

To construct it, let $\mathcal{C}$ be the circle of Euclidean center $\begin{pmatrix} 0 \\ 2 \end{pmatrix}$ and radius $1$, $K$ a compact set like in Proposition~\ref{domPeauc} such that $\begin{pmatrix}\pm 0.5 \\ 2\end{pmatrix} \in \overset{\circ}{K}$ and $\begin{pmatrix}-4 \\ 2\end{pmatrix} \in \overset{\circ}{K}$, and $\mathcal{L}$ a Peaucellier linkage such that $K \subseteq \mathrm{Reg}_{\mathbb{H}^2}^P(\mathcal{L})$. Then $\mathcal{L}$ is functional for
\[ \begin{aligned} f: \mathrm{Conf}_{\mathbb{H}^2}^P(\mathcal{L}) (\subseteq \hat{\mathcal{I}_0}) & \to \hat{\mathcal{I}_0} \\ x & \longmapsto \frac{1}{x}. \end{aligned} \]

Now, as in the Minkowski case, we use the algebraic trick first described by Kapovich and Millson~\cite{mk}:
\[ \forall x \in \mathbb{R} \setminus \{-0.5, 0.5\} ~~~ x^2 = 0.25 + \frac{1}{\frac{1}{x - 0.5} - \frac{1}{x + 0.5}}, \]

Thus, the desired linkage is obtained by composition of the previous linkages.

\subsection{The multiplier}

\subsubsection{A first version} \label{multfirstver}

The multiplier is a linkage with $P = \{a, b\}$ and $Q = \{c\}$, which is a functional linkage for the function:
\[ \begin{aligned} f: \mathrm{Conf}_{\mathbb{H}^2}^P(\mathcal{L}) (\subseteq \hat{\mathcal{I}_0}^P) & \to \hat{\mathcal{I}_0} \\ (x_1, x_2) & \longmapsto x_1x_2 \end{aligned} \]
such that $\mathrm{Conf}_{\mathbb{H}^2}^P(\mathcal{L}) \subseteq \hat{\mathcal{I}_0}^P$ and $\mathrm{Reg}_{\mathbb{H}^2}^P(\mathcal{L})$ contains a neighborhood $\mathcal{U}$ of $\begin{pmatrix}0\\2\end{pmatrix}$ in $\hat{\mathcal{I}_0}^P$.

We simply construct it by combining square function linkages and adders, using the identity:
\[ \forall x_1, x_2 \in \mathbb{R} ~~~ x_1x_2 = \frac{1}{4}((x_1+x_2)^2 - (x_1-x_2)^2). \]

\subsubsection{Modifying $\mathrm{Conf}_{\mathbb{H}^2}^P(\mathcal{L})$}
We are now going to construct a multiplier such that $\mathrm{Conf}_{\mathbb{H}^2}^P(\mathcal{L}) \subseteq \hat{\mathcal{I}_0}^P$ and $\mathrm{Reg}_{\mathbb{H}^2}^P(\mathcal{L})$ contains $\left\{ \begin{pmatrix}x\\2\end{pmatrix} ~ \middle\arrowvert ~ x \in [-\eta, \eta] \right\}^P$.

Let $n$ be an integer such that $\left\{ \begin{pmatrix}x\\2\end{pmatrix} ~ \middle\arrowvert ~ x \in [-\frac{\eta}{n}, \frac{\eta}{n}] \right\}^P \subseteq \mathcal{U}$ (where $\mathcal{U}$ is defined in Section~\ref{multfirstver}). Using two rational homothety linkages and one multiplier (first version), and the formula:
\[ \forall x_1, x_2 \in \mathbb{R} ~~~ x_1x_2 = n^2 \left(\frac{x_1}{n}\right)\left(\frac{x_2}{n}\right) \]
we obtain the desired linkage.

\subsection{The polynomial linkage} \label{polylink-hyp}
Let $f: \mathbb{R}^n \to \mathbb{R}^m$ be a polynomial of degree $d \geq 1$ and coefficients in $[-K, K]$ with $1 \leq K \leq \eta$. We still identify $\mathbb{R}$ with $\hat{\mathcal{I}_0}$. Our aim is to construct a functional linkage for
$ f|_{\mathrm{Conf}_{\mathbb{H}^2}^P(\mathcal{L})}$ ,
with $\mathrm{Conf}_{\mathbb{H}^2}^P(\mathcal{L}) \subseteq \hat{\mathcal{I}_0}^P$ and $\mathrm{Reg}_{\mathbb{H}^2}^P(\mathcal{L})$ containing $\mathcal{U}_{K, d, n}:= \left[- M_{K, d, n}, M_{K, d, n}\right]^P$, where $M_{K, d, n} = \frac{1}{K (d+1)^n} \eta^{1/d}$.

It is obtained by combining adders and multipliers. The coefficients are represented by fixed vertices.

\begin{remarque}
At this stage, it would be possible to fix the outputs of the polynomial linkage to $\begin{pmatrix}0 \\ 2 \end{pmatrix}$ to prove directly Theorem~\ref{univ-hyp}. However, the proof of Theorem~\ref{kempe-hyp} is more complicated, since we need the input vertices to move outside the line $\hat{\mathcal{I}_0}$, as it will be explained in the next section.
\end{remarque}

\section{End of the proof of Theorem~\refkempehyp}

Let $A$ be a compact semi-algebraic subset of $(\mathbb{H}^2)^n$. First, we assume that $A$ is a compact \emph{algebraic} subset of $(\mathbb{H}^2)^n$.

We want to construct a linkage with $P = \{d_2, d_4, \dots, d_{2n}\}$ such that $\mathrm{Conf}_{\mathbb{H}^2}^P(\mathcal{L}) = A$. The idea is to identify each point $\alpha$ of the Poincaré half-plane with three coordinates $X_\alpha^1, X_\alpha^2, X_\alpha^3$, defined by:
\[ \forall i \in \{1, 2, 3\} ~~~ X_\alpha^i = \delta\left(\begin{pmatrix}i \\ 2 \end{pmatrix}, \alpha \right). \]

Since $\begin{pmatrix}1 \\ 2 \end{pmatrix}, \begin{pmatrix}2 \\ 2 \end{pmatrix}, \begin{pmatrix}3 \\ 2 \end{pmatrix}$ are not aligned, these three coordinates characterize the point $\alpha$.

Let $f: (\mathbb{R}^2)^n = \mathbb{R}^{2n} \to \mathbb{R}^m$ be a polynomial function (of degree $d$) with coefficients in $[-1, 1]$ such that $A = f^{-1}(0)$.

We may assume that $A$ is contained in the set \[ \mathcal{V} := \left(\left(- M_{1, d, 2n}, M_{1, d, 2n}\right) \times \left(3, M_{1, d, 2n}\right)\right)^n \] (see Section~\ref{polylink-hyp} for the definition of $M_{1, d, 2n}$). If it is not, choose an isometry $\Phi$ of $\mathbb{H}^2$ such that $\Phi(A)$ is contained in this set (for a large enough $\eta$), construct the desired linkage, and then replace $\phi_0$ by $\Phi^{-1} \circ \phi_0$.

If necessary, increase $\eta$ (but do not change the definition of $\mathcal V$ by doing so) so that
\[ \small M_{100, 2, 2} \geq \max \left\{\lvert x \rvert ~ \middle\arrowvert ~ x \in \mathbb R, i \in \{1, 2, 3\}, \delta\left(\begin{pmatrix} x \\ 2 \end{pmatrix}, \begin{pmatrix}i\\2\end{pmatrix}\right) \leq \max_{\substack{(\alpha_1, \dots, \alpha_n) \in \mathcal V \\ k \in \{1, \dots, n\}}} \delta\left(\alpha_k, \begin{pmatrix}i\\2\end{pmatrix}\right)\right\}. \normalsize \]

We are now ready to construct our linkage. Start with a linkage $\mathcal{L}$ with the input set $P = \{d_2, d_4, \dots, d_{2n}\}$ and no edge. Add other vertices
\[ a_1, a_2, \dots, a_{2n}, b_2^1, b_2^2, b_2^3, b_4^1, b_4^2, b_4^3, \dots, b_{2n}^1, b_{2n}^2, b_{2n}^3, c_2^1, c_2^2, c_2^3, c_4^1, c_4^2, c_4^3, \dots, c_{2n}^1, c_{2n}^2, c_{2n}^3 \] which are restricted to move on $\hat{\mathcal{I}_0}$ using line linkages.

Combine the linkage with equidistance linkages (with parameters $k_1 = 1$ and $k_2 = M_{1, d, 2n} + 4$) so that for all $\phi \in \mathrm{Conf}_{\mathbb{H}^2}(\mathcal{L})$, all $k \in \{1, \dots, n\}$ and all $i \in \{1, 2, 3\}$: \[ X_{\phi(b_{2k}^i)}^i = X_{\phi(d_{2k})}^i. \]
Then, use polynomial linkages so that for all $\phi \in \mathrm{Conf}_{\mathbb{H}^2}(\mathcal{L})$, all $k \in \{1, \dots, n\}$ and all $i \in \{1, 2, 3\}$:
\[ x_{\phi(c_{2k}^i)} = x_{\phi(a_{2k})} \cdot (x_{\phi(b_{2k}^i)} - i)^2 \]
and
\[ x_{\phi(c_{2k}^i)} = 2 \cdot ((x_{\phi(a_{2k - 1})} - i)^2 + (x_{\phi(a_{2k})} - 2)^2). \]

Thus we have for all $i$ and $k$:
\[ x_{\phi(a_{2k})} \cdot (x_{\phi(b_{2k}^i)} - i)^2 = 2 \cdot ((x_{\phi(a_{2k - 1})} - i)^2 + (x_{\phi(a_{2k})} - 2)^2) \]
\[ \arcosh \left(1 + \frac{(x_{\phi(b_{2k}^i)} - i)^2}{2 \cdot 2}\right) = \arcosh \left(1 + \frac{(x_{\phi(a_{2k - 1})} - i)^2 + (x_{\phi(a_{2k})} - 2)^2}{2 \cdot x_{\phi(a_{2k})}}\right) \]
\[ \delta\left(\phi(b_{2k}^i), \begin{pmatrix}i\\2\end{pmatrix}\right) = \delta\left( \begin{pmatrix}x_{\phi(a_{2k-1})} \\ x_{\phi(a_{2k})}\end{pmatrix}, \begin{pmatrix}i\\2\end{pmatrix}\right) \]
\[ X_{\phi(b_{2k}^i)}^i = X_{\tiny \begin{pmatrix}x_{\phi(a_{2k-1})}\\ x_{\phi(a_{2k})}\end{pmatrix}}^i \]
\[ X_{\phi(d_{2k})}^i = X_{\tiny \begin{pmatrix}x_{\phi(a_{2k-1})}\\ x_{\phi(a_{2k})}\end{pmatrix}}^i. \]
Thus:
\[ \phi(d_{2k}) = \begin{pmatrix}x_{\phi(a_{2k-1})}\\ x_{\phi(a_{2k})}\end{pmatrix}. \]

Add vertices $e_1, \dots, e_m$ and use a polynomial linkage so that for all $\phi \in \mathrm{Conf}_{\mathbb{H}^2}(\mathcal{L})$:
\[ f(x_{\phi(a_1)}, \dots, x_{\phi(a_{2n})}) = (x_{\phi(e_1)}, \dots, x_{\phi(e_m)}). \]

Now, notice that $p|_{\pi^{-1}(\mathcal V)}$ is a smooth finite covering onto $\mathcal V$, which is necessarily trivial since $\mathcal V$ is simply connected.

To finish the construction, fix the vertices $e_1, \dots, e_m$ to the point $\begin{pmatrix}0\\2\end{pmatrix}$.

Thus, \[ \mathrm{Reg}_{\mathbb{H}^2}^{\{a_1, \dots, a_{2n}\}}(\mathcal{L}) = \mathrm{Conf}_{\mathbb{H}^2}^{\{a_1, \dots, a_{2n}\}}(\mathcal{L}) = A \subseteq \mathbb{R}^{2n} = \hat{\mathcal{I}_0}, \] and finally:
\[ \mathrm{Reg}_{\mathbb{H}^2}^P(\mathcal{L}) = \mathrm{Conf}_{\mathbb{H}^2}^P(\mathcal{L}) = A \subseteq (\mathbb{H}^2)^n \subseteq (\mathbb{R}^2)^n. \]

Moreover, the restriction map $\mathrm{Conf}_{\mathbb{H}^2}(\mathcal{L}) \to \mathrm{Conf}_{\mathbb{H}^2}^P(\mathcal{L})$ is a smooth finite covering, which is trivial as the restriction of a trivial covering.

If $A$ is only a compact semi-algebraic subset of $(\mathbb{H}^2)^n$, we know from Proposition~\ref{projAlgHyp} that $A$ is the projection onto the first coordinates of a compact algebraic set $B$. Apply the above construction to $B$ and remove some vertices from the input set to obtain $\mathrm{Conf}_{\mathbb{H}^2}^P(\mathcal{L}) = A$, which ends the proof of Theorem~\refkempehyp.

\chapter{Linkages in the sphere} \label{chapSphere}

The aim of this chapter is to prove Theorem~\ref{kempe-sphere}. In the first three sections, we focus on the two-dimensional sphere, while higher dimensions are studied in the last section.

The sphere $\mathbb S^2$ will be considered as the unit sphere of $\mathbb R^3$. Thus, a point $\alpha \in \mathbb S^2$ is denoted by three coordinates $x_\alpha, y_\alpha, z_\alpha$.

\section{Elementary linkages for geometric operations} \label{link21}

\subsection{The articulated arm linkage} \label{artArm}

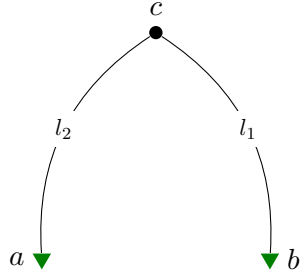
\begin{center}
\shorthandoff{:}
\begin{tikzpicture}[scale=1.5]
\node[input, label=right:$a$] (a) at (-1,0) {};
\node[input, label=left:$b$] (b) at (1,0) {};
\node[vertex, label=above:$c$] (c) at (0,2) {};
\draw (c) to[bend right] node[edge] {$l_2$} (a);
\draw (c) to[bend left] node[edge] {$l_1$} (b);
\end{tikzpicture}
\shorthandon{:}
\end{center}

$P = \{a, b\}, Q = \emptyset, F = \emptyset$. The lengths of the edges $l_1$ and $l_2$ satisfy $0 < l_1, l_2 \leq \pi$.

A standard computation gives:

\[ \mathrm{Conf}_{\mathbb S^2}^P(\mathcal{L}) = \left\{ \psi \in (\mathbb{S}^2)^P ~ \middle\arrowvert ~ \lvert l_1 - l_2 \rvert \leq \delta(\psi(a), \psi(b)) \leq \min (l_1 + l_2, 2\pi - (l_1 + l_2)) \right\}. \]

\subsection{The great circle linkage} \label{gcl}
One vertex $a$ is linked to $k$ other vertices $v_1, \dots, v_k$, by edges of length $\pi/2$.

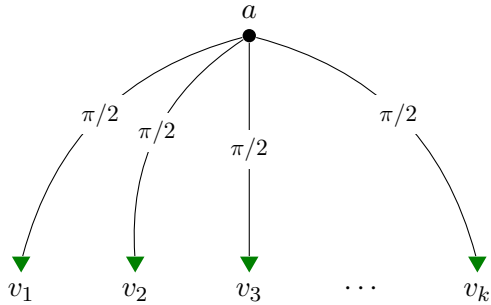
\begin{center}
\begin{tikzpicture}[scale=1.5]
\node[input, label=above:$v_k$] (v5) at (2,0) {};
\node[fill=white, label=below:$\cdots$] (v4) at (1,0) {};
\node[input, label=above:$v_3$] (v3) at (0,0) {};
\node[input, label=above:$v_2$] (v2) at (-1,0) {};
\node[input, label=above:$v_1$] (v1) at (-2,0) {};
\node[vertex, label=above:$a$] (a) at (0,2) {};
\draw (a) to[bend right] node[edge] {$\pi/2$} (v1);
\draw (a) to[bend right] node[edge] {$\pi/2$} (v2);
\draw (a) to node[edge] {$\pi/2$} (v3);
\draw (a) to[bend left] node[edge] {$\pi/2$} (v5);
\end{tikzpicture}
\end{center}

$P = \{v_1, \dots, v_k\}, Q = \emptyset, F = \emptyset$.

This linkage forces $v_1, \dots, v_k$ to be on the same great circle of the sphere:
\[ \mathrm{Conf}_{\mathbb S^2}^P(\mathcal{L}) = \left\{ \psi \in (\mathbb{S}^2)^P ~ \middle\arrowvert ~ \exists f \in (\mathbb{R}^3)^* ~~ \forall i \in \{1, \dots, k\} ~~ f(\psi(v_k)) = 0 \right\}. \]

The \emph{fixed great circle linkage} is a variant in which $a \in F$. Then
\[ \mathrm{Conf}_{\mathbb S^2}^P(\mathcal{L}) = (\mathbb{S}^2 \cap a^\bot)^P. \]

\subsection{The symmetrizer}
The symmetrizer is a functional linkage for symmetry with respect to a great circle (\emph{i.e.} orthogonal symmetry with respect to a plane $\mathcal{P}$ in $\mathbb{R}^3$). It is the key to the construction of several other linkages, but it is also the most difficult to obtain.

\subsubsection{First version}

Here is a first attempt, which we shall call $\mathcal{L}_1$.

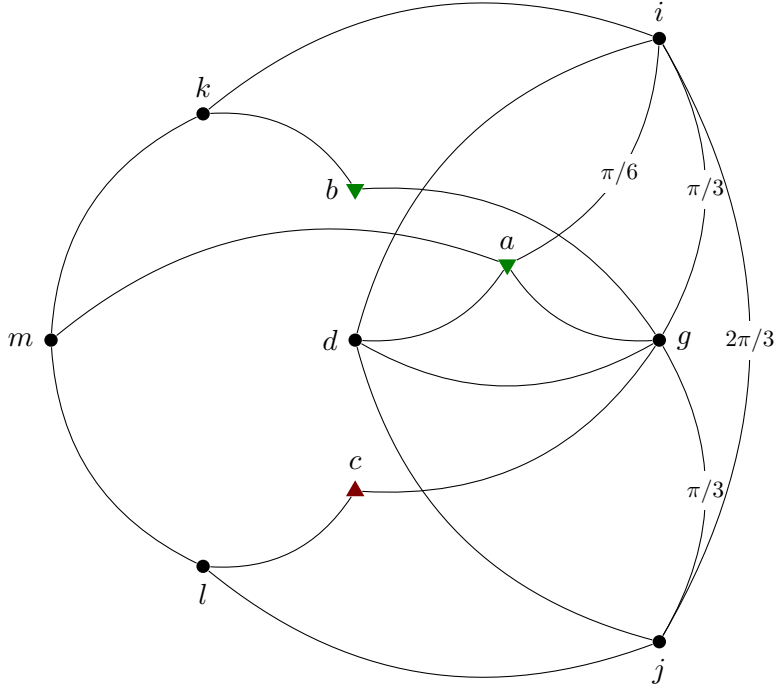
\begin{center}
\shorthandoff{:}
\begin{tikzpicture}[scale=2]
\node[input, label=below:$a$] (a) at (1,0.5) {};
\node[input, label=right:$b$] (b) at (0,1) {};
\node[output, label=above:$c$] (c) at (0,-1) {};
\node[vertex, label=left:$d$] (d) at (0,0) {};
\node[vertex, label=right:$g$] (g) at (2,0) {};
\node[vertex, label=above:$i$] (i) at (2,2) {};
\node[vertex, label=below:$j$] (j) at (2,-2) {};
\node[vertex, label=above:$k$] (k) at (-1,1.5) {};
\node[vertex, label=below:$l$] (l) at (-1,-1.5) {};
\node[vertex, label=left:$m$] (m) at (-2,0) {};

\draw (m) to[bend left] (k);
\draw (k) to[bend left] (i);
\draw (i) to[bend left] node[edge] {$\pi/3$} (g);
\draw (g) to[bend left] node[edge] {$\pi/3$} (j);
\draw (i) to[bend left] node[edge] {$2\pi/3$} (j);
\draw (a) to[bend right] node[edge] {$\pi/6$} (i);
\draw (j) to[bend left] (l);
\draw (l) to[bend left] (m);
\draw (a) to[bend right] (m);
\draw (d) to[bend right] (a);
\draw (a) to[bend right] (g);
\draw (g) to[bend left] (d);
\draw (k) to[bend left] (b);
\draw (b) to[bend left] (g);
\draw (g) to[bend left] (c);
\draw (c) to[bend left] (l);
\draw (j) to[bend left] (d);
\draw (d) to[bend left] (i);
\end{tikzpicture}
\shorthandon{:}
\end{center}
All the edges have length $\pi/2$, except when another length is indicated.

Let $P = \{a, b\}, Q = \{c\}$, and $F = \emptyset$. We want $a$ to be the unit normal vector to the (linear) plane $\mathcal{P}$ of symmetry, and $b$ to be the point to which we want to apply the symmetry. The result of the symmetry is $c$.

\begin{prop} \label{confsym1}
Fix some $\psi \in (\mathbb{S}^2)^P$. Let $\alpha \in \mathbb{S}^2$ be symmetric to $\psi(b)$ with respect to $\psi(a)^\bot$. Then:
\[ \left\{ \phi(c) ~ \middle\arrowvert ~ \phi \in p^{-1}(\psi) \right\} = \{\alpha, -\alpha\}. \] 
\end{prop}
\begin{proof}[Proof]
Assume that $\phi \in \mathrm{Conf}_{\mathbb S^2}^P(\mathcal{L}_1)$ is such that $\phi|_P = \psi$. Applying a rotation to the sphere if necessary, we may assume that $\psi(a) = (0, 0, 1)$ and $y_{\psi(b)} = 0$. Let $\phi \in \mathrm{Conf}_{\mathbb S^2}(\mathcal{L}_1)$. If $x_{\psi(b)} = 0$, we may also assume up to rotation that $y_{\phi(d)} = 0$. If $x_{\psi(b)} \neq 0$, then $\phi(g)^\bot$ contains the two distinct points $\psi(a)$ and $\psi(b)$, so $\phi(g) \in \{ \pm (0, 1, 0) \}$. Applying a symmetry with respect to $\phi(g)^\bot$ if necessary, we may assume that $\phi(g) = (0, 1, 0)$. But $\phi(d) \in \phi(g)^\bot \cap \phi(a)^\bot$, so that $\phi(d) \in \{\pm (1, 0, 0)\}$. Therefore, whether or not $x_{\psi(b)} = 0$, we may assume $\phi(d) \in \{\pm (1, 0, 0)\}$ and $\phi(g) = (0, 1, 0)$. Hence, $\phi(i) = (0, 1/2, \sqrt{3}/2)$ and $\phi(j) = (0, 1/2, -\sqrt{3}/2)$.

Since $\phi(k)$ is on the line $\phi(i)^\bot \cap \phi(b)^\bot$, it has two possible (opposite) values.

$\phi(k) \not\in \mathbb{R}\phi(a)$ because $\phi(k) \in \phi(i)^\bot$ and $\phi(a) \not\in \phi(i)^\bot$.

Since $m$ is on the line $\phi(k)^\bot \cap \phi(a)^\bot$, it has two possible opposite values.

$\phi(j) \not\in \mathbb{R}\phi(m)$ because $\phi(m) \in \phi(a)^\bot$ and $\phi(j) \not\in \phi(a)^\bot$.

Since $l$ is on the line $\phi(j)^\bot \cap \phi(m)^\bot$, $\phi(l)$ has two possible opposite values.

$\phi(l) \not\in \mathbb{R}\phi(g)$ because $\phi(l)\in\phi(j)^\bot$ and $\phi(g)\not\in\phi(j)^\bot$.

Since $c$ is on the line $\phi(l)^\bot \cap \phi(g)^\bot$, $\phi(c)$ has two possible opposite values.

Note that the construction of $\phi$ described above really provides a realization of the linkage, which proves that $\mathrm{Conf}_{\mathbb S^2}^P(\mathcal{L}_1)$ is the whole $(\mathbb{S}^2)^P$.

To see that one of the possible values of $\phi(c)$ is symmetric to $\phi(b)$ with respect to $\phi(a)^\bot$, use the symmetries of the abstract linkage: take $\phi \in \mathrm{Conf}_{\mathbb S^2}^{\{a, b, d, g, i, k, m\}}(\mathcal L_1)$, and extend $\phi$ to $V$ (the set of all vertices) by letting $\phi(c)$ be symmetric to $\phi(b)$ with respect to $\phi(a)^\bot$, $\phi(l)$ symmetric to $\phi(k)$ with respect to $\phi(a)^\bot$, and $\phi(j)$ symmetric to $\phi(i)$ with respect to $\phi(a)^\bot$. Then it is clear that $\phi \in \mathrm{Conf}_{\mathbb S^2}(\mathcal{L}_1)$.
\end{proof}

In order to avoid the configurations in which $\phi(c)$ is not symmetric to $\phi(b)$ with respect to $\phi(a)^\bot$, we introduce a second version of the symmetrizer.

\subsubsection{Second version}

Here is a different version of the symmetrizer, $\mathcal{L}_2$:

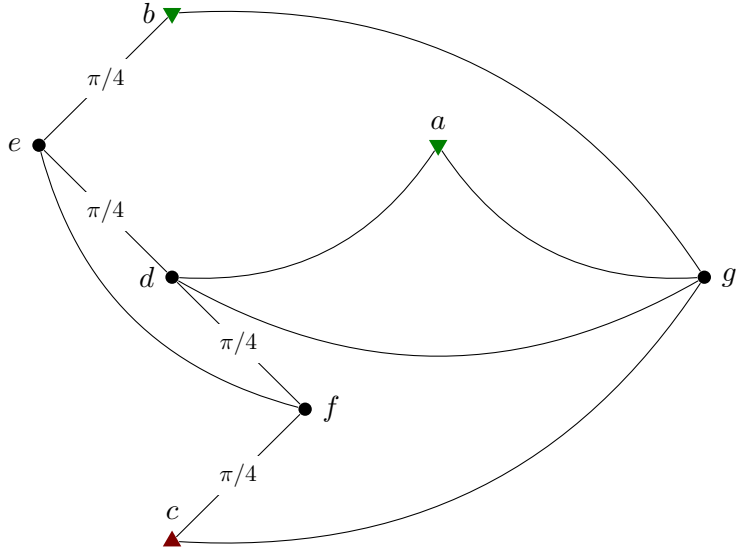
\begin{center}
\shorthandoff{:}
\begin{tikzpicture}[scale=3.5]
\node[input, label=below:$a$] (a) at (1,0.5) {};
\node[input, label=right:$b$] (b) at (0,1) {};
\node[output, label=above:$c$] (c) at (0,-1) {};
\node[vertex, label=left:$d$] (d) at (0,0) {};
\node[vertex, label=right:$g$] (g) at (2,0) {};
\node[vertex, label=left:$e$] (e) at (-0.5,0.5) {};
\node[vertex, label=right:$f$] (f) at (0.5,-0.5) {};

\draw (d) to[bend right] (a);
\draw (a) to[bend right] (g);
\draw (g) to[bend left] (d);
\draw (b) to[bend left] (g);
\draw (g) to[bend left] (c);
\draw (f) to[bend left] (e);
\draw (b) to node[edge] {$\pi/4$} (e) to node[edge] {$\pi/4$} (d) to node[edge] {$\pi/4$} (f) to node[edge] {$\pi/4$} (c);

\end{tikzpicture}
\shorthandon{:}
\end{center}

All the edges have length $\pi/2$, except when another length is indicated; $P = \{a, b\}, Q = \{c\},$ and $F = \emptyset$. As before, we want $c$ to be symmetric to $b$ with respect to $a^\bot$.

\begin{prop} \label{confsym2}
\begin{enumerate}
\item For all $\psi \in (\mathbb{S}^2)^P$, there exists $\phi \in p^{-1}(\psi)$ such that $\phi(c)$ is symmetric to $\phi(b)$ with respect to $\phi(a)^\bot$.
\item There does not exist $\phi \in \mathrm{Conf}_{\mathbb S^2}(\mathcal{L}_2)$ such that $-\phi(c)$ is symmetric to $\phi(b)$ with respect to $\phi(a)^\bot$.
\end{enumerate}
\end{prop}
\begin{proof}[Proof]
Let us prove the first assertion. Let $\psi \in (\mathbb{S}^2)^P$, choose $\phi(g)$ anywhere in $\psi(a)^\bot \cap \psi(b)^\bot$, and a point $\alpha \in \phi(g)^\bot \cap \phi(a)^\bot$. If $\delta(\alpha, \phi(b)) \leq \pi/2$, let $\phi(d) = \alpha$, else let $\phi(d) = -\alpha$. In any case we have $\delta(\phi(d), \phi(b)) \leq \pi/2$ so we can choose $\phi(e)$ on the intersection of the circles $\mathcal{C}(\phi(b), \pi/4)$ and $\mathcal{C}(\phi(d), \pi/4)$. Let $\phi(c)$ be symmetric to $\phi(b)$ with respect to $\phi(a)^\bot$, and let $\phi(f)$ be the image of $\phi(e)$ by a half turn of axis $\mathbb R \phi(d)$. Then, $\phi$ is a realization of the linkage.

We now prove the second assertion. Let $\phi \in \mathrm{Conf}_{\mathbb S^2}(\mathcal{L}_2)$. If $\delta(\phi(c), \phi(d)) < \pi/2$ then $\delta(-\phi(c), \phi(d)) > \pi/2$ whereas $\delta(\phi(b), \phi(d)) \leq \pi/2$, so $-\phi(c)$ is not symmetric to $\phi(b)$ with respect to $\phi(a)^\bot$. If $\delta(\phi(c), \phi(d)) = \pi/2$ then $\phi(c) = \pm \phi(a)$, which means that $\phi(c)$, $\phi(f)$, $\phi(d)$, $\phi(e)$ are on the same geodesic and $\delta(\phi(c), \phi(e)) = 3\pi/4$. Therefore $\phi(c) \neq \phi(b)$, so $-\phi(c)$ is not symmetric to $\phi(b)$ with respect to $\phi(a)^\bot$.
\end{proof}

$\mathcal{L}_2$ is not a functional linkage for symmetry. There is a possible degenerate configuration which seems difficult to avoid: for any position of the inputs $\psi \in \mathrm{Conf}_{\mathbb S^2}^P(\mathcal{L}_2)$, there is a $\phi \in \mathrm{Conf}_{\mathbb S^2}(\mathcal{L}_2)$ such that $\phi(b) = \phi(d)$. This problem is very related to the problem of the degenerate configurations of the parallelogram, which Kempe did not see when he wrote his original proof. The solution to this problem in the plane is the rigidification of the parallelogram, but the usual rigidification does not work in the sphere.

\subsubsection{Gluing the two versions} \label{gluesym}

We now have two linkages, $\mathcal{L}_1$ and $\mathcal{L}_2$, which are almost functional linkages for symmetry, and have different degenerate configurations.

We glue them together: let $W_1 = \{a_1, b_1, c_1\}$ and $\beta(a_1) = a_2$, $\beta(b_1) = b_2$, $\beta(c_1) = c_2$, and $\mathcal{L} = \mathcal{L}_1 \cup_\beta \mathcal{L}_2$.

We rename some vertices for future reference: $a := a_2$, $b := b_2$, $c := c_2$, $d := d_2$, $g := g_2$.

\begin{prop}
\begin{enumerate}
\item $\mathrm{Conf}_{\mathbb S^2}^P(\mathcal{L}) = (\mathbb{S}^2)^P$.
\item $\mathcal{L}$ is a functional linkage for symmetry: for all $\phi \in \mathrm{Conf}_{\mathbb S^2}(\mathcal{L})$, $\phi(c)$ is symmetric to $\phi(b)$ with respect to $\phi(a)^\bot$.
\end{enumerate}
\end{prop}
\begin{proof}[Proof]
This is an immediate consequence of propositions \ref{confsym1} and \ref{confsym2}.
\end{proof}

\subsection{The parallelizer}

The parallelizer has three inputs $b$, $c$, $h$, such that
\[ \mathrm{Conf}_{\mathbb S^2}^P(\mathcal{L}) = \left\{ \psi \in (\mathbb{S}^2)^P ~ \middle\arrowvert ~ \delta(\psi(h),\psi(b)) = \delta(\psi(h), \psi(c)) \right\}. \]

Notice that the equality $\delta(\psi(h),\psi(b)) = \delta(\psi(h), \psi(c))$ is equivalent to $\left(\psi(a) \middle\arrowvert \phi(h)\right) = \left(\psi(b) \middle\arrowvert \psi(h)\right)$, where $\left(\cdot \middle\arrowvert \cdot\right)$ denotes the scalar product in $\mathbb R^3$. Therefore, for any linear form $f: \mathbb R^3 \to \mathbb R$, there exists $\alpha \in \mathbb S^2$ such that any realization $\phi$ of the parallelizer with the vertex $h$ fixed at $\alpha$ satisfies:
\[ f(\phi(a)) = f(\phi(b)). \]

To construct the parallelizer, we use the following characterization: $\delta(\psi(h),\psi(b)) = \delta(\psi(h), \psi(c))$ if and only if there exists a linear plane $\mathcal{P}$ containing $\psi(h)$ such that $\psi(b)$ is the reflection of $\psi(c)$ with respect to $\mathcal{P}$.

Start with a symmetrizer $\mathcal{L}_1$ and consider the following linkage $\mathcal{L}_2$:

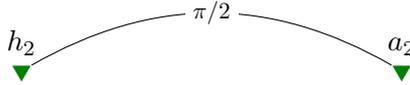
\begin{center}
\shorthandoff{:}
\begin{tikzpicture}[scale=2.5]
\node[input, label=below:$a_2$] (a) at (1,0) {};
\node[input, label=below:$h_2$] (h) at (-1,0) {};

\draw (h) to[bend left] node[edge] {$\pi/2$} (a);

\end{tikzpicture}
\shorthandon{:}
\end{center}

$P_2 = \{a_2, h_2\}, Q_2 = \emptyset, F_2 = \emptyset$.

Then let $W_1 = \{a_1\}$, $\beta(a_1) = a_2$, and $\mathcal{L} = \mathcal{L}_1 \cup_\beta \mathcal{L}_2$. Change the input set of $\mathcal{L}$ so that $P = \{b_1, c_1, h_2\}$, and rename the inputs: $b = b_1$, $c = c_1$, $h = h_2$, which ends the construction.

\section{Elementary linkages for algebraic operations} \label{link22}

\subsection{The homothety linkage}
Let $\lambda \in \left(0, 1\right)$. Our aim is to construct a linkage which takes one point $\phi(a) = (x_{\phi(a)}, y_{\phi(a)}, 0)$ and, when possible, forces another point $\phi(b) = (x_{\phi(b)}, y_{\phi(b)}, 0)$ to satisfy
\[ y_{\phi(b)} = \lambda y_{\phi(a)}. \]

Start with the following linkage $\mathcal{L}_1$:

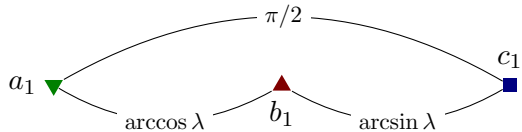
\begin{center}
\shorthandoff{:}
\begin{tikzpicture}[scale=3]
\node[input, label=right:$a_1$] (a) at (-2,0) {};
\node[output, label=below:$b_1$] (b) at (-1,0) {};
\node[fixed, label=above:$c_1$] (c) at (0,0) {};
\draw (a) to[bend right] node[edge] {$\arccos \lambda$} (b);
\draw (b) to[bend right] node[edge] {$\arcsin \lambda$} (c);
\draw (c) to[bend right] node[edge] {$\pi/2$} (a);
\end{tikzpicture}
\shorthandon{:}
\end{center}

$P_1 = \{a_1\}$, $Q_1 = \{b_1\}$, $F_1 = \{c_1\}$, $\phi_{01}(c_1) = (0, 0, 1)$.

This linkage is functional for a homothety from the equator to the (smaller) circle of latitude $\arccos \lambda$: thus for all $\phi \in \mathrm{Conf}_{\mathbb S^2}(\mathcal{L}_1)$, $y_{\phi(b_1)} = \lambda y_{\phi(a_1)}$, and
\[ \mathrm{Conf}_{\mathbb S^2}^{P_1}(\mathcal{L}_1) = \mathbb{S}^2 \cap (Oxy) \]

However, $z_{\phi(b_1)} \neq 0$ so we need to improve the construction. Let $\mathcal{L}_2$ be a parallelizer for the linear form $f(x, y, z) = y$ and $\mathcal{L}_3$ a parallelizer for $g(x, y, z) = z$. Let $W_1 = \{b_1\}$, $\beta_1(b_1) = b_2$, and $\mathcal{L}_4 = \mathcal{L}_1 \cup_{\beta_1} \mathcal{L}_2$. Let $W_4 = \{a_1, c_2\}$, $\beta_4(a_1) = b_3$, $\beta_4(c_2) = c_3$, and $\mathcal{L}_5 = \mathcal{L}_4 \cup_{\beta_4} \mathcal{L}_3$.

We get:
\[ \mathrm{Conf}_{\mathbb S^2}^{P_1}(\mathcal{L}_1) = \mathbb{S}^2 \cap (Oxy), \]
and for all $\phi \in \mathrm{Conf}_{\mathbb S^2}(\mathcal{L})$, $y_{\phi(c_3)} = \lambda y_{\phi(a_1)}$, $z_{\phi(c_3)} = 0$.

Finally, rename the two vertices: $a = a_1$ and $b = c_3$.

\subsection{The adder}
Our aim is to construct a linkage which takes two points $\phi(a) = (x_{\phi(a)}, y_{\phi(a)}, 0)$, $\phi(b) = (x_{\phi(b)}, y_{\phi(b)}, 0)$ and, when possible, forces a third point $\phi(c) = (x_{\phi(c)}, y_{\phi(c)}, 0)$ to satisfy
\[ y_{\phi(c)} = y_{\phi(a)} + y_{\phi(b)}. \]

There are several steps to construct such a linkage $\mathcal{L}$.
\begin{enumerate}
\item Restrict the two points $a$, $b$ to move in the $Oxy$ plane by using the fixed great circle linkage for $k = 2$.
\item Using a symmetrizer, extend this linkage to a new one having a vertex $d$ such that $d$ is symmetric to $b$ with respect to the plane $\{y - z = 0\}$. Then for all $\phi \in \mathrm{Conf}_{\mathbb S^2}(\mathcal{L})$ we have $y_{\phi(b)} = z_{\phi(d)}$.
\item With two parallelizers, extend this linkage to a new one having a vertex $e$ such that for all $\phi \in \mathrm{Conf}_{\mathbb S^2}(\mathcal{L})$:
   \begin{enumerate}
   \item $y_{\phi(e)} = y_{\phi(a)}$;
   \item $z_{\phi(e)} = z_{\phi(d)}$.
   \end{enumerate}
\item With two parallelizers, extend this linkage to a new one with a vertex $c$ such that for all $\phi \in \mathrm{Conf}_{\mathbb S^2}(\mathcal{L})$:
   \begin{enumerate}
   \item $y_{\phi(c)} + z_{\phi(c)} = y_{\phi(e)} + z_{\phi(e)}$;
   \item $z_{\phi(c)} = 0$.
   \end{enumerate}
\end{enumerate}

Then for all $\phi \in \mathrm{Conf}_{\mathbb S^2}(\mathcal{L})$ we have $y_{\phi(c)} = y_{\phi(a)} + y_{\phi(b)}$, as desired. Let $P = \{a, b\}$ and $Q = \{c\}$.

We have
\[ \mathrm{Conf}_{\mathbb S^2}^P(\mathcal{L}) = \left\{\psi \in (\mathbb{S}^2 \cap (Oxy))^P ~ \middle\arrowvert ~ y_{\phi(a)} + y_{\phi(b)} \in \left[-1, 1\right] \right\} . \]

\subsection{The multiplier}
Identify the plane $(Oxy)$ with the complex plane: to a point $(x, y, 0) \in \mathbb{R}^3$, associate the complex number $\zeta_{(x, y, 0)} = x + iy$. We want to construct a functional linkage which takes two complex numbers and returns their product.

Since we work in the sphere, we only need to multiply complex numbers $\alpha_1$ and $\alpha_2$ in the unit circle. This corresponds to adding the arguments. We split this operation into two steps:
\begin{enumerate}
\item Compute $\frac{\arg(\alpha_1) + \arg(\alpha_2)}{2} \text{ mod } \pi$;
\item Double the argument.
\end{enumerate}

The following linkage $\mathcal{L}_1$ will be the basis of the construction:

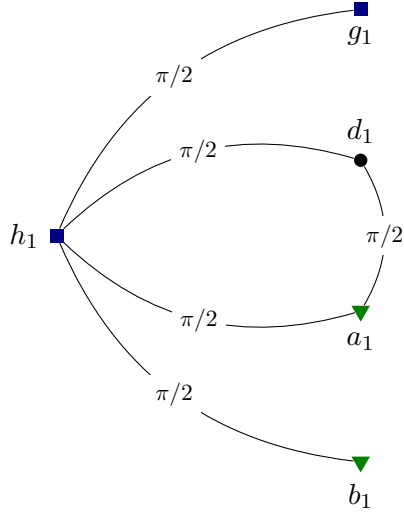
\begin{center}
\shorthandoff{:}
\begin{tikzpicture}[scale=2]
\node[fixed, label=left:$h_1$] (h) at (0,0) {};
\node[fixed, label=below:$g_1$] (g) at (2,1.5) {};
\node[vertex, label=above:$d_1$] (d) at (2,0.5) {};
\node[input, label=above:$a_1$] (a) at (2,-0.5) {};
\node[input, label=above:$b_1$] (b) at (2,-1.5) {};
\draw (h) to[bend left] node[edge] {$\pi/2$} (g);
\draw (h) to[bend left] node[edge] {$\pi/2$} (d);
\draw (h) to[bend right] node[edge] {$\pi/2$} (a);
\draw (h) to[bend right] node[edge] {$\pi/2$} (b);
\draw (a) to[bend right] node[edge] {$\pi/2$} (d);
\end{tikzpicture}
\shorthandon{:}
\end{center}

$P_1 = \{a_1, b_1\}, Q = \emptyset$, $F_1 = \{g_1, h_1\}$, $\phi_{01}(g_1) = (1, 0, 0)$, $\phi_{01}(h_1) = (0, 0, 1)$.

We have $\mathrm{Conf}_{\mathbb S^2}^{P_1}(\mathcal{L}_1) = (\mathbb{S}^2)^{P_1}$.

Take two copies $\mathcal{L}_2$ and $\mathcal{L}_3$ of the symmetrizer. Let $W_1 = \{a_1, b_1\}$, $\beta_1(a_1) = a_2$, $\beta_1(b_1) = b_2$, and $\mathcal L_5 = \mathcal L_1 \cup_{\beta_1} \mathcal L_2$. Then let $W_5 = \{a_2, g_1\}$, $\beta_5(a_2) = a_3$, $\beta_5(g_1) = b_3$, $\mathcal{L}_4 = \mathcal L_5 \cup_{\beta_5} \mathcal L_3$. We write $a_4 := a_3$, $b_4 := b_2$, $c_4 := c_2$, $d_4 := d_1$, $f_4 := c_3$, $g_4 := b_3$.

Now for all $\phi \in \mathrm{Conf}_{\mathbb S^2}(\mathcal L_4)$, $\phi(c_4)$ is symmetric to $\phi(b_4)$ with respect to $\phi(a_4)^\bot$, and $\phi(f_4)$ is symmetric to $\phi(g_4)$ with respect to $\phi(a_4)^\bot$. In other words, \[ \arg \zeta_{\phi(d_4)} = \frac{\arg(\zeta_{\phi(b_4)}) + \arg(\zeta_{\phi(c_4)})}{2} \text{ mod } \pi \] and \[ \arg \zeta_{\phi(f_4)} = 2 \arg \zeta_{\phi(d_4)} \text{ mod } 2\pi. \]

Let $\mathcal{L}_6$ be the linkage $\mathcal{L}_4$ with the input set $P_6 = \{ b_4, c_4 \}$. We have $\mathrm{Conf}_{\mathbb S^2}^{P_6}(\mathcal{L}_6) = (\mathbb{S}^2)^{P_6}$.

Taking $Q_6 = \{f_4\}$, $\mathcal{L}_6$ becomes a functional linkage for multiplication.

\subsection{The polynomial linkage}
Let $f : \mathbb{R}^n \to \mathbb{R}^m$ be a polynomial. Our aim is now to construct a linkage with $n$ inputs $a_1, \dots, a_n$, such that:

\[ \mathrm{Conf}_{\mathbb S^2}^P(\mathcal{L}) = \left\{ \psi \in (\mathbb{S}^2 \cap (Oxy))^P ~ \middle\arrowvert ~ f(y_{\psi(a_1)}, \dots, y_{\psi(a_n)}) = 0 \right\}. \]

Let us assume first that $m = 1$.

Recall that we write $\zeta_{\psi(a_k)} = x_{\psi(a_k)} + iy_{\psi(a_k)}$. We can also write:
\[ y_{\psi(a_k)} = \frac{\zeta_{\psi(a_k)} - \overline{\zeta_{\psi(a_k)}}}{2i} \].

Thus, there exists a polynomial $g : \mathbb{C}^{2n} \to \mathbb{C}$ such that for all $\psi \in (\mathbb{S}^2 \cap (Oxy))^P$:
\[ g(\zeta_{\psi(a_1)}, \overline{\zeta_{\psi(a_1)}}, \dots, \zeta_{\psi(a_n)}, \overline{\zeta_{\psi(a_n)}}) = f(y_{\psi(a_1)}, \dots, y_{\psi(a_n)}). \]

We write
\[ g = \sum_{j=1}^{r} g_j \]
where each $g_j$ is a monomial:
\small \[ g_j(\zeta_{\psi(a_1)}, \overline{\zeta_{\psi(a_1)}}, \dots, \zeta_{\psi(a_n)}, \overline{\zeta_{\psi(a_n)}}) = \lambda_j \epsilon_j (\zeta_{\psi(a_1)})^{\gamma_{j, 1}} (\overline{\zeta_{\psi(a_1)}})^{\gamma_{j, 2}} \cdots (\zeta_{\psi(a_n)})^{\gamma_{j, 2n-1}} (\overline{\zeta_{\psi(a_n)}})^{\gamma_{j, 2n}} \]
\normalsize with $\epsilon_j \in \{1, i, -1, -i\}$ and $\lambda_j$ a positive real number.

Observe that without changing the locus \[ \left\{ \psi \in (\mathbb{S}^2 \cap (Oxy))^P ~ \middle\arrowvert ~ f(y_{\psi(a_1)}, \dots, y_{\psi(a_n)}) = 0 \right\}, \] one may assume $\lambda_j < \lambda_0$ for all $j$, where $\lambda_0$ is arbitrary in $(0,1)$ (if necessary, multiply $f$ by a small constant).

We are now ready to construct the linkage. Start with a fixed great circle linkage $\mathcal{L}$ which forces all the $a_k$ to move in the plane $(Oxy)$. Use symmetrizers to extend $\mathcal{L}$ to a new linkage with vertices $a_k'$ such that for all $k \in \{1, \dots, n\}$, \[ \zeta_{\phi(a_k')} = \overline{\zeta_{\phi(a_k)}}. \]

For each $j \in \{1, \dots, r\}$:

\begin{enumerate}
\item Use multipliers to extend $\mathcal{L}$ to a new linkage with a vertex $c_j$ such that for all $\phi \in \mathrm{Conf}_{\mathbb S^2}(\mathcal{L})$,
\[ \zeta_{\phi(c_j)} = (\zeta_{\psi(a_1)})^{\gamma_{j, 1}} (\overline{\zeta_{\psi(a_1)}})^{\gamma_{j, 2}} \cdots (\zeta_{\psi(a_n)})^{\gamma_{j, 2n-1}} (\overline{\zeta_{\psi(a_n)}})^{\gamma_{j, 2n}}. \]
\item Use a multiplier to extend the linkage to a new one with a vertex $d_j$ such that:
\[ \zeta_{\phi(d_j)} = i \epsilon_j \zeta_{\phi(c_j)}. \]
\item Use a homothety linkage to extend the linkage to a new one with a vertex $b_j$ such that:
\[ y_{\phi(b_j)} = \lambda_j y_{\phi(d_j)}. \]
\end{enumerate}

Thus we have for all $\phi \in \mathrm{Conf}_{\mathbb S^2}(\mathcal{L})$:
\[ y_{\phi(b_j)} = \Im (i g_j (\zeta_{\psi(a_1)}, \overline{\zeta_{\psi(a_1)}}, \dots, \zeta_{\psi(a_n)}, \overline{\zeta_{\psi(a_n)}})). \]

Then, use several adders to extend the linkage to a new one, still called $\mathcal{L}$, with a vertex $c$ such that for all $\phi \in \mathrm{Conf}_{\mathbb S^2}(\mathcal{L})$:
\[ y_{\phi(c)} = \sum_{j=1}^{r} y_{\phi(b_j)}. \]

Thus
\[ y_{\phi(c)} = \text{Im} (i g (\zeta_{\psi(a_1)}, \overline{\zeta_{\psi(a_1)}}, \dots, \zeta_{\psi(a_n)}, \overline{\zeta_{\psi(a_n)}})), \]
which means that
\[ y_{\phi(c)} = f(y_{\psi(a_1)}, \dots, y_{\psi(a_n)}). \]

Choose $\lambda_0$ so small that all the steps of the computation remain in $[-1, 1]$. Then:
\[ \mathrm{Conf}_{\mathbb S^2}^P(\mathcal{L}) = (\mathbb{S}^2 \cap (Oxy))^{P}. \]

Finally, if $m \geq 2$, just write $f = (f_1, \dots, f_m)$ and use $m$ linkages like above.

\section{End of the proof of Theorem~\refkempesphere~for $d = 2$}

First, we assume that $A$ is an \emph{algebraic} subset of $(\mathbb{S}^2)^n$. Let $f: (\mathbb{R}^3)^n = \mathbb{R}^{3n} \to \mathbb{R}^m$ be a polynomial function such that $A = f^{-1}(0)$.

\begin{enumerate}
\item Take a polynomial linkage $\mathcal{L}$ with inputs $a_1, \dots, a_{3n}$ such that:
\[ \mathrm{Conf}_{\mathbb S^2}^P(\mathcal{L}) = \left\{ \psi \in (\mathbb{S}^2 \cap (Oxy))^P ~ \middle\arrowvert ~ f(y_{\psi(a_1)}, \dots, y_{\psi(a_{3n})}) = 0 \right\}. \]
\item With several symmetrizers, extend this linkage to a new one with vertices $b_1$, $b_4$, $b_7$, $\dots$, $b_{3n-2}$ such that for all $\phi \in \mathrm{Conf}_{\mathbb S^2}(\mathcal{L})$ and for all $k \in \{1, \dots, n\}$:
\[ x_{\phi(b_{3k-2})} = y_{\phi(a_{3k-2})}. \]
\item With several symmetrizers, extend this linkage to a new one with vertices $c_3$, $c_6$, $c_9$, $\dots$, $c_{3n}$ such that for all $\phi \in \mathrm{Conf}_{\mathbb S^2}(\mathcal{L})$ and for all $k \in \{1, \dots, n\}$:
\[ z_{\phi(c_{3k})} = y_{\phi(a_{3k})}. \]
\item With several parallelizers, extend this linkage to a new one with vertices $d_3$, $d_6$, $d_9$, $\dots$, $d_{3n}$ such that for all $\phi \in \mathrm{Conf}_{\mathbb S^2}(\mathcal{L})$ and for all $k \in \{1, \dots, n\}$:
\[ x_{\phi(d_{3k})} = x_{\phi(b_{3k-2})} ; \]
\[ y_{\phi(d_{3k})} = y_{\phi(a_{3k-1})} ; \]
\[ z_{\phi(d_{3k})} = z_{\phi(c_{3k})}. \]
\end{enumerate}

Now, let $P = \{d_3, d_6, \dots, d_{3n}\}$. We have:
\[ \mathrm{Conf}_{\mathbb S^2}^P(\mathcal{L}) = f^{-1}(0) = A. \]

If $A$ is only a compact semi-algebraic subset of $(\mathbb{S}^2)^n$, we know from Proposition~\ref{projAlgSph} that $A$ is the projection onto the first coordinates of an algebraic subset $B$ of the sphere: apply the above construction to $B$ and remove some vertices from the input set to obtain $\mathrm{Conf}_{\mathbb{S}^2}^P(\mathcal{L}) = A$; thus, Theorem~\ref{kempe-sphere} is proved.

\section{Higher dimensions}

In this section, we fix a number $d \geq 2$ and consider realizations in the sphere $\mathbb{S}^d$.

\subsection{The 3-plane linkage}
This linkage forces several points to move in the same (linear) $3$-plane. It is to be compared with the ``great circle linkage" described in section~\ref{gcl}, which forces several points to move in the same (linear) $2$-plane.

There are $k$ inputs $v_1, \dots, v_k$, and $d-2$ other vertices $a_1, \dots, a_{d-2}$. For all $i, j \in \{1, \dots, d-2\}$, there is an edge $a_ia_j$ of length $\pi/2$. For all $i \in \{1, \dots, d-2\}$ and $l \in \{1, \dots, k\}$, there is an edge $a_iv_l$ of length $\pi/2$.

Here is an example with $d = 5$ and $k = 3$.

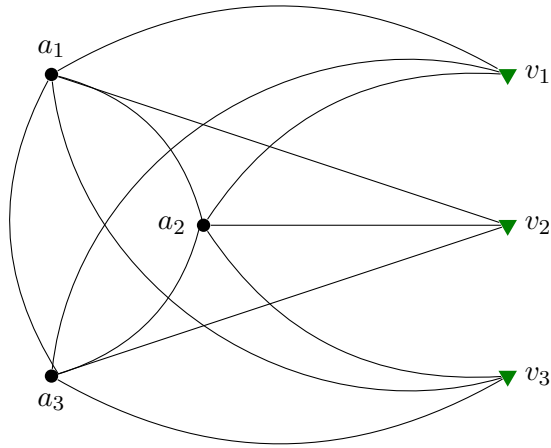
\begin{center}
\shorthandoff{:}
\begin{tikzpicture}[scale=2]
\node[vertex, label=above:$a_1$] (a1) at (-1, 1) {};
\node[vertex, label=left:$a_2$] (a2) at (0,0) {};
\node[vertex, label=below:$a_3$] (a3) at (-1,-1) {};
\node[input, label=left:$v_1$] (v1) at (2,1) {};
\node[input, label=left:$v_2$] (v2) at (2,0) {};
\node[input, label=left:$v_3$] (v3) at (2,-1) {};
\draw (a1) to[bend left] (a2) to[bend left] (a3) to[bend left] (a1);
\draw (a1) to[bend left] (v1);
\draw (a1) to (v2);
\draw (a1) to[bend right=50] (v3);
\draw (a2) to[bend left] (v1);
\draw (a2) to (v2);
\draw (a2) to[bend right] (v3);
\draw (a3) to[bend left=50] (v1);
\draw (a3) to(v2);
\draw (a3) to[bend right] (v3);
\end{tikzpicture}
\shorthandon{:}
\end{center}

\begin{prop}
We have $\mathrm{Conf}_{\mathbb S^d}^P(\mathcal{L}) = E$, where
\[ E = \left\{ \psi \in (\mathbb{S}^d)^{P} ~ \middle\arrowvert ~ \exists \mathcal{F} \text{ subspace of } \mathbb{R}^{d+1}, \dim \mathcal{F} = 3, \forall i \in \{1, \dots, k\} ~~ \psi(v_k) \in \mathcal{F} \right\}. \]
\end{prop}
\begin{proof}[Proof]
First, we prove that $\mathrm{Conf}_{\mathbb S^d}^P(\mathcal{L}) \subseteq E$.
Let $\psi \in \mathrm{Conf}_{\mathbb S^d}^P(\mathcal{L})$ and $\phi \in p^{-1}(\psi)$. Let
\[ \mathcal{F} = \bigcap_{1 \leq i \leq d-2} \phi(a_i)^\bot. \]
We know that $\{\phi(a_1), \dots, \phi(a_{d-2})\}$ is an orthonormal set, so $\dim \mathcal{F} = 3$.
Moreover, for all $l \in \{1, \dots, k\}$, $\psi(v_l) \in \mathcal{F}$.

Now, we prove that $E \subseteq \mathrm{Conf}_{\mathbb S^d}^P(\mathcal{L})$. Let $\psi \in E$. Let $\mathcal{F}$ be a subspace of $\mathbb{R}^{d+1}$ with $\dim \mathcal{F} = 3$ containing $\psi(v_l)$ for $l \in \{1, \dots, k\}$. Construct $\phi \in (\mathbb{S}^d)^P$ by letting $\{\phi(a_1), \dots, \phi(a_{d-2})\} \subseteq \mathcal{F}^\bot$ be an orthonormal set and let $\phi|_P = \psi$. Then $\phi \in \mathrm{Conf}_{\mathbb S^d}(\mathcal{L})$ so $\psi \in \mathrm{Conf}_{\mathbb S^d}^P(\mathcal{L})$.
\end{proof}

The \emph{fixed 3-plane linkage} is a variant in which $a_1, \dots, a_{d-2} \in F$ (namely, they are fixed vertices). Then there exists $\mathcal{F}$ a subspace of $\mathbb{R}^{d+1}$ with $\dim \mathcal{F} = 3$ and
\[ \mathrm{Conf}_{\mathbb S^d}^P(\mathcal{L}) = \left\{ \psi \in (\mathbb{S}^2)^{P} ~ \middle\arrowvert ~ \forall i \in \{1, \dots, k\} ~~ \psi(v_k) \in \mathcal{F} \right\}. \]

\subsection{The $d$-dimensional symmetrizer}
Like in the $2$-dimensional case, the $d$-dimensional symmetrizer has two inputs $a$ and $b$, and one output $c$. It is a functional linkage for symmetry: for all $\phi \in \mathrm{Conf}_{\mathbb S^d} (\mathcal{L})$, $\phi(c)$ is symmetric to $\phi(b)$ with respect to $\phi(a)^\bot$. The idea is that the symmetry takes place in a $3$-plane containing $\phi(a)$, $\phi(b)$ and $\phi(c)$.

Let $\mathcal{L}_1$ be a classical symmetrizer and $\mathcal{L}_2$ be a $3$-plane linkage, with $k = \mathrm{card} (V_1)$. Let $W_1 = V_1$, $\beta$ a bijection between $V_1$ and $\{v_1, \dots, v_k\} ~ (\subseteq V_2)$, and $\mathcal{L} = \mathcal{L}_1 \cup_\beta \mathcal{L}_2$. Letting $a := \beta(a_1)$, $b := \beta(b_1)$, $c := \beta(c_1)$, we obtain as desired:
\[ \mathrm{Conf}_{\mathbb S^d}(\mathcal{L}) = \left\{ \phi \in (\mathbb{S}^d)^V ~ \middle\arrowvert ~ \phi(c) \text{ is symmetric to } \phi(b) \text{ with respect to } \phi(a)^\bot \right\}. \]

\subsection{The $d$-dimensional parallelizer}
Like in the $2$-dimensional case, the $d$-dimensional parallelizer forces two points to have the same scalar product with a third one. We restrict the vertices of a classical parallelizer to move on a $3$-plane containing its three inputs.

Let $\mathcal{L}_1$ be a classical parallelizer and $\mathcal{L}_2$ be a $3$-plane linkage, with $k = \mathrm{card} (V_1)$. Let $W_1 = V_1$, $\beta$ a bijection between $V_1$ and $\{v_1, \dots, v_k\} ~ (\subseteq V_2)$, and $\mathcal{L} = \mathcal{L}_1 \cup_\beta \mathcal{L}_2$. Let $h := \beta(h_1)$, $b := \beta(b_1)$, $c := \beta(c_1)$. Then we obtain as desired:
\[ \mathrm{Conf}_{\mathbb S^d}^P(\mathcal{L}) = \left\{ \psi \in (\mathbb{S}^d)^P ~ \middle\arrowvert ~ \delta(\psi(h),\psi(b)) = \delta(\psi(h), \psi(c)) \right\}. \]

\subsection{End of the proof of Theorem~\refkempesphere~for $d \geq 2$}
Here, we prove the algebraic universality in $\mathbb S^d$. The proof is similar to the case $d = 2$. There are only two differences.
\begin{enumerate}
\item The polynomial linkage $\mathcal{L}$ is attached to a fixed $3$-plane linkage.
\item We use $d$-dimensional symmetrizers and $d$-dimensional parallelizers.
\end{enumerate}

\bigskip

\noindent \textbf{Acknowledgements.} This paper corresponds to Part I of my PhD thesis: I would like to thank my advisor Abdelghani Zeghib for his help.

\bibliographystyle{alpha}
\bibliography{linkages}

\end{document}